\documentclass[a4paper,12pt,reqno]{amsart}

\usepackage{amsmath, amsfonts, amssymb, amsthm, mathrsfs, mathtools, mathalfa, setspace}
\usepackage{caption}
\usepackage{comment}
\usepackage{yfonts}
\usepackage{enumerate}
\usepackage{enumitem}
\usepackage{stmaryrd}
\usepackage{fancyhdr}
\usepackage{bm}
\usepackage[utf8]{inputenc}
\usepackage[pdfencoding=auto,hyperfootnotes=false,backref=page]{hyperref}
\usepackage{tikz,tikz-cd}
\usepackage[hang,flushmargin]{footmisc}
\usepackage{xr}
\usepackage{extarrows}
\usepackage{todonotes}
\usepackage{faktor}
\usepackage{xfrac}
\usepackage{times}
\usepackage{color}
\usepackage{graphicx}
\usepackage{quiver}
\usetikzlibrary{babel}
\usepackage{xurl}

\usetikzlibrary{arrows}
\usetikzlibrary{positioning, arrows.meta}

\makeatletter
\def\@tocline#1#2#3#4#5#6#7{\relax
  \ifnum #1>\c@tocdepth 
  \else
    \par \addpenalty\@secpenalty\addvspace{#2}%
    \begingroup \hyphenpenalty\@M
    \@ifempty{#4}{%
      \@tempdima\csname r@tocindent\number#1\endcsname\relax
    }{%
      \@tempdima#4\relax
    }%
    \parindent\z@ \leftskip#3\relax \advance\leftskip\@tempdima\relax
    \rightskip\@pnumwidth plus4em \parfillskip-\@pnumwidth
    #5\leavevmode\hskip-\@tempdima
      \ifcase #1
       \or\or \hskip 1em \or \hskip 2em \else \hskip 3em \fi%
      #6\nobreak\relax
    \dotfill\hbox to\@pnumwidth{\@tocpagenum{#7}}\par
    \nobreak
    \endgroup
  \fi}
\makeatother

\newtheorem{theorem}{Theorem}[section]
\newtheorem{lemma}[theorem]{Lemma}

\newtheorem{corollary}[theorem]{Corollary}
\newtheorem{proposition}[theorem]{Proposition}
\newtheorem{question}[theorem]{Question}

\theoremstyle{definition}
\newtheorem{defn}[theorem]{Definition}
\newtheorem{remark}[theorem]{Remark}

\newcommand*{\sbr}[1]{\scalebox{0.8}{$(#1)$}}
\newcommand*{\db}[1]{\llbracket #1\rrbracket}

\newcommand{\mk}{\mathfrak}
\newcommand{\mc}{\mathcal}

\newcommand{\mf}{\mathbf}
\newcommand{\mb}{\mathbb}

\newcommand{\wh}{\widehat}
\newcommand{\wt}{\widetilde}

\newcommand{\ud}{\,\mathrm{d}}
\newcommand{\id}{\mathrm{id}}

\newcommand{\stab}{\mathrm{Stab}}

\newcommand{\ab}{\mathrm{Z}}

\newcommand{\tran}{\mathrm{\Theta}}

\newcommand{\lcm}{\mathrm{lcm}}

\newcommand{\poly}{\mathrm{poly}}

\newcommand{\Lip}{\mathrm{Lip}}

\newcommand{\Id}{\mathrm{Id}}

\newcommand{\q}{\mathrm{c}}
\newcommand{\ns}{\mathrm{X}}
\newcommand{\nss}{\mathrm{Y}}
\newcommand{\co}{\circ\hspace{-0.02cm}}
\newcommand{\cu}{\mathrm{C}}

\newcommand{\inv}{\mathrm{inv}}
\newcommand{\orb}{\mathrm{orb}}

\providecommand{\comm}[1]{[#1]}

\begin{document}
\singlespacing


\title[An inverse theorem for all finite abelian groups via nilmanifolds]{An inverse theorem for all finite abelian groups\\ via nilmanifolds}

\author{Pablo Candela}
\address{Instituto de Ciencias Matem\'aticas, Calle Nicol\'as Cabrera 13-15, Madrid 28049, Spain}
\email{pablo.candela@icmat.es}

\author{Diego Gonz\'alez-S\'anchez}
\address{Universit\'e Paris Cit\'e, Sorbonne Universit\'e, CNRS, IMJ-PRG, F-75013 Paris, France}
\email{gonzalezsanchez@imj-prg.fr}

\author{Bal\'azs Szegedy}
\address{HUN-REN Alfr\'ed R\'enyi Institute of Mathematics\\ 
Re\'altanoda utca 13-15\\
Budapest, Hungary, H-1053}
\email{szegedyb@gmail.com}

\vspace*{-1cm}
\begin{abstract} 
We prove a first inverse theorem for Gowers norms on all finite abelian groups that uses only nilmanifolds (rather than possibly more general nilspaces). This makes progress toward confirming the Jamneshan--Tao conjecture. 
The correlating function in our theorem is a \emph{projected nilsequence}, obtained as the fiber-wise average of a nilsequence defined on a boundedly-larger abelian group extending the original abelian group. This result is tight in the following sense: we prove also that $k$-step projected nilsequences of bounded complexity are genuine obstructions to having small Gowers $U^{k+1}$-norm. This inverse theorem relies on a new result concerning compact finite-rank (\textsc{cfr}) nilspaces, which is the main  contribution in this paper: every $k$-step \textsc{cfr} nilspace is a factor of a $k$-step nilmanifold. This new connection between the classical theory of nilmanifolds and the more recent theory of nilspaces has applications beyond arithmetic combinatorics. We illustrate this with an application in topological dynamics, by proving the following result making progress on a question of Jamneshan, Shalom and Tao: every minimal $\mb{Z}^\omega$-system of order $k$ is a factor of an inverse limit of $\mb{Z}^\omega$-polynomial orbit systems of order $k$, these being natural generalizations of nilsystems alternative to translational systems.
\end{abstract}

\maketitle


\section{Introduction}\label{sec:intro}
\noindent Nilmanifolds play a fundamental role in higher-order Fourier analysis. Important early evidence of this was given by work of Host and Kra in ergodic theory describing the characteristic factors for uniformity seminorms  using nilmanifolds \cite{HK-non-conv}, and then by the inverse theorem for Gowers norms on finite cyclic groups using nilsequences, proved by Green, Tao and Ziegler \cite{GTZ}. 

A deeper understanding of these developments soon became necessary, motivated in particular by the goal of extending the inverse theorem to other abelian groups and by further applications in ergodic theory and topological dynamics. This led to a search for axiomatic structures that could clarify and adequately generalize this role of nilmanifolds. Initiated by Host and Kra in \cite{HK-par}, this approach led to the definition of nilspaces by Antol\'in Camarena and the third named author in \cite{CamSzeg}. Nilspace theory has grown rapidly since then (see \cite{Cand:Notes1, Cand:Notes2, CGSS, CGSS-p-hom, CGSS-abramov, CGSS-doucos, CGSS-bndtor, CGSS-spec, CScouplings, CSinverse, GGY, GMV1, GMV2, GMV3,SzegHigh, SzegFin} and references therein), with applications including the following: in ergodic theory, a description of general Host--Kra factors (for any countable discrete nilpotent group action) as compact nilspaces \cite[Theorem 5.11]{CScouplings}; in topological dynamics, a characterization of factors of dynamical systems corresponding to higher-order regionally proximal relations  \cite[Theorem 7.15]{GGY}; in arithmetic combinatorics, an inverse theorem for Gowers norms on any finite (or compact) abelian group, in terms of balanced nilspace polynomials \cite[Theorem 5.2]{CSinverse}.\footnote{This theorem implies the Green--Tao--Ziegler inverse theorem from \cite{GTZ} in the integer setting (see \cite{CSinverse}) and the Tao--Ziegler inverse theorem from \cite{T&Z-Low} in the finite-field setting (see \cite{CGSS-p-hom}).}

Nilmanifolds are the most important examples of compact nilspaces, and already in the initial paper \cite{CamSzeg} it was proved that large classes of compact nilspaces can be described in terms of nilmanifolds. An interesting aspect of this theory, though, is that there exist compact nilspaces that are \emph{not} nilmanifolds. This led to the question of whether, to carry out higher-order Fourier analysis on arbitrary finite abelian groups, nilmanifolds are sufficient, or whether, on the contrary, some families of finite abelian groups need more general nilspaces for this purpose. In this direction, a central conjecture of Jamneshan and Tao proposes an inverse theorem for arbitrary finite abelian groups which involves polynomial maps taking values in nilmanifolds; see \cite[Conjecture 1.11]{J&T}.

In this paper, we make progress towards confirming the Jamneshan--Tao conjecture, by proving a first inverse theorem for all finite abelian groups that uses only nilmanifolds (rather than possibly more general nilspaces). In our result, the correlating function based on a nilmanifold is what we call a \emph{projected nilsequence}. Roughly speaking, the value of this function at a given point in the abelian group is obtained by averaging the values of a nilsequence defined on a boundedly larger abelian group extending the original one (a similar concept of a \emph{projected polynomial phase function} appeared in the bounded-torsion setting, in \cite[Definition 1.10]{CGSS-bndtor}).

Recall that the \emph{exponent} of a group $\ab$, denoted $\exp_{\ab}$, is the least common multiple of the orders of all elements of $\ab$.

\begin{defn}[Projected nilsequence]\label{def:pronilseq}
Let $\ab$ be a finite abelian group and let $k\in \mb{N}$. A \emph{projected $k$-step nilsequence} on $\ab$ is a function $\phi_{*\tau}:\ab\to\mb{C}$ of the following form. There is a finite abelian group $\ab'$, a surjective homomorphism $\tau:\ab'\to \ab$, a filtered nilmanifold $(G/\Gamma,G_\bullet)$ of degree $k$, a polynomial map $g:\ab'\to G/\Gamma$ relative to $G_\bullet$, and a 1-bounded continuous function $F:G/\Gamma\to\mb{C}$ such that $\phi_{*\tau}(x) = \mb{E}_{y\in \tau^{-1}(x)}F(g(y))$ for every $x\in\ab$. We say that $\phi_{*\tau}$ has \emph{complexity} at most $C$ if the nilsequence $F(g(y))$ has complexity at most $C$ as per \cite[Definition 1.3]{CSinverse}. We say that $\phi_{*\tau}$ is \emph{rank-preserving}\footnote{The \emph{rank} $\textrm{rk}(\ab)$ is the minimum cardinality of a set of generators of $\ab$.} if $\textrm{rk}(\ab')=\textrm{rk}(\ab)$, and that $\phi_{*\tau}$ has \emph{torsion} $m'$ if $\exp_{\ab'}\le m'$.
\end{defn}
We can now state the inverse theorem.
\begin{theorem}\label{thm:projinv}
For any $k\in\mb{N}$ and $\delta>0$, there exists $C=C(k,\delta)>0$ and $\varepsilon=\varepsilon(\delta,k)>0$ such that the following holds. For any finite abelian group $\ab$ and any 1-bounded function $f:\ab\to \mb{C}$ with $\|f\|_{U^{k+1}}\ge \delta$, there exists a projected $k$-step nilsequence $\phi_{*\tau}$ on $\ab$ of complexity at most $C$ such that $|\mb{E}_{x\in\ab} f(x) \overline{\phi_{*\tau}(x)}|\ge \varepsilon$. Moreover $\phi_{*\tau}$ is rank-preserving and has torsion $C\exp_{\ab}^C$.
\end{theorem}
\noindent We prove this in Section \ref{sec:projinv}. We also prove in that section that projected nilsequences are genuine obstructions to a function having small Gowers norm (see Proposition \ref{prop:directthm}). Thus Theorem \ref{thm:projinv} provides a first genuine inverse theorem for finite abelian groups solely in terms of nilmanifolds.

While nilmanifolds suffice for the statement of Theorem \ref{thm:projinv}, the proof makes crucial use of more general nilspace theory. In particular, it relies on a new result which is the main  contribution of this paper. To state this, we use the following definition (the basic terms from nilspace theory involved in this definition are recalled in Section \ref{sec:prelim}).
\begin{defn}[Nilspace extensions and factors]
A \emph{nilspace extension} of a nilspace $\ns$ is a nilspace $\nss$ such that there is a fibration $\nss\to\ns$. Accordingly, if $\nss$ is a nilspace extension of $\ns$ then $\ns$ is said to be a \emph{nilspace factor} of $\nss$. When $\ns$ and $\nss$ have topologies compatible with the cube structures\footnote{E.g.\ if $\ns$ and $\nss$ are both compact nilspaces, or locally-compact-Hausdorff nilspaces (see \cite[Definition 2.1]{CGSS-doucos}).}, then we require the fibration to be continuous.
\end{defn}
\noindent Our main result identifies a special class of compact nilspaces of finite rank which are also nilmanifolds and are similar to the \emph{universal nilmanifolds} introduced in \cite[Definition 9.1]{GTZ} (though not identical; see Definition \ref{def:univCFRns}). These \emph{universal nilspaces}, as we call them, are the central objects in our main result, which we can now state.
\begin{theorem}\label{thm:cfr-are-factor-of-nilmanifolds}
Every $k$-step compact finite-rank nilspace is a factor of a $k$-step universal nilspace.
\end{theorem}
\noindent Let us outline how this result yields Theorem \ref{thm:projinv} (we give a more detailed explanation in Section \ref{sec:projinv}). The inverse theorem \cite[Theorem 1.6]{CSinverse} yields a nilspace polynomial of the form $F\co \varphi:\ab\to \mb{C}$ correlating with the original function $f:\ab\to\mb{C}$. By Theorem \ref{thm:cfr-are-factor-of-nilmanifolds}, we have a fibration $\psi:\nss\to\ns$ where the nilspace extension $\nss$ is a nilmanifold. In general, we presently are unable to ensure the existence of a morphism from $\ab$ to $\nss$ whose composition with $\psi$ equals $\varphi$ (the existence of such a morphism would confirm the Jamneshan--Tao conjecture). However, we are able to find a finite abelian group $\ab'$ boundedly larger than $\ab$ with a surjective homomorphism $\tau:\ab'\to\ab$ and a morphism $\varphi':\ab'\to\nss$ such that $\varphi\co\tau=\psi\co\varphi'$. This yields the correlating projected nilsequence $\phi_{*\tau}$ (where the polynomial map $g$ is the nilspace morphism $\varphi'$). This plan can be visualized with the following diagram.

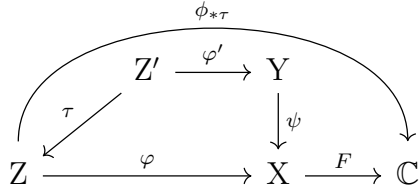
\begin{figure}[h]
\vspace{-0.6cm}
\[
\begin{tikzcd}
	&& {} \\
	& {\ab'} & \nss \\
	\ab && \ns & {\mb{C}}
	\arrow["{\varphi'}", from=2-2, to=2-3]
	\arrow["\tau"', from=2-2, to=3-1]
	\arrow["\psi", from=2-3, to=3-3]
	\arrow["\varphi", from=3-1, to=3-3]
	\arrow["{\phi_{*\tau}}", shift left=1, curve={height=-30pt}, from=3-1, to=3-4,  bend left=90]
	\arrow["F", from=3-3, to=3-4]
\end{tikzcd}
\]
\captionsetup{justification=centering,margin=2cm}
\caption{Summary of the proof of Theorem \ref{thm:projinv}.}
\end{figure}
\noindent Note that the process of extending from the original morphism $\varphi:\ab\to\ns$ to the new morphism $\varphi':\ab'\to\nss$, and using the corresponding projected nilsequence, resonates naturally with notions of generalized extensions of measure-preserving group actions in ergodic theory, such as those considered in \cite[Definition 1.15]{Shalom}.

Theorem \ref{thm:cfr-are-factor-of-nilmanifolds} has applications in areas other than arithmetic combinatorics. To illustrate this in this paper, we provide an application in topological dynamics, giving a description of topological $\mb{Z}^\omega$-systems of finite order in terms of much simpler systems analogous to nilsystems. This application is motivated by a question of Jamneshan, Shalom and Tao on the structure of ergodic $\mb{Z}^\omega$-systems of finite order \cite[Question 1.5]{JST1}. To explain this further, we shall use standard notions from ergodic theory and topological dynamics. We assume these notions in this introduction but we will recall them in more detail in Section \ref{sec:ergapp}.

Recall that a nilsystem consists of a nilmanifold $G/\Gamma$ equipped with a group action (classically, a $\mb{Z}$-action) consisting of maps $T:G/\Gamma\to G/\Gamma$ defined by $x\Gamma\mapsto tx\Gamma$ for some $t\in G$. Nilsystems and their inverse limits are crucial for the structural analysis of Host--Kra factors for various kinds of measure-preserving group actions, as well as for the related analysis in topological dynamics concerning higher-order regionally proximal relations. However,  most of these results deal with actions of groups that are either finitely generated (in ergodic theory, see the original result \cite[Theorem 10]{HK-non-conv} and extensions such as \cite[Theorem 5.12]{CScouplings}) or compactly generated (in topological dynamics, see \cite{GMV3}). For actions of larger groups, much less is known on whether such structural descriptions still hold in terms of nilsystems. Some classes of such larger group actions have been studied. For instance, in the ergodic setting, the actions of the (non-finitely generated) additive group of the countable vector space $\mb{F}_p^\omega$ have been studied in \cite{BTZ,CGSS-abramov, JST2}, and other so-called \emph{totally disconnected systems} are analyzed in \cite{JST1} (see also  \cite{JST3, Shalom,Shalom2}). 

In \cite[\S 1.4]{JST1} the authors pose several questions regarding the structure of Host--Kra factors for actions of groups that are not finitely generated. In particular, Question 1.5 in \cite{JST1} asks whether every ergodic $\mb{Z}^\omega$-system of order $k$ is isomorphic to an inverse limit of \emph{translational systems} of degree $k$, these being generalizations of nilsystems (see \cite[Definition A.2]{JST1}).\footnote{Strictly speaking, \cite[Question 1.5]{JST1} is posed in the negative sense: is there any ergodic $\mb{Z}^\omega$-system of order $k$ that is \emph{not} isomorphic to an inverse limit of translational systems of degree $k$? We use the positive version above.} Using Theorem \ref{thm:cfr-are-factor-of-nilmanifolds}, we provide a partial positive answer to this question in the topological setting. Instead of translational systems, the generalizations of nilsystems that we use are what we call \emph{$\mb{Z}^\omega$-polynomial orbit systems} (see Definition \ref{def:gen-poly-orb}). These systems are indeed natural generalizations of nilsystems, in the following sense. Given a filtered nilmanifold $(G/\Gamma,G_\bullet)$, the set of polynomial maps $\mb{Z}\to G/\Gamma$, denoted by $\poly(\mb{Z},G/\Gamma)$ and equipped with the shift $(g(n)\Gamma)_n \mapsto (g(n+1)\Gamma)_n$, can be viewed as a nilsystem called a \emph{nilsystem of polynomial orbits} (see \cite[Ch.\ 14, Proposition 13]{HKbook}). If $(G,G_\bullet)$ is of degree $k$, then the closure, inside $\poly(\mb{Z},G/\Gamma)$, of the set of shifts of a single polynomial orbit $\big(g(n)\Gamma\big)_{n\in \mb{Z}}$, can also be viewed as a degree-$k$ nilsystem (see \cite[Ch.\ 14, Proposition 14]{HKbook}). Conversely, any ergodic or minimal\footnote{For nilsystems, these properties are equivalent; see \cite[Ch.\ 11, Theorem 11]{HKbook}.} nilsystem $(G/\Gamma,T:x\Gamma\mapsto tx\Gamma)$ is isomorphic to the nilsystem of the polynomial orbit $\overline{\orb(g\Gamma)}$ generated by $g:\mb{Z}\to G$ given by $n\mapsto t^n$ (this can be proved using Lemma \ref{lem:equality-orbits}). The notion of a $\mb{Z}^\omega$-polynomial orbit system is obtained by replacing the original polynomial map $g(n)\Gamma$ in $\poly(\mb{Z},G/\Gamma)$ by a polynomial $g(x)\Gamma\in \poly(\mb{Z}^\omega,G/\Gamma)$; see Definition \ref{def:gen-poly-orb}.

With these notions, our application in dynamics can now be stated.
\begin{theorem}\label{thm:main-dynam}
Let $\ns$ be a minimal $\mb{Z}^\omega$-system of order $k$. Then $\ns$ is the factor of an inverse limit of minimal $\mb{Z}^\omega$-polynomial orbit systems of order $k$.
\end{theorem}
\noindent We believe that with additional work in the ergodic setting, Theorem \ref{thm:cfr-are-factor-of-nilmanifolds} can yield a result closer still to a complete answer to \cite[Question 1.5]{JST1}. Specifically, this result would express every ergodic $\mb{Z}^\omega$-system of finite order as a factor of an inverse limit of translational systems.

\medskip

The paper has the following outline. In Section \ref{sec:prelim} we gather and explain some necessary background notions from nilspace theory. In Section \ref{sec:cfrmainresult} we prove Theorem \ref{thm:cfr-are-factor-of-nilmanifolds}. The inverse theorem with projected nilsequences, Theorem \ref{thm:projinv}, is proved in Section \ref{sec:projinv}. Finally, in Section \ref{sec:ergapp} we present the application in dynamics, Theorem \ref{thm:main-dynam}.

\section{Preliminaries}\label{sec:prelim}
\noindent As mentioned in the introduction, nilspace theory already has much literature devoted to it. In this section, we review only the main concepts and results that we need for this paper, providing references where appropriate. 

Let us begin by defining nilspaces, following \cite[\S 1]{Cand:Notes1}. This requires the following concepts. For each positive integer $n$ let $\db{n}$ be the set $\{0,1\}^n$, and let $\db{0}:=\{0\}$.

\begin{defn}[Discrete-cube morphisms] Let $n,m\ge 0$. A function $\phi:\db{n}\to \db{m}$ is a \emph{discrete-cube morphism} if it is the restriction of an affine homomorphism\footnote{An affine homomorphism $f:Z_1\to Z_2$ between abelian groups $Z_1$ and $Z_2$ is a function $f(z)=g(z)+t$, where $g:Z_1\to Z_2$ is a homomorphism and $t\in Z_2$.} $f:\mb{Z}^n\to \mb{Z}^m$.
\end{defn}

\begin{defn}[Faces and face maps] Let $n\ge 0$. A \emph{face} $F$ of dimension $m\le n$ is a subset of $\db{n}$ defined by fixing $n-m$ coordinates, that is, a set of the form $F=\{\underline{v}\in\db{n}: v(i)=t(i), \ i\in I\}$ for some $I\subset \{1,\ldots,n\}$ with $|I|=n-m$ and $t(i)\in \{0,1\}$ for all $i\in I$. A \emph{face map} $\phi:\db{k}\to \db{n}$ is an injective discrete-cube morphism such that $\phi(\db{k})$ is a face.
\end{defn}

\begin{defn}[Nilspaces]\label{def:nilspace} A \emph{nilspace} is a set $\ns$ equipped with a collection of sets $\cu^n(\ns)\subset \ns^{\db{n}}$ for $n\in \mb{Z}_{\geq 0}$, such that the following axioms are satisfied:
\setlength{\leftmargini}{0.8cm}
\begin{enumerate}
    \item (Composition) For every discrete-cube morphism $\phi:\db{m}\to\db{n}$ and every $\q\in\cu^n(\ns)$, we have $\q\co\phi \in\cu^m(\ns)$.
    \item (Ergodicity) $\cu^1(\ns)=\ns^{\db{1}}$.
    \item (Corner completion) Let $\q':\db{n}\setminus \{1^n\}\to\ns$ be a function such that if $\phi:\db{n-1}\to\db{n}$ is a face map with $\phi(\db{n-1})\subset \db{n}\setminus\{1^n\}$, then $\q'\co\phi\in\cu^{n-1}(\ns)$. Then there exists $\q\in\cu^n(\ns)$ such that $\q(v)=\q'(v)$ for all $v\in \db{n}\setminus\{1^n\}$.
\end{enumerate}
\noindent The elements of $\cu^n(\ns)$ are called the \emph{$n$-cubes} on $\ns$. A map $\q'$ as in axiom $(iii)$ is called an \emph{$n$-corner} on $\ns$, and a cube $\q$ satisfying the conclusion of this axiom is a \emph{completion} of $\q'$. We say that $\ns$ is a \emph{$k$-step} nilspace if for $n=k+1$ every $n$-corner has a unique completion. 

Finally, if $\ns$ is endowed with a compact second-countable Hausdorff topology and for every $n\in\mb{N}$ the set $\cu^{n}(\ns)$ is closed in the product topology on $\ns^{\db{n}}$, then $\ns$ is a \emph{compact nilspace}.
\end{defn}
For any nilspace and any positive integer $k$, there is a smallest equivalence relation such that the corresponding quotient space is a $k$-step nilspace. This yields the following key notion.
\begin{defn}[Canonical $k$-step factor of a nilspace]
Let $\ns$ be a nilspace. For each integer $k\ge 1$, we define the relation $\sim_k$ on $\ns$ by $x\sim_k y \iff \exists\, \q_0,\q_1\in \cu^{k+1}(\ns)\ \text{such that}\ \q_0(0^{k+1})=x,\ \q_1(0^{k+1})=y,\ \text{and}\ \q_0(v)=\q_1(v)\ \forall\, v\neq 0^{k+1}$. Let $\pi_k:\ns\to\ns/\sim_k$ be the quotient map. Then the space $\ns/\sim_k$ together with the cubes $\cu^n(\ns/\sim_k):=\pi_k^{\db{n}}(\cu^n(\ns))$ is a $k$-step nilspace, known as the (canonical) \emph{$k$-step factor} of $\ns$.
\end{defn}
\noindent Nilspace theory has proved to be very useful for giving unified descriptions and explanations of a large variety of phenomena in higher-order Fourier analysis. Notwithstanding this, nilspace theory remains less developed than the machinery based on the more specific and classical concept of nilmanifolds (we recall their definition below). In particular, in ergodic theory and topological dynamics, nilsystems have been much more studied than their more recent nilspace-theoretic generalizations (known as \emph{nilspace systems}), see \cite[Ch.\ 11]{HKbook}; in arithmetic combinatorics, some of the most important inverse theorems use nilmanifolds \cite{GT08, GTZ-U4, GTZ}. 

The Jamneshan--Tao conjecture \cite[Conjecture 1.11]{J&T} can be viewed as motivation for developing stronger connections between compact nilspaces and nilmanifolds. The main result of this paper, Theorem \ref{thm:cfr-are-factor-of-nilmanifolds}, offers a new connection of this type which is the strongest to date in its level of generality. To explain it in more detail, let us now recall the concept of a filtered nilmanifold, which involves the following notion.

\begin{defn}[Filtration]
A \emph{filtration} on a group $G$ is a sequence $G_\bullet = (G_i)_{i=0}^\infty$ of subgroups of $G$ satisfying $G = G_0 = G_1 \geq G_2 \geq \cdots$and such that $[G_i, G_j] \subseteq G_{i+j}$ for all $i,j \ge 0$.  We refer to $(G, G_\bullet)$ as a \emph{filtered group}. If $G_{k+1} = \{\mathrm{id}_G\}$ we say that the filtered group $(G, G_\bullet)$ is \emph{of degree $k$}.
\end{defn}

\begin{defn}[Nilmanifold]\label{def:nilmanifold}
A \emph{nilmanifold} is a quotient space $G/\Gamma$ where $G$ is a nilpotent Lie group\footnote{In our definition we do \emph{not} require $G$ to be connected. For clarity, we occasionally emphasize this in statements of theorems. When we work with \emph{connected} nilmanifolds, we always state it explicitly.} and $\Gamma$ is a discrete, cocompact subgroup of $G$. If $G_\bullet$ is a filtration on $G$ of degree at most $k$, with each $G_i$ a closed subgroup of $G$ and with each subgroup $\Gamma \cap G_i$ cocompact in $G_i$, then we call $(G/\Gamma, G_\bullet)$ a \emph{filtered nilmanifold of degree $k$}.
\end{defn}
\noindent To begin relating compact nilspaces with nilmanifolds, let us first recall that any degree-$k$ nilmanifold $G/\Gamma$ is automatically a compact $k$-step nilspace (see \cite[Proposition 1.1.2]{Cand:Notes2})\footnote{The statement of this result assumes $G$ is connected, but the proof works exactly the same without this assumption.} when equipped with the projections of the \emph{Host-Kra cubes} on $(G,G_\bullet)$, which we recall here.
\begin{defn}[Host-Kra cubes, group nilspaces and coset nilspaces]\label{def:hk-cubes}
Let $(G, G_\bullet)$ be a filtered group. Given a face $F\subseteq\{0,1\}^n$ and $g\in G$, let $g^{F}\in G^{\{0,1\}^n}$ be defined by $g^{F}(v)=g$ if $v\in F$ and $\id_G$ otherwise. Then if we let $\cu^n(G):=\{\,g^{F}:F\text{ face in }\db{n},g\in G_{\mathrm{codim}(F)}\,\}\le G^{\{0,1\}^n}$ the set $G$ together with the cube sets $\cu^n(G)$ for $n\ge 0$ is a nilspace, called the \emph{group nilspace} associated with $(G,G_\bullet)$. If $G_\bullet$ has degree $k$ then the corresponding group nilspace has step $k$. Finally, if $\Gamma$ is any subgroup of $G$ then the quotient set $G/\Gamma$ equipped with the cubes $\cu^n(G/\Gamma):=\cu^n(G_\bullet)\Gamma^{\db{n}}$ is a $k$-step nilspace called a \emph{coset} nilspace.\footnote{This is proved in \cite[\S 2.3]{Cand:Notes1}.}
\end{defn}
\noindent An important example of group nilspace is the one associated with an abelian group $G$ equipped with the filtration $G_\bullet$ where $G_i=G$ for $i\le k$ and $G_{i}=\{\id\}$ for $i>k$. In this case, we will denote this nilspace (often called a \emph{degree-$k$ abelian group}) by $\mc{D}_k(G)$.

When viewed as coset nilspaces, filtered nilmanifolds are actually included in a specific class of compact nilspaces which play an important role in the theory, namely the \emph{compact finite-rank} (\textsc{cfr}) nilspaces. We briefly recall their definition. 

Recall first that given a $k$-step nilspace $\ns$, for $i\in [k]$ the \emph{$i$-th structure group} $\ab_i$ of $\ns$ is an abelian group such that the $i$-th nilspace factor $\ns_i$ is an \emph{abelian $\ab_i$-bundle} over $\ns_{i-1}$ whose projection map is precisely the nilspace factor map $\ns_i\to\ns_{i-1}$ (see \cite[\S 3.2.3]{Cand:Notes1}). If $\ns$ is a compact nilspace, then $\ab_i$ becomes a \emph{compact} abelian group \cite[\S 2.1.1]{Cand:Notes2} (with the relative topology when $\ab_i$ is identified with any fiber of the bundle $\ns_i$).

\begin{defn}[\textsc{cfr} nilspaces]
A compact $k$-step nilspace $\ns$ is of \emph{finite rank} if for every $i\in[k]$ the structure group $\ab_i$ of $\ns$ is a compact abelian \emph{Lie} group (equivalently, the Pontryagin dual of $\ab_i$ has finite rank, i.e.\ finitely many generators).
\end{defn}
\noindent Every filtered nilmanifold is a \textsc{cfr} nilspace. On the other hand, there are \textsc{cfr} nilspaces that are not nilmanifolds, in fact not even coset spaces \cite[Example 6]{HK-par}. The above-mentioned goal of connecting nilspaces with nilmanifolds thus focuses on expressing the former, as directly and usefully as possible, in terms of the latter. An early result of this type follows from \cite[Theorems 4 and 7]{CamSzeg} (see also \cite[Theorems 2.7.3 and 2.9.17]{Cand:Notes2}) and expresses every compact nilspace with connected structure groups as a strict inverse limit of connected nilmanifolds. This is useful for various applications, but more direct expressions have since been shown to be more useful for progress on central questions in the area. For instance, the authors proved in \cite[\S 7.4]{CGSS-doucos} that a 2-step \textsc{cfr} nilspace can always be embedded (via a nilspace morphism) into a connected nilmanifold, and that this suffices to prove the case $k=3$ of the Jamneshan--Tao conjecture. We shall now recall some of the main tools that went into these results, as they will also play a key role in this paper.

\begin{defn}[Translation groups and related homomorphisms]\label{def:trans-gr}
Let $\ns$ be a $k$-step nilspace and let $i\in[k]$. A map $\alpha:\ns\to \ns$ is a \emph{translation} of height $i$ if for every $n\ge i$ and every face $F\subseteq\db{n}$ of codimension $i$, the map $\alpha^{F}:\ns^{\db{n}}\to\ns^{\db{n}}$ given by $\alpha^F(f)(v)=\alpha(f(v))$ if $v\in F$ and $\alpha^F(f)(v)=f(v)$ otherwise is cube-preserving (i.e.\ $\alpha^{F}(\q)\in \cu^{n}(\ns)$ for all $\q\in \cu^{n}(\ns)$). These translations form a group, denoted $\tran_{i}(\ns)$, and the sequence ($\tran_{i}(\ns)\big)_{i\geq 0}$ is a filtration of degree $k$ (see \cite[\S 3.2.4]{Cand:Notes1}). The canonical nilspace factor map $\pi_i:\ns\to\ns_i$ induces, for every $j\in[k]$, a homomorphism $\eta_i:\tran_j(\ns)\to\tran_j(\ns_i)$, well-defined by the formula $\eta_i(\alpha)(\pi_i(x)):=\pi_i(\alpha(x))$.
\end{defn}

\begin{remark}
We leave it as an exercise to check that if in the previous definitions, instead of groups, filtrations, subgroups, and nilspaces we consider Lie groups, filtrations of closed subgroups, discrete co-compact subgroups, and compact finite-rank nilspaces, then, with the natural changes, the previous results adapt to the topological setting. For example, a coset nilspace $G/\Gamma$ as in Definition \ref{def:hk-cubes},  satisfying the filtered-nilmanifold hypotheses of Definition \ref{def:nilmanifold}, would be a compact nilspace, and if the translations $\alpha$ in Definition \ref{def:trans-gr} are assumed to be continuous then $\tran_i(\ns)$ is a Lie group. For further details and proofs, see \cite{Cand:Notes2}.
\end{remark}
We now briefly recall another expression for \textsc{cfr} nilspaces that was given in \cite{CGSS-doucos}, which will be a key ingredient in this paper. Let us start by recalling some concepts from \cite[\S 1]{CGSS-doucos}.

\begin{defn}\label{def:free-nil-intro}
A \emph{free nilspace} is a direct product (in the nilspace category) of finitely many components of the form $\mc{D}_i(\mb{R})$ and $\mc{D}_i(\mb{Z})$ where $i\in \mb{N}$. We say that a free nilspace is \emph{discrete} if it is a direct product of components of the form $\mathcal{D}_i(\mathbb{Z})$, and that it is \emph{continuous} if it is a direct product of components of the form $\mc{D}_i(\mb{R})$. Given a general $k$-step free nilspace $F$, we write $F=C\times D$ where $C=\prod_{i=1}^k\mc{D}_i(\mb{R}^{b_i})$ is called the \emph{continuous part} of $F$ and $D=\prod_{i=1}^k\mc{D}_i(\mb{Z}^{a_i})$ is called the \emph{discrete part}. We define the \emph{projection to the discrete part} on $F$ to be the coordinate projection $p_D:C\times D\to D$.
\end{defn}

It was shown in \cite{CGSS-doucos} that free nilspaces are key to understanding \textsc{cfr} nilspaces. Indeed, the main results of \cite{CGSS-doucos} establish that free nilspaces play a role of \emph{projective objects} in this category, and that every \textsc{cfr} nilspace has a \emph{projective presentation} in terms of a free nilspace. More precisely, by Theorem 6.2 in \cite{CGSS-doucos},  every \textsc{cfr} nilspace can be expressed as a quotient of a free nilspace by the action of a certain type of group of translations, which we now recall.

\begin{defn}[Fiber-transitive group of translations]\label{def:fib-tran}
Let $F$ be a $k$-step free nilspace and let $\Gamma$ be a subgroup of the translation group $\tran(F)$. We say that $\Gamma$ is a \emph{fiber-transitive group on} $F$ if the following holds: for all $x,y\in F$, if there exists $\gamma\in \Gamma$ and $i\in[k]$ such that $\gamma(x)=y$ and $\pi_i(x)=\pi_i(y)$, then there exists $\gamma'\in \Gamma\cap \tran_{i+1}(F)$ such that $\gamma'(x)=y$.
\end{defn}
\noindent To this purely algebraic notion it is convenient to add the following properties when working with topological nilspaces.
\begin{defn}[Translation lattices on free nilspaces]\label{def:FDCA}
Let $F$ be a $k$-step free nilspace and let $\Gamma$ be a subgroup of $\tran(F)$. We say that $\Gamma$ is \emph{fiber-discrete} if for every $j\in[k]$, the group $\eta_j(\Gamma)\cap \tran_j(F_j)$ is a discrete subgroup of $\tran_j(F_j)\cong \mb{Z}^{a_j}\times \mb{R}^{b_j}$. We say that $\Gamma$ is \emph{fiber-cocompact} if for every $j\in[k]$, the group $\eta_j(\Gamma)\cap \tran_j(F_j)$ is cocompact in $\tran_j(F_j)$. Finally, we say that $\Gamma$ is a \emph{translation lattice} on $F$ if $\Gamma$ is fiber-transitive, fiber-discrete and fiber-cocompact.
\end{defn}

\noindent We can now give the statement of the above-mentioned result.

\begin{theorem}[See Theorem 6.2 in \cite{CGSS-doucos}]\label{thm:cfr-nil}
Let $\ns$ be a $k$-step \textsc{cfr} nilspace. Then there is a $k$-step free nilspace $F$, and a finitely-generated translation lattice $\Gamma\le \tran(F)$ such that, letting $\pi_{\Gamma}$ be the corresponding quotient map \textup{(}sending each $x\in F$ to the orbit of $x$ under the action of $\Gamma$\textup{)}, we have that $\ns$ is isomorphic as a compact nilspace to $\pi_\Gamma(F)$. 
\end{theorem}

\noindent Abusing the notation, we will often write $F/\Gamma$ instead of $\pi_\Gamma(F)$. 

The description of a \textsc{cfr} nilspace $\ns$ as a quotient $F/\Gamma$ provides a method to locate useful nilspace  extensions of $\ns$ (recall that these are nilspaces $\nss$ with fibrations $\nss\to\ns$), by replacing the quotienting group $\Gamma$  by adequate subgroups of $\Gamma$. This method is quite naturally motivated by algebraic topology (recall that if $X$ is a topological space, then coverings of $X$ can be represented by subgroups of the fundamental group of $X$).

In Section \ref{sec:cfrmainresult} we apply this method by locating subgroups $H\le \Gamma$ which are also translation lattices on $F$, and such that $F/H$ has additional useful properties (this method is especially used to prove Theorem \ref{thm:wscover} below). 

We close this section by recalling a property of translations which was introduced in \cite{CGSS-doucos}, called \emph{pureness}. Note that for any nilspace $\ns$, for the canonical homomorphisms $\eta_i:\tran(\ns)\to \tran(\ns_i)$ we have $
\tran(\ns)=\ker(\eta_0)\geq \ker(\eta_1) \geq \cdots \geq \ker(\eta_{k-1})\geq \ker(\eta_k=\id)=\{\id\}$.
Thus we have the partition
\begin{equation}\label{eq:transparti}
\tran(\ns)=\{\id\}\cup\bigsqcup_{i\in [k]} \ker(\eta_{i-1})\setminus \ker(\eta_i).
\end{equation}
We also have in general $\ker(\eta_{i-1})\supset \tran_i(\ns)$ and this inclusion can be strict, which can often complicate the analysis of the structure of $\ns$. Pureness ensures that this inclusion is an equality.

\begin{defn}[Pure translations; see Definition 5.40 in \cite{CGSS-doucos}]\label{def:pure} Let $\ns$ be a nilspace. A translation $\alpha\in\tran(\ns)$ is \emph{pure} if for every $i\in [k]$ we have $\alpha\in \ker(\eta_{i-1})\Rightarrow \alpha\in\tran_i(\ns)$. Equivalently $\alpha$ is pure if, for the maximal $i\in [k]$ such that $\alpha\in\ker(\eta_{i-1})$, we have $\alpha\in \tran_i(\ns)$. We say that a group of translations $G\leq\tran(\ns)$ is \emph{pure} if every translation in $G$ is pure, which is equivalent to the following property: for every $i\in [k]$, we have $G\cap \ker(\eta_{i-1})=G\cap \tran_i(\ns)$.
\end{defn}

\section{\texorpdfstring{\textsc{cfr} nilspaces are factors of nilmanifolds}{CFR nilspaces are factors of nilmanifolds}}\label{sec:cfrmainresult}
\noindent In this section we prove the main result of this paper, Theorem \ref{thm:cfr-are-factor-of-nilmanifolds}. The plan for this is, firstly, to outline the main new notions and ingredients used in the proof, then present the proof itself assuming some key results. After this, we shall present the proofs of these key results in various subsections. 

\medskip

The first main notion is the following new class of nilspaces. 

\begin{defn}[Weakly-splitting nilspaces]\label{def:weak-splitting}
A $k$-step \textsc{cfr} nilspace $\ns$ is \emph{weakly-splitting} if for every $i\in [k]$ the $i$-step factor $\ns_i$ is a product nilspace of the form $\ns_i=\nss_i\times\mc{D}_i(B_i)$ where $B_i$ is a finite abelian group and the $i$-th structure group of $\nss_i$ is a torus.
\end{defn}
\noindent To motivate this notion, recall that to prove Theorem \ref{thm:cfr-are-factor-of-nilmanifolds} we need to find a nilspace extension $\nss\to\ns$ where $\nss$ is a filtered nilmanifold. If $\nss$ had all of its structure groups being connected (hence tori), the proof would be complete, since nilspaces of this type (called \emph{toral} nilspaces) are known to be (connected) nilmanifolds; see \cite[Theorem 2.9.17]{Cand:Notes2}. Even if such an extension $\nss$ could be found, this would be possible only when $\ns$ was connected to begin with. In general this is not the case, and it is then natural to aim for a nilspace extension $\nss\to\ns$ where, in $\nss$, one can clearly separate the connected part from the discrete finite part. Indeed, it would be ideal to be always able to find an extension $\nss$ that is a product nilspace $\nss_t\times\nss_f$ where $\nss_t$ is a toral nilspace and $\nss_f$ is a finite nilspace,\footnote{This product decomposition of $\nss$ would be a nilspace analogue of the expression of compact abelian Lie groups as products of tori with finite abelian groups.} as then $\nss_t$ would already be a nilmanifold, and $\nss_f$ could further be extended to a finite coset nilspace by known results (see  \cite[Theorem 7]{SzegFin}). We do not reach this ideal situation in this paper, nor do we know whether it can always be reached. However, in Subsection \ref{subsec:ws} we take a significant step towards it. Indeed, we obtain the following theorem guaranteeing that $\ns$ can always be extended to a weakly-splitting nilspace (so that the product decomposition desired above occurs at least in the $i$-th bundle-layer of each nilspace factor $\nss_i$). Proving this is also a first important step in our proof of Theorem \ref{thm:cfr-are-factor-of-nilmanifolds}.

\begin{theorem}\label{thm:wscover} 
For every $k$-step \textsc{cfr} nilspace $\ns$, there is a $k$-step weakly-splitting \textsc{cfr} nilspace $\nss$ and a fibration $\varphi:\nss\to\ns$.
\end{theorem}
\begin{remark}
It can be deduced from the arguments that for every $x\in \ns$ the preimage $\varphi^{-1}(\{x\})$ is finite.
\end{remark}
\noindent Note that this result generalizes \cite[Theorem 7]{SzegFin}, the latter being the special case for finite nilspaces. Indeed, a finite and weakly-splitting nilspace is always of the form $\prod_{i=1}^k\mc{D}_i(B_i)$ where $B_i$ is a finite abelian group for every $i\in [k]$.

Another useful feature of a nilspace $\ns$ being weakly-splitting is that this property can be reformulated in terms of a group-theoretic property in a certain presentation of the form $\ns=F/\Gamma$ where $F$ is a free nilspace and $\Gamma$ is a translation lattice in $\tran(F)$. We call this property \emph{orthogonality} of $\Gamma$.

\begin{defn}[Orthogonality]\label{def:ortlattice} Let $F=\prod_{i=1}^k\mathcal{D}_i(\mathbb{Z}^{b_i})$ be a discrete free nilspace. We can identify $F$ with $\mb{Z}^b$ where $b=\sum_{i=1}^k b_i$. We say that $\Gamma\leq\tran(F)$ is an \emph{orthogonal lattice} if $\Gamma$ is generated by translations of the form $x\mapsto x+(0,\ldots,0,w_i,0,\ldots,0)$ for $w_i\in \mb{Z}$ and $i\in [b]$. We call such  translations \emph{elementary generators}. For a general $k$-step nilspace $F'$ with discrete part $D$, we say that $\Gamma\leq\tran(F')$ has the \emph{orthogonality property} (or just that $\Gamma$ is \emph{orthogonal}) if $\wh{p_D}(\Gamma)$ is an orthogonal lattice and for every $i\in[k]$ we have $\wh{p_D}(\Gamma)\cap \tran_i(D)=\wh{p_D}(\Gamma\cap \tran_i(F'))$.
\end{defn}
\noindent We develop this property's relationship with weak splitting in Subsection \ref{subsec:ortho-ws} (see Proposition \ref{prop:ws<=>nr}).

\medskip

The next main concept used in our proof of Theorem \ref{thm:cfr-are-factor-of-nilmanifolds} is another useful class of nilspaces. 
Here let us simply state the definition and defer its more detailed discussion to Subsection \ref{subsec:LieParam}.
\begin{defn}[Ported nilspaces]\label{def:nicerep}
Given a free nilspace $F$, we say that a subgroup $\Gamma\le \tran(F)$ is \emph{porting} if it is a pure and orthogonal translation lattice. We say that a $k$-step \textsc{cfr} nilspace $\ns$ has a  \emph{ported presentation} if there is a $k$-step free nilspace $F$ and a porting subgroup $\Gamma\leq \tran(F)$ such that $\ns$ is isomorphic to $F/\Gamma$ as a compact nilspace. We say that a \textsc{cfr} nilspace is \emph{ported} if it has a ported presentation.
\end{defn}
\noindent This is a key notion for what follows. In particular, we will prove that ported nilspaces are coset nilspaces, hence nilmanifolds. Among ported nilspaces, there are those with the additional property that the porting group $\Gamma$ is a \emph{free graded} nilpotent group. Let us recall the latter concept before we formalize this special class of nilspaces.

\begin{defn}[Free graded nilpotent group]\label{def:fgngp}
Let $k\in\mb{N}$ and $\ell_1,\ldots,\ell_k\in \mb{Z}_{\ge 0}$. The \emph{free graded $k$-step nilpotent group} $\Gamma=\Gamma((\gamma_{i,j})_{i\in[k],j\in [\ell_i]})$ (or \emph{free $k$-step nilpotent group generated by the graded sequence  $(\gamma_{i,j})_{i\in[k],j\in [\ell_i]}$}; see \cite[Ch.\ 5]{MKS76}) is defined as follows. First let $\Gamma^*:=\langle \gamma_{i,j}\rangle_{i\in[k],j\in[\ell_i]}$ be the free group with generators $\gamma_{i,j}$ for $i\in [k], j\in [\ell_i]$. We define the \emph{degree} of any iterated commutator of the $\gamma_{i,j}$ by declaring that $\deg(\gamma_{i,j}):=i$ and $\deg([f,g]):=\deg(f)+\deg(g)$.\footnote{For example $\deg([[\gamma_{1,2},\gamma_{2,4}],\gamma_{5,1}])=\deg(\gamma_{1,2})+\deg(\gamma_{2,4})+\deg(\gamma_{5,1})=1+2+5=8$.} Then $\Gamma:=\Gamma^*/H$ where $H$ is the normal subgroup generated by iterated commutators of degree larger than $k$, i.e.\ $H=\langle h:h\in\comm{\Gamma^*},\;\deg(h)\ge k+1\rangle$.
\end{defn}
\noindent Recall that such a grading yields a filtration $(\Gamma_i)_{i\geq 0}$ on the group $\Gamma$ by letting $\Gamma_i$ be the subgroup generated by all commutators of degree at least $i$, see \cite[Theorem 5.13A]{MKS76}.

The special class of ported nilspaces in question is defined as follows.

\begin{defn}[Universal \textsc{cfr} nilspaces]\label{def:univCFRns}
A \emph{universal \textsc{cfr} nilspace} is a \textsc{cfr} nilspace $\ns$ which has a ported presentation $\ns=F/\Gamma$ such that $\Gamma$ is additionally a (finitely generated) free graded $k$-step nilpotent group where the degree of each generator $\gamma$ is the greatest integer $i$ such that $\gamma\in\tran_i(F)$, and where the resulting filtration on $\Gamma$ is equal to $\Gamma_\bullet = (\Gamma\cap \tran_i(F))_{i\in[k]}$.
\end{defn}
Note that the required equality of filtrations here can fail to hold for some ported presentations, so it is not a superfluous requirement. Indeed, setting the degree of the distinguished free generators to be their translation heights ensures that the group $\langle \text{commutators of degree at least }i\rangle$ is included in $\Gamma\cap\tran_i(F)$ for every $i$, but this inclusion can be strict.

The term ``universal" in Definition \ref{def:univCFRns} has several justifications. One of these is that universal \textsc{cfr} nilspaces are similar to the \emph{universal nilmanifolds} introduced in \cite[Definition 9.1]{GTZ} (note that universal nilspaces need not be connected as in \cite{GTZ}, and that they come equipped by definition with a nilspace structure). The main justification is that, to obtain Theorem \ref{thm:cfr-are-factor-of-nilmanifolds}, we will prove the following stronger result, establishing indeed a form of universality of these objects.

\begin{theorem}\label{thm:realmain}
Every $k$-step \textsc{cfr} nilspace is a factor of a $k$-step universal \textsc{cfr} nilspace.
\end{theorem}

\noindent As we shall see below, the proof of this theorem will proceed by induction with a sequence of nilspace extensions. This is a good moment to show that the properties of being a ported nilspace or a universal nilspace are stable under nilspace extensions of a very simple type. This is the purpose of the following lemma, which will also be instrumental later.

\begin{lemma}\label{lem:dirproduniv}
Let $\ns$ be a $k$-step ported \textsc{cfr} nilspace, and let $A$ be a compact abelian Lie group. Then $\ns\times\mc{D}_k(A)$ is also ported. Moreover, if $\ns$ is universal, then so is $\ns\times\mc{D}_k(A)$.
\end{lemma}

\begin{proof}
By assumption there is a ported presentation $\ns\cong F/\Gamma$. Let $B$ be a finite abelian group, let $r$ be the rank of $B$, and let $n\geq 0$ such that $A= \mb{T}^n\times B$. Let $F'$ denote the free nilspace $F\times\mc{D}_k(\mb{R}^n\times \mb{Z}^r)$. There is then a natural embedding of $\Gamma$ as a subgroup of $\tran(F')$, namely for each $\gamma\in\Gamma$, its image under this embedding is the translation $\gamma'$ defined by $\gamma'(x,y)=(\gamma x,y)$ for any $x\in F, y\in \mb{R}^n\times \mb{Z}^r$. 

Identifying $\Gamma$ with this embedded image, the idea now is to add new generators, to the generating set for $\Gamma$, to obtain a new porting and free group $\Gamma'$ such that $F'/\Gamma'\cong \ns\times\mc{D}_k(A)$. To define these generators, first for $i\in [r]$ let $b_i\in \mb{N}$ be such that $B\cong \prod_{i\in [r]} \mb{Z}/b_i\mb{Z}$. Then the list of new generators is as follows: for $j\in [n]$ we let $g_j$ be the translation in $\tran(F')$ which adds 1 in the $j$-th component of this new part $\mb{R}^n\times \mb{Z}^r$ that we added to $F$ to obtain $F'$ (note that this $j$-th component is a copy of $\mb{R}$), and for $j\in [n+1,n+r]$ we let $g_j$ be the translation in $\tran(F')$ which adds $b_{j-n}$ in the $j$-th coordinate of said new part (note that this $j$-th component is a copy of $\mb{Z}$). We thus immediately obtain that $\Gamma'$ is still orthogonal, and is in fact a porting subgroup if $\Gamma$ is. Note that the other properties, fiber-transitive, fiber-discrete, and fiber-cocompact follow easily for $\Gamma'$ from the fact that $\Gamma$ satisfies them, and that the translations that we are adding all belong to $\tran_k(F')$. It is also clear that $F'/\Gamma'\cong \ns\times \mc{D}_k(A)$.

To complete the proof, it now suffices to check that if $\Gamma$ is free graded $k$-step nilpotent, then so too is $\Gamma'$ and the filtration of $\Gamma'$ is the one spanned by the generators as a free graded nilpotent group. For this we need to check that a finite word $w$ in this group is trivial if and only if its representative in the free group belongs to the normal subgroup generated by the commutators of weighted degree at least $k+1$. But since the new generators $g_j$ are in the center of $\tran(F')$ (being in $\tran_k(F')$), we can collect them into a word $w_k$ such that $w=w'w_k$ where $w'$ is a word only in elements of $\Gamma$. Moreover, the definition of the embedding of $\Gamma$ and of the generators $g_j$ implies (looking at the coordinates that these different translations act on) that $\Gamma'=\Gamma\times \langle (g_j)_{j\in [n+r]}\rangle$. Hence, we must have $w'=\id$ in $\Gamma$, which implies that $w'$ belongs to the normal subgroup generated by commutators of degree at least $k+1$ (since $\Gamma$ is free graded $k$-step nilpotent), and we must also have $w_k=\id$, so we conclude that $w=w'$ is indeed a commutator of degree at least $k+1$. Finally, note that clearly we have that $\Gamma'_i = \Gamma_i\times \langle (g_j)_{j\in [n+r]}\rangle$ for all $i\in[k]$ and this is precisely the filtration generated as a free graded nilpotent group for the generators of $\Gamma$ and the $g_j$ (which all have degree $k$ by definition).
\end{proof}

We now define the following final important notion for our strategy. 

\begin{defn}[Toral-splitting nilspaces]\label{def:tor-split}
We say that a compact nilspace $\ns$ is \emph{$\ell$-toral-splitting} if for any $t\in[\ell]$ and any $n\in \mb{N}$, any extension of $\ns$ by $\mc{D}_t(\mb{T}^n)$ splits.\footnote{We may talk about \emph{toral splitting nilspaces} when we do not want to specify the parameter $\ell$ for expository reasons or when it is clear from the context.}
\end{defn}

\begin{remark}
Recall the well-known fact, from abelian Lie group theory, that any abelian group extension of an abelian Lie group by a torus splits.\footnote{A short proof of this fact is as follows. Let $A$ be an abelian Lie group and suppose that we have an injective homomorphism $\varphi:\mb{T}^n\to A$. Then, passing to the duals we have a surjective homomorphism $\wh{\varphi}:\wh{A}\to \mb{Z}^n$. Clearly we can define a cross-section $s:\mb{Z}^n\to \wh{A}$ and it is easy to see then that $\wh{A}\cong \mb{Z}^n\times\ker(\wh{\varphi})$. Taking duals again the result follows.} Definition \ref{def:tor-split} captures a similar phenomenon more generally among compact nilspaces.
\end{remark}

Let us now state our main result from this section, which will imply Theorem \ref{thm:realmain} and thus Theorem \ref{thm:cfr-are-factor-of-nilmanifolds}. Note that, to prove Theorem \ref{thm:realmain}, instead of starting with a general \textsc{cfr} nilspace, we can assume (by Theorem \ref{thm:wscover}) that the initial nilspace is already weakly splitting (recall Definition \ref{def:weak-splitting}). The main result is then the following.

\begin{theorem}\label{thm:cosetnilspacecovering}
Let $\ns$ be a $k$-step weakly-splitting nilspace. Then there is a $k$-step universal \textsc{cfr} nilspace $\ns'$ and a fibration $\ns'\to\ns$. Thus $\ns$ is a factor of the filtered nilmanifold $\ns'$.
\end{theorem}

We shall prove this by induction on $k$. The argument relies on two key ingredients, the following two theorems.

\begin{theorem}\label{thm1}
Let $\ns$ be a $k$-step universal \textsc{cfr} nilspace. Then there is a $(k+1)$-step universal \textsc{cfr} nilspace $\nss$ such that $\nss_k\cong\ns$ and such that $\nss$ is a toral extension of $\ns$.
\end{theorem}

\begin{theorem}\label{thm2}
Every $k$-step universal \textsc{cfr} nilspace is $k$-toral splitting.
\end{theorem}
\noindent These two ingredients are combined in the main inductive argument below. The following auxiliary lemma will be used to shorten the main proof.

\begin{lemma}\label{lem:engine}
Let $\nss$ be a degree-$k$ toral extension of a $(k-1)$-step universal \textsc{cfr} nilspace $W$, and let $B$ be a finite abelian group. Then $\nss\times \mc{D}_k(B)$ is a nilspace factor of a $k$-step universal \textsc{cfr} nilspace $\ns'$.
\end{lemma}
The following diagram can help to follow the proof.

\medskip

\[
\begin{tikzcd}
	{\ns:=\nss\times\mc{D}_k(B)} & {\ns':=W'\times\mc{D}_k(\mb{T}^n)\times \mc{D}_k(B)} \\
	\nss & {\nss\times_W W'=W'\times\mc{D}_k(\mb{T}^n)} \\
	W & {W'}
	\arrow[from=1-1, to=2-1]
	\arrow[from=1-2, to=1-1]
	\arrow[from=1-2, to=2-2]
	\arrow[from=2-1, to=3-1]
	\arrow[from=2-2, to=2-1]
	\arrow[from=2-2, to=3-2]
	\arrow[from=3-2, to=3-1]
\end{tikzcd}
\]

\medskip

\begin{proof}
By Theorem \ref{thm1}, the nilspace $W$ is a factor of a $k$-step universal \textsc{cfr} nilspace $W'$. Let $\ns=\nss\times\mc{D}_k(B)$. Consider the fiber-product $\nss\times_W W'$. By \cite[Lemma A.16]{CGSS-p-hom} (applied with $Q=W'$ and $\ns$ the nilspace $\nss$ here), this fiber-product is a degree-$k$ toral extension of $W'$ (by $\mb{T}^n$ say). By Theorem \ref{thm2} we have that $W'$ is $k$-toral splitting, whence $\nss\times_W W'=W'\times \mc{D}_k(\mb{T}^n)$. Let $\ns'=W'\times \mc{D}_k(\mb{T}^n)\times \mc{D}_k(B)\cong W'\times \mc{D}_k(A)$, where $A=\mb{T}^n\times B$ is a compact abelian Lie group.  As $\nss$ is a factor of $\nss\times_W W'$, we have that $\ns$ is a factor of $\ns'$. 

Now it only remains to prove that $\ns'$ is universal. This holds because $\ns'$ is a direct product (in the compact nilspace category) of the $k$-step universal \textsc{cfr} nilspace $W'$ with the degree-$k$ compact abelian Lie group $\mc{D}_k(A)$, and by Lemma \ref{lem:dirproduniv} this is again a universal \textsc{cfr} nilspace. This completes the proof.
\end{proof}

With Lemma \ref{lem:engine}, we can now prove the main result.

\begin{proof}[Proof of Theorem \ref{thm:cosetnilspacecovering}]
The following diagram represents the proof.
\[
\begin{tikzcd}
	{\ns:=\nss'\times\mc{D}_k(B_k)} & {\nss\times \mc{D}_k(B_k)} \\
	{\nss'} & {\nss:=\nss'\times_{\ns_{k-1}}W} \\
	{\ns_{k-1}} & W
	\arrow[from=1-1, to=2-1]
	\arrow[from=1-2, to=1-1]
	\arrow[from=1-2, to=2-2]
	\arrow[from=2-1, to=3-1]
	\arrow[from=2-2, to=2-1]
	\arrow[from=2-2, to=3-2]
	\arrow["{\psi_{k-1}}"', from=3-2, to=3-1]
\end{tikzcd}
\]

\medskip
\noindent We argue by induction on $k$. The case $k=1$ follows easily from the theory of compact abelian Lie groups.

Suppose then that $k>1$ and the theorem holds for $k-1$. The assumed weakly-splitting property of $\ns$ implies that $\ns=\nss'\times \mc{D}_k(B_k)$ where $\nss'$ is a toral extension of $\ns_{k-1}$ and $B_k$ is a finite abelian group. We have that $\ns_{k-1}$ is weakly splitting (by Definition \ref{def:weak-splitting}), so by induction there is a $(k-1)$-step universal \textsc{cfr} nilspace $W$ and a fibration $\psi_{k-1}:W\to\ns_{k-1}$. Now we take the fiber-product $\nss:=\nss' \times_{\ns_{k-1}} W$. Then $\nss'$ is a factor of $\nss$, and $\nss$ satisfies the assumptions in Lemma \ref{lem:engine} by \cite[Proposition A.16]{CGSS-p-hom}. We also have that $\nss\times \mc{D}_k(B_k)$ extends $\ns$. Then, by Lemma \ref{lem:engine} applied to the nilspace $\nss\times \mc{D}_k(B_k)$, we obtain that this nilspace (and hence also $\ns$) is a nilspace factor of a $k$-step universal nilspace, so we are done.
\end{proof}
\begin{remark}[Invariance of the discrete part]
By the proof of Theorem \ref{thm:cosetnilspacecovering}, the following holds for every $i\in[k]$. If $\ns_i\cong \nss'_i\times \mc{D}_i(B_i)$ where $\nss_i'$ is a toral extension of $\ns_{i-1}$, then the extension nilspace $\nss$ satisfies that $\nss_i\cong W_i\times \mc{D}_i(B_i)$ where $W_i$ is a toral extension of $\nss_{i-1}$.
\end{remark}

\noindent Thus, the task for the remainder of this section is to prove Theorems \ref{thm:wscover}, \ref{thm1}, and \ref{thm2}. The outline of the next subsections is as follows. In Subsection \ref{subsec:ws}, we prove Theorem \ref{thm:wscover}. Next, we connect weak splitting with orthogonality in Subsection \ref{subsec:ortho-ws}. Then, Subsection \ref{subsec:LieParam} develops useful machinery to locate the ``correct" (``most economical") group yielding the nilmanifold presentation; we call this machinery the \emph{Lie parametrization}. Finally, in Subsection \ref{subsec:thm1} we prove Theorem \ref{thm1}, and in Subsection \ref{subsec:thm2} we prove Theorem \ref{thm2}.

\subsection{\texorpdfstring{Extending a \textsc{cfr} nilspace to a weakly-splitting nilspace}{Extending a CFR nilspace to a weakly-splitting nilspace}}\label{subsec:ws}\hfill\\
\noindent In this subsection we prove Theorem \ref{thm:wscover}. For this, we need additional notation and a purely group-theoretic lemma on nilpotent groups. 

If $H$ is a subset in a group $G$, then we define $H^i:= \{h^i:h\in H\}$. Note carefully that $H^i$ is not the group generated by the $i$-th powers as usual, rather it is just the \emph{set} of $i$-th powers of elements of $H$.

\begin{lemma}\label{lem:grouppower}
Let $(G,G_\bullet)$ be a filtered group of degree $k$. Let $m_1,m_2,\dots,m_k$ be natural numbers such that, for every $i\in[k-1]$, we have that $m_i/m_{i+1}$ is an integer multiple\footnote{A tighter condition can be deduced from the proof, but for our purposes this one is enough.} of $k!$. Then the following statements hold.
\begin{enumerate}
\item The product set $T=T(G,m_1,m_2,\dots,m_k):=G_1^{m_1}G_2^{m_2}\cdots G_k^{m_k}$ is a group.
\item For every $i\in[k]$ we have $G_1^{m_1}G_2^{m_2}\cdots G_i^{m_i}G_{i+1}\leq G_1^{m_i}G_{i+1}$.
\end{enumerate}
\end{lemma}

\begin{proof} We will use the Hall--Petresco (H--P) formula \cite[End of page 356]{Petresco}, which says that for any $j\in[k]$, every $g,h\in G_j$, and $n\in \mb{N}$, the following holds:
\begin{equation}\label{eq:H-P}
g^nh^n=(gh)^n\prod_{i=2}^kc_i^{\binom{n}{i}},
\end{equation}
where for every $i\in[k]$ we have $c_i\in G_{ij}$. Now assume that we have two elements $a,b$ in $T$. We need to prove that $ab\in T$ and that $a^{-1}\in T$. We will prove a generalization of both: suppose we have a product $a_1a_2\cdots a_n$ where each $a_i$ is in $G_{r_i}^{m_{r_i}}$ for some sequence $r_i\in[k]$. Observe that since each $G_i$ is a normal subgroup, we can interchange terms in this product by the rule $xy=yx^y$ (where $x^y:=y^{-1}xy$). Such a transformation may change an individual element $a_i$, but preserves its property of lying in the set $G_{r_i}^{m_{r_i}}$, since if $a_i=d_{i}^{m_{r_i}}$ then $a_i^y=(d_{i}^y)^{m_{r_i}}\in G_{r_i}^{m_{r_i}}$.

By applying such transpositions we can assume without loss of generality that the sequence $(r_i)$ is nondecreasing. The main idea is that if $r_i=r_{i+1}=s$ for some $i$ then we can simplify the product using \eqref{eq:H-P} so that the set $\{j:r_j\leq s\}$ becomes smaller. By iterating this transformation we can eliminate multiplicities in the sequence $r_i$, which then proves that any such product is in the set $T$. To make this work, we need to show that the higher-order terms appearing in \eqref{eq:H-P} lie in the correct powers of a group. To see this, note that if $a_i=d_i^{m_s}$ and $a_{i+1}=d_{i+1}^{m_s}$ then $a_ia_{i+1}=(d_id_{i+1})^{m_s}\prod_{j=2}^kc_j^{\binom{m_s}{j}}$. Hence, as $c_j\in G_{sj}$, we need to prove that $\binom{m_s}{j}$ is a multiple of $m_{js}$ for all $j\in[2,k]$. But this clearly follows from our assumption on the sequence $m_1,m_2,\ldots,m_k$. This proves statement $(i)$.

Statement $(ii)$ also follows from \eqref{eq:H-P}, but in a different way. We prove it by a recursive process. Assume that we have an expression $g_1^{m_1}g_2^{m_2}\dots g_i^{m_i}$ where $g_j\in G_j$ for $1\leq j\leq i$. Our goal is to reduce this expression, step by step, to a $m_i$-th power modulo $G_{i+1}$. The first step already illustrates the process. Observe that by \eqref{eq:H-P} we have
\[g_1^{m_1}g_2^{m_2}=(g_1^{d_1})^{m_2}g_2^{m_2}=(g_1^{d_1}g_2)^{m_2}\prod_{t=2}^k c_t^{\binom{m_2}{t}}\] where $d_1=m_1/m_2$. Observe that $c_2=[g_1^{d_1},g_2]\in G_3$ and since $m_3$ divides ${\binom{m_2}{2}}$, we have that the term $c_2^{\binom{m_2}{2}}$ is in $G_3^{m_3}$. Since $m_3$ divides ${\binom{m_2}{t}}$ for all $1\leq t\leq k$ we have in somewhat wasteful way (not using that higher commutators are in an even better subgroup) that all the commutator terms in \eqref{eq:H-P} are in $G_3^{m_3}$. Thus using a similar reduction as in the first part of the proof we can collapse these terms and reduce our initial expression to $h_1^{m_2}h_3^{m_3}\dots h_i^{m_i}$ modulo $G_{i+1}$ where $h_1\in G_1,h_3\in G_3,h_4\in G_4,\dots$. We can then continue the process with decomposing $h_1^{m_2}h_3^{m_3}$ in a similar way using \eqref{eq:H-P}, which works since in particular we have $[h_1,h_3]\in G_4$. At the $j$-th step, we have an expression of the form $h_1^{m_j}h_{j+1}^{m_{j+1}}h_{j+2}^{m_{j+2}}\dots$. Once we reach the desired $i\in[k]$, the result follows.
\end{proof}

Now we apply this lemma to fiber-transitive groups $G$ of translations on free nilspaces, and deduce that the corresponding group $T$ is also fiber-transitive.

\begin{theorem}\label{thm:power-fib-tran} Let $F$ be a $k$-step free nilspace and let $G$ be a fiber-transitive subgroup of $\tran(F)$. Let $G_i:=G\cap\tran_i(F)$. Let $m_1,m_2,\dots,m_k$ be natural numbers such that $m_i/m_{i+1}$ is an integer multiple of $k!$. Then the group $T:=T(G,m_1,m_2,\dots,m_k)$ is also fiber-transitive. 
\end{theorem}
\begin{proof}
We need to prove that for all $j\in[k]$ the following statement holds
\begin{equation}\label{eq:T-fib-tran}
\forall x,y\in F \text{ if } x\sim_j y \text{ and }\exists t\in T\text{ such that } t(x)=y \implies \exists t^*\in T_{j+1} \text{ with } t^*(x)=y.
\end{equation}
We prove this by reverse induction on $j$. Note that, for $j=k$, the result holds trivially as $x=y$ and we may let $t^*=\id$. Assume now that \eqref{eq:T-fib-tran} holds for $j=i+1$ and let $x,y\in F$ and $t\in T$ be such that $x\sim_i y$ and $t(x)=y$.

Notice that $t= hg^{m_{i+1}}$ holds for some $g\in G, h\in G_{i+2}$, by the second statement of Lemma \ref{lem:grouppower}. Hence $g^{m_{i+1}}(x)\sim_{i+1}y$ and so (combined with $y\sim_i x$) we have  $g^{m_{i+1}}(x)\sim_i x$. We claim that this last relation implies that $g(x)\sim_i x$. We  prove this by induction using \cite[Theorem 3.15]{CGSS-doucos}. Indeed, note that $g^{m_{i+1}}(x)\sim_1 x$ but then, letting $s_1\in \ab_1$ be such that $\eta_1(g)(\pi_1(x))=\pi_1(x)+s_1$, we have $\eta_1(g^{m_{i+1}})(\pi_1(x))=\pi_1(x)+m_{i+1}s_1$. Then $g^{m_{i+1}}(x)\sim_1 x$ implies that $m_{i+1}s_1=0$. Since $F$ is a free nilspace, we have $\pi_1(F)=\mc{D}_1(\mb{Z}^{a_1}\times\mb{R}^{b_1})$, so $m_{i+1}s_1=0$ implies that $s_1=0$. This proves the case $i=1$. Now assume by induction that $g(x)\sim_j x$ for $j\in[i-1]$. By \cite[Theorem 3.15]{CGSS-doucos}, the $j+1$-th coordinate of the translation $g$ adds a polynomial $s_{j+1}(\pi_j(x))$ to the $\mc{D}_{j+1}$ factor of $F$. By assumption we have $\pi_{j}(g(x))=\pi_{j}(x)$, so $\pi_{j+1}(g^{m_{i+1}}(x))=\pi_{j+1}(x)+m_{i+1}s_{j+1}(\pi_{j}(x))$. However, by assumption we also had $\pi_{j+1}(g^{m_{i+1}}(x))=\pi_{j+1}(x)$. Therefore $s_{j+1}(\pi_{j}(x))=0$ and hence $g(x)\sim_{j+1}x$. Our claim $g(x)\sim_i x$ then follows.

The above claim together with the fiber-transitive property of $G$ implies that there is $g_*\in G_{i+1}$ such that $g_*(x)=g(x)$. Thus, we have that $g$ and $g_*$ act in the same way on the $\sim_i$ class of $x$ in $\pi_{i+1}(F)$. To see this, similarly as before, note that $\pi_i(g(x))=\pi_i(x)$ and hence, by \cite[Theorem 3.15]{CGSS-doucos}, we must have that $\pi_{i+1}(g(x))=\pi_{i+1}(x)+s_{i+1}(\pi_i(x))$. Thus $\eta_{i+1}(g_*)$ must be precisely the addition of $s_{i+1}(\pi_i(x))$, i.e.\ $\eta_{i+1}(g_*)(\pi_{i+1}(x'))=\pi_{i+1}(x')+s_{i+1}(\pi_i(x))$ for any $x'\in F$. In particular ${g_*}^{m_{i+1}}(x)\sim_{i+1}g^{m_{i+1}}(x)$. We claim now that $g_*^{m_{i+1}}(x)$ and $y$ satisfy
\[
g_*^{m_{i+1}}(x) \sim_{i+1}y \text{ and } \exists t'\in T\text{ such that }t'(g_*^{m_{i+1}}(x)) = y.
\]
If we manage to prove this then the result will follow by induction, as we may then apply \eqref{eq:T-fib-tran} with $j=i+1$, the pair $g_*^{m_{i+1}}(x)$ and $y$, and $t'\in T$.

To prove the first of these statements, note that $y=t(x)=hg^{m_{i+1}}(x)\sim_{i+1} hg_*^{m_{i+1}}(x)$. But as $h\in G_{i+2}$, we have that the latter is $\sim_{i+1}g_*^{m_{i+1}}(x)$. To prove the second statement, note that $y=hg^{m_{i+1}}g_*^{-m_{i+1}}g_*^{m_{i+1}}(x)$ and both $t=hg^{m_{i+1}}$ and $g_*^{m_{i+1}}$ are in $T$. The result follows.
\end{proof}

\begin{corollary}\label{cor:weak-split-factor-with-subgroup}
Under the same assumptions as in Theorem \ref{thm:power-fib-tran}, if $G$ is fiber-discrete, then so is $T$. If, in addition, $G$ is fiber-cocompact, then so is $T$. Moreover, the map $F/T\to F/G$ given by $\pi_T(x)\mapsto \pi_G(x)$ defines a fibration.
\end{corollary}

\begin{proof}
Note that it is enough to prove the result for the last structure group and then proceed by induction projecting onto the lower step factors. We need to compute $T\cap G_k$. Note that clearly $ T\cap G_k\le G_k\le \ab_k(F)=\mb{Z}^{a_k}\times\mb{R}^{b_k}$ where $G_k$ is a discrete (resp. cocompact) lattice. If $G_k$ is discrete, then $T\cap G_k$ is discrete in $\ab_k(F)$. Moreover, if $G_k$ is cocompact, choose $z_1,\ldots,z_{a_k+b_k}\in G_k$ generating a cocompact lattice in $\ab_k(F)\cong \mathbb{Z}^{a_k}\times \mb{R}^{b_k}$. Their $m_k$-th powers generate a cocompact lattice contained in $ T\cap G_k$, so $T\cap G_k$ is cocompact as well. The fact that $T\le G$ clearly implies that the map $\pi_T(x)\mapsto \pi_G(x)$ is well-defined. It follows from \cite[proof of Lemma 4.1]{CGSS} that as both $\pi_T$ and $\pi_G$ are fibrations then so is $\pi_T(x)\mapsto \pi_G(x)$.
\end{proof}

For the proof of Theorem \ref{thm:wscover} we need the next lemmas. Recall that a translation $\alpha$ on a nilspace $\ns$ is said to be \emph{consistent} with a map $p:\ns\to\nss$ if for any $x,y\in \ns$, if $p(x)=p(y)$ then $p(\alpha(x))=p(\alpha(y))$ (in other words $\alpha$ preserves the fibers of $p$); see \cite[Definition 1.2]{CGSS}.

\begin{lemma}\label{lem:Dpartproj}
Every translation $\alpha\in \tran(F)$ is consistent with $p_D$ and thus the map $\wh{p_D}:\tran(F)\to\tran(D)$ is a well-defined homomorphism.    
\end{lemma}

\begin{proof}
This is a straightforward consequence of the continuity of $\alpha$, since $p_D(x)=p_D(y)$ implies that $x,y$ are in the same connected component of $F$.
\footnote{Alternatively, by \cite[Lemma 3.5]{CGSS-doucos},
for every $\alpha\in\tran(F)$, the discrete-coordinate
functions in the expression of $\alpha$ given by
\cite[Theorem 3.15]{CGSS-doucos} involve no continuous
coordinate from $F$. Thus $p_D(\alpha(x))$ depends only
on $p_D(x)$.}
\end{proof}

\begin{lemma}\label{highpower} 
Let $F$ be a free nilspace and let $p_D$ be the projection to the discrete part on $F$. Then for every $n\in \mb{N}$ there exists $f(n)\in \mb{N}$ such that if $g\in\tran(F)$ and $x\in F$, then $p_D(g^{f(n)}(x))$ is congruent to $p_D(x)$ modulo $n$ in every coordinate.
\end{lemma}

\begin{proof}
Since the conclusion can be written as
$\wh{p_D}(g)^{f(n)}(p_D(x))\equiv p_D(x)$ mod $n$ and thus concerns
just the discrete part of $F$, we can assume without loss of
generality that $F$ equals its discrete part $D$. Now by
\cite[Theorem 3.15]{CGSS-doucos} and
\cite[Lemma 3.5]{CGSS-doucos}, every translation $g\in\tran(D)$
is described by integer-valued polynomials with integer
coefficients in the binomial basis, and it follows that
$\tran(D)$ is a finitely generated torsion-free nilpotent group.
By a well-known theorem of P.\ Hall
(see \cite[Theorem 6.5]{Clement&al}), we have that $\tran(D)$
can be identified with a group of unipotent upper-triangular
$r\times r$ integer matrices, i.e.\ matrices of the form $I_r+M$
where $M\in\mb{Z}^{r\times r}$ is strictly upper triangular.
This identification can be chosen so that the matrix $A_g$
corresponding to $g$ satisfies $B(g(x))=A_gB(x)$, where $B$
is a fixed column vector of integer-valued polynomials
containing every coordinate function.\footnote{Write
$D=\prod_{a=1}^d\mc{D}_{w_a}(\mb{Z})$, with $w_a\le k$.
Take as the entries of $B$ all binomial monomials
$\prod_{a=1}^d\binom{x_a}{e_a}$ with
$\sum_{a=1}^d w_ae_a\le k$, ordered by decreasing weighted
degree. Substitution by a translation preserves
integer-valuedness and changes each such monomial only by
terms of strictly lower weighted degree. Its expansion in
this binomial basis therefore gives an upper-unitriangular
integer matrix $A_g$. Moreover, $A_{gh}=A_gA_h$, and the
representation is faithful because $B$ contains every
coordinate function.}
Note now that
$(I_r+M)^m=\sum_{i=0}^m M^i\binom{m}{i}$ and that $M^r=0$.
Taking $m=nr!$, the binomial coefficients $\binom{m}{i}$
are divisible by $n$ for every $1\le i<r$, so
$A_g^m\equiv I_r\pmod n$. Hence
$B(g^m(x))=A_g^mB(x)\equiv B(x)\pmod n$.
Since $B$ contains every coordinate function, the result
follows with $f(n)=nr!$.
\end{proof}

The last tool we need for the proof of Theorem \ref{thm:wscover} is the following.

\begin{proposition}\label{prop:cond-for-equivalence}
Let $\ns$ be a $k$-step nilspace and $H^1_\bullet,H^2_\bullet$ be filtered subgroups of $ \tran(\ns)$. Suppose that the following holds.
\begin{enumerate}
    \item\label{it:equ-fil-simplified-1} $H^1$ is fiber-transitive,
    \item\label{it:equ-fil-simplified-2} $\forall i\in[k]$, $\forall h\in H^1_i$, and $\forall x\in \ns$ there exists $h'\in H^2_i$ such that $h(x)=h'(x)$, and
    \item\label{it:equ-fil-simplified-3} $\forall h'\in H^2$ and $\forall x\in \ns$ there exists $h\in H^1$ such that $h(x)=h'(x)$.
\end{enumerate}
Then $H^1_\bullet$ and $H^2_\bullet$ are equivalent (see \cite[Definition 5.15]{CGSS-doucos}).
\end{proposition}

\begin{proof}
First, let us prove that $H^2_\bullet$ is also fiber-transitive. Let $x,y\in\ns$, $i\in[k]$, and $h'\in H^2$ be such that $h'(x)=y$ and $\pi_i(x)=\pi_i(y)$. By \eqref{it:equ-fil-simplified-3}, there exists $h\in H^1$ such that $h(x)=h'(x)$. By \eqref{it:equ-fil-simplified-1}, there exists $h^*\in H^1_{i+1}$ such that $h^*(x)=y$. By \eqref{it:equ-fil-simplified-2}, let $\tilde{h}\in H^2_{i+1}$ be such that $\tilde{h}(x)=h^*(x)=h(x)=h'(x)=y$. Hence $H^2_\bullet$ is also fiber-transitive.

We now claim that \eqref{it:equ-fil-simplified-2} holds but with the roles of $H^1$ and $H^2$ swapped. That is, let $i\in[k]$, $h'\in H^2_i$, and $x\in \ns$. Then, clearly $\pi_{i-1}(x)=\pi_{i-1}(h'(x))$. By \eqref{it:equ-fil-simplified-3}, let $h\in H^1$ be such that $h(x)=h'(x)$. By \eqref{it:equ-fil-simplified-1}, there exists $h^*\in H^1_{i}$ such that $h^*(x)=h(x)=h'(x)$ and the claim follows.

As both $H^1$ and $H^2$ are fiber-transitive, by \cite[Lemma 5.7]{CGSS-doucos}, both $\pi_{H^1}(\ns)$ and $\pi_{H^2}(\ns)$ are $k$-step nilspaces. Let us now define the map $\pi_{H^1}(\ns) \to \pi_{H^2}(\ns)$ as $\pi_{H^1}(x)\mapsto \pi_{H^2}(x)$. We claim that this map is a nilspace isomorphism. First, note that by \eqref{it:equ-fil-simplified-2} the map is well-defined. Moreover, it is a morphism because for any $\q\in \cu^n(\ns)$ we have that $\pi_{H^1}\co\q\mapsto\pi_{H^2}\co \q$ which is clearly in $\cu^n(\pi_{H^2}(\ns))$. Moreover, letting $d\in \cu^n(H^1)$ we have that $\pi_{H^1}(d\cdot \q)=\pi_{H^1}(\q)\mapsto \pi_{H^2}(\q)=\pi_{H^2}(d\cdot \q)$. By \cite[Lemma 5.13]{CGSS-doucos} there exists $d'\in \cu^n(H^2)$ such that $d'\cdot\q=d\cdot \q$. Swapping the roles of $H^1$ and $H^2$ (and using that \eqref{it:equ-fil-simplified-2} also works with those roles swapped) it follows that $H^1_\bullet$ and $H^2_\bullet$ are equivalent filtrations.
\end{proof}

\begin{remark}
It may be tempting to relax \eqref{it:equ-fil-simplified-2} to simply \eqref{it:equ-fil-simplified-3} with the roles changed. That is, that given two filtrations $H^1_\bullet$, $H^2_\bullet$, we have that $H^1_\bullet$ is fiber-transitive and for any $j\in\{1,2\}$, $h\in H^j$, and $x\in\ns$ there exists $h'\in H^{3-j}$ such that $h(x)=h'(x)$. However, in this situation $H^2_\bullet$ may fail to be fiber-transitive. As an example, let $F=\prod_{i=1}^3\mc{D}_i(\mb{Z})$, let $H^1$ be generated by $\{(x,y,z)\mapsto (x,y+2,z),(x,y,z)\mapsto(x,y,z+2)\}$ and $H^2$ be generated by $\{(x,y,z)\mapsto (x,y+2,z+2x^2),(x,y,z)\mapsto(x,y,z+2)\}$. In this case, it is easy to check that $H^2_\bullet$ is not fiber-transitive but satisfies the aforementioned conditions.
\end{remark}

\begin{proof}[Proof of Theorem \ref{thm:wscover}] We argue by induction on the step $k$. The case $k=1$ is trivial. Hence, let $k>1$, let $\ns$ be a $k$-step \textsc{cfr} nilspace, and assume that Theorem \ref{thm:wscover} holds up to step $k-1$.

By Theorem \ref{thm:cfr-nil}, we have $\ns=F/G$ for some free $k$-step nilspace $F$ and some translation fiber-cocompact lattice $G\leq\tran(F)$. Then we make the following claim:
\begin{align}\label{claim}
\exists\, m_1,\dots,m_k\in \mb{N}\text{ s. t. }T=T(G,m_1,\ldots,m_k)\leq G\text{ is a fiber-cocompact translation }\\
\text{lattice such that }F/T\text{ is a nilspace extension of }F/G\text{ with }F/T\cong \ns'\times \mc{D}_k(B)\nonumber \\
\text{ where $\ns'$ is a toral extension of $\ns'_{k-1}$ and $B$ is a finite abelian group}. \nonumber
\end{align}
To prove this claim, let $m_{i}=m (k!)^{k-i}$ for $i<k$ and $m_{k}=1$ for a sufficiently composed integer $m$. Let $F=\prod_{i=1}^k\mc{D}_i(\mb{Z}^{a_i}\times\mb{R}^{b_i})$ and let us assume, without loss of generality, that $G_k\le \mb{Z}^{a_k}\times \mb{R}^{b_k}$ is of the form $G_k=\wh{p_D}(G_k)\times \mb{Z}^{b_k} \le \mb{Z}^{a_k}\times \mb{R}^{b_k}$.\footnote{To prove this, we can see that it follows from the proof of \cite[Theorem 6.2]{CGSS-doucos} or by the following argument. Since $G_k\cap( \{0^{a_k}\}\times \mb{R}^{b_k})$ is cocompact in $\{0^{a_k}\}\times \mb{R}^{b_k}$, there exists an isomorphism of $ \mb{R}^{b_k}$ such that $G_k\cap( \{0^{a_k}\}\times \mb{R}^{b_k})\cong\{0^{a_k}\}\times \mb{Z}^{b_k}$. Then we have the short exact sequence $0\to \{0^{a_k}\}\times \mb{Z}^{b_k} \to G_k\to \wh{p_D}(G_k)\to 0$. As $\wh{p_D}(G_k)$ is a free abelian group this sequence splits, so we can assume that $G_k=\langle \{0^{a_k}\}\times \mb{Z}^{b_k} \cup \{g_1,\ldots,g_{a_k}\}\rangle $. Each $g_i=(u_i,v_i)\in \mb{Z}^{a_k}\times\mb{R}^{b_k}$ but note that, by choosing an appropriate linear map $B:\mb{R}^{a_k}\to \mb{R}^{b_k}$ and letting $T:\mb{Z}^{a_k}\times\mb{R}^{b_k}\to \mb{Z}^{a_k}\times\mb{R}^{b_k}$ be $(x,y)\mapsto (x,y+Bx)$, we can eliminate the $v_i$ components and then, modulo composition with the previous isomorphisms, we have $G_k=\wh{p_D}(G_k)\times \mb{Z}^{b_k}$.}

Then, the choice of $m$ is given by Lemma \ref{highpower} as $m=f(d)$ where $d:=|\mb{Z}^{a_k}/\wh{p_D}(G_k)|$. To see that this works,  let $\underline{x}:=(x_i,x'_i)_{i\in[k]}$ for $x_i\in \mb{Z}^{a_i}$ and $x_i'\in \mb{R}^{b_i}$ be any element of $F$. Note that, for every $g\in \tran(F)$, we have a (unique) expression, given by \cite[Theorem 3.15]{CGSS-doucos}, of the form 
\begin{equation}\label{eq:expre-free-trans}
g(\underline{x})=\big(x_1+s_1,x_1'+s_1',\ldots,x_k+s_k((x_i)_{i\in[k-1]}),x_k'+s_k'((x_i,x_i')_{i\in[k-1]})\big).\end{equation} 
Let us now focus on the $x_k$ coordinate.

By our choice of $m_{1},\ldots,m_{k}$ and Lemma \ref{highpower}, if $g\in G$ and $i<k$ we have that the $x_k$ coordinate of $p_D(g^{m_{i}}(x))$ is congruent modulo $d$ with $x_k$ for any $\underline{x}\in F$. Hence, for $g^{m_i}$, its corresponding function $s_k$ (evaluated at any point) is always a multiple of $d$. Recall that, by the definition of $T$ (Lemma \ref{lem:grouppower}), any element $t\in T$ has an expression of the form $t=\prod_{i=1}^{k}g_i^{m_i}$ where $g_i\in G_i$. Thus, the polynomial $s_k$ corresponding to $tg_k^{-1}=\prod_{i=1}^{k-1}g_i^{m_i}$ is always a multiple of $d$ (this is simply because this polynomial $s_k$ in \eqref{eq:expre-free-trans} will correspond to the sum of the $s_k$ polynomials corresponding to each of the $g_i^{m_i}$ evaluated at various points and all of those are multiples of $d$). As the last structure group of $F/G$ is precisely $(\mb{Z}^{a_k}\times\mb{R}^{b_k})/G_k$ and its discrete part $\mb{Z}^{a_k}/\wh{p_D}(G_k)$ has order $d$ by definition, we have that the lattice $d\mb{Z}^{a_k}$ is a subgroup of $ \wh{p_D}(G_k)$. Therefore, for any $t=(\prod_{i=1}^{k-1}g_i^{m_i})g_k$ its corresponding polynomial $s_k$ takes values in $\wh{p_D}(G_k)$.

We now define a new group $T'\le \tran(F)$ as follows. For each $t\in T$ consider its corresponding expression given by \eqref{eq:expre-free-trans} and define a new $\kappa_t\in \tran(F)$ by changing $s_k$ of $t$ by 0 and leaving the rest invariant. Then we let $T':=\langle\{\kappa_t\}_{t\in T}\cup \wh{p_D}(G_k)\rangle$. As $T$ is a subgroup of $\tran(F)$, it is easy to check that $T'$ is as well. Indeed, as in the expression \eqref{eq:expre-free-trans} the $x_k$ term does not appear in any of the $(s_i,s_i')_{i\in[k]}$ we clearly have that $\kappa_t\kappa_{t'}=\kappa_{tt'}$. On the other hand, the group $\wh{p_D}(G_k)$ (seen as a subgroup of $\tran(F)$ in the obvious way) is in the center of $\tran(F)$ so clearly $T'=\{\kappa_tz:t\in T,z\in \wh{p_D}(G_k)\}$ with multiplication given by $\kappa_tz\kappa_{t'}z'=\kappa_{tt'}zz'$. Moreover, we want to apply Proposition \ref{prop:cond-for-equivalence} to $H^1=T$ and $H^2=T'$. To see that we can do this, note first  that condition \eqref{it:equ-fil-simplified-1} in that proposition holds because $T$ is fiber-transitive. To see that condition \eqref{it:equ-fil-simplified-2} holds, note that by definition of $\kappa_t$, if $t\in T_i$ then $\kappa_t\in T'_i=T'\cap \tran_i(F)$. Moreover, given any $x\in F$, we have $t(x)=\kappa_t(x)+(p_x,0)$, where $p_x=s_t(x)\in \wh{p_D}(G_k)$. By the equality $G_k=\widehat p_D(G_k)\times\mb{Z}^{b_k}$ established at the beginning of this proof, the constant translation by $(p_x,0)$ belongs to $G_k$. 
Then, the map $u\mapsto \kappa_t(u)+(p_x,0)$ is clearly in $T'_i$. The third condition \eqref{it:equ-fil-simplified-3} follows similarly. Indeed, put $P:=\wh{p_D}(G_k)$, and for $p\in P$ let $\tau_p$ denote the constant translation by  $(p,0)$ in the top-degree coordinates. Every $h'\in T'$ has the form $h'=\kappa_t\tau_p$ for some $t\in T$ and $p\in P$. Given $x\in F$, write $t(x)=\kappa_t(x)+(p_x,0)$, where $p_x\in P$. Since these constant translations are central, we have $h'(x)=\tau_p\kappa_t(x)=\tau_{p-p_x}t(x)$. Since $G_k=P\times\mb Z^{b_k}$, we have $\tau_{p-p_x}\in G_k\subseteq T$. Thus $h:=\tau_{p-p_x}t\in T$ satisfies $h(x)=h'(x)$, proving \eqref{it:equ-fil-simplified-3}.

Therefore $T'$ is fiber-transitive and $T$ and $T'$ are equivalent filtrations. Hence, by \cite[Lemma 5.18]{CGSS-doucos} we have that $F/T\cong F/T'$. But now note that $F/T'$ is clearly of the form $\ns'\times\mc{D}_k(B)$ because $T'\cong \{\kappa_t\}_{t\in T}\times \wh{p_D}(G_k)$ and this group acts on $F'\times\mc{D}_k(\mb{Z}^{a_k})$ where $F':=\big(\prod_{i=1}^{k-1}\mc{D}_i(\mb{Z}^{a_i}\times\mb{R}^{b_i})\big)\times\mc{D}_k(\mb{R}^{b_k})$ independently on each coordinate. In particular $F/T'\cong (F'/\{\kappa_t\}_{t\in T})\times \mc{D}_k(\mb{Z}^{a_k}/\wh{p_D}(G_k))$. This proves claim \eqref{claim}.

By Corollary \ref{cor:weak-split-factor-with-subgroup}, we have a fibration $\ns'\times \mc{D}_k(B)\to \ns$ where $\ns'\cong F'/\{\kappa_t\}_{t\in T}$ and $B=\mb{Z}^{a_k}/\wh{p_D}(G_k)$. Note that $\ns'$ is a toral extension of $\ns'_{k-1}$. By the induction hypothesis, let $\nss'$ be a $k-1$-step weakly-splitting nilspace such that there exists a fibration $\varphi:\nss'\to \ns'_{k-1}$. Thus, we may consider the fiber-product $\ns'\times_{\ns'_{k-1}} \nss'$ (see \cite[Lemma 4.2]{CGSS}). By \cite[Proposition A.16]{CGSS-p-hom}, we have that $\ns'\times_{\ns'_{k-1}} \nss'$ is a degree-$k$ extension of $\nss'$ by the last structure group of $\ns'$, which is a torus and its $k-1$-factor is $\nss'$. Thus $\ns'\times_{\ns'_{k-1}} \nss'$ is weakly-splitting and there exists a fibration $\ns'\times_{\ns'_{k-1}} \nss'\to \ns'$. Letting $\nss:=(\ns'\times_{\ns'_{k-1}} \nss')\times\mc{D}_k(B)$ we have found a weakly-splitting nilspace and a fibration $\nss\to \ns'\times \mc{D}_k(B)\to\ns$.\end{proof}

\begin{remark}
We leave it as an exercise to check that another route to proving Theorem \ref{thm:wscover} consists in constructing iteratively sequences $m_{i,j}$ for each $i\in[k]$ and $j\in[k-i+1]$. The first step is to set $m_{1,j}:=m_{j}$ as in the proof above, and then construct a second sequence $m_{2,1},\ldots,m_{2,k-1}$ that attains the weakly-splitting property in the $(k-1)$-th structure group, and so on. The final sequence is then $m_i':=\prod_{j=1}^{k-i+1} m_{j,i}$ for $i=1,\ldots,k$, and the weakly-splitting nilspace $\nss$ is given by $F/T$, where $T=T(G,m_1',\ldots,m_k')$. We omit the details.
\end{remark}

\begin{remark}
Note that 2-step weakly-splitting nilspaces are always coset nilspaces. This follows from the fact that abelian toral extensions of abelian Lie groups always split, combined with \cite[Proposition 3.3.39]{Cand:Notes1}.
\end{remark}

\begin{question}
Is every weakly-splitting nilspace a coset nilspace?  
\end{question}
\noindent A positive answer to this question would provide a more direct way to prove Theorem \ref{thm:cfr-are-factor-of-nilmanifolds} than the arguments we use in what follows.

\subsection{Weak splitting as orthogonality}\label{subsec:ortho-ws}\hfill\\
\noindent In this subsection we shall establish that the orthogonality property from Definition \ref{def:ortlattice} is essentially equivalent to weak splitting.

First we prove that the orthogonality property is compatible with taking factors.

\begin{lemma}\label{lem:orth-and-factors}
Let $\ns\cong F/\Gamma$ where $\Gamma \le \tran(F)$ is an orthogonal translation lattice. Then, for every $i\in[k]$, the group $\eta_i(\Gamma)\le \tran(F_i)$ is also an orthogonal translation lattice. 
\end{lemma}

\begin{proof}
It suffices to prove the result for $i=k-1$ and then proceed by induction. The part of $\eta_i(\Gamma)$ being a translation lattice follows from \cite{CGSS-doucos}. Clearly, if $(p_D)_{k-1}:F_{k-1}\to D_{k-1}$ is the factor morphism of $p_D$, then $\wh{(p_D)_{k-1}}(\eta_{k-1}(\Gamma)) = \wh{(p_D)_{k-1}\co \pi_{k-1}}(\Gamma) = \wh{\pi_{k-1}\co p_D}(\Gamma) = \eta_{k-1}(\wh{p_D}(\Gamma))$ and as $\eta_{k-1}:\tran(D)\to \tran(D_{k-1})$ is simply removing the action on the $\mc{D}_k$ factor, we clearly have that $\wh{(p_D)_{k-1}}(\eta_{k-1}(\Gamma))$ is also a lattice generated by \emph{elementary generators}. Next, we have to prove that $\big(\wh{(p_D)_{k-1}}(\eta_{k-1}(\Gamma))\big)_j = \wh{(p_D)_{k-1}}(\eta_{k-1}(\Gamma)_j)$ for all $j\in[k-1]$. Clearly the right-hand side (RHS) is included in the left-hand side (LHS), so it suffices to prove the opposite inclusion. Given $\alpha\in \text{LHS}$ note that this is just a shift in the coordinates of $D_{k-1}$ which equals zero for factors $\mc{D}_s$ for $s<j$. Thus, if we let $\alpha'\in \tran(D)$ be the translation such that $(x,y)\in D_{k-1}\times \mc{D}_k(\ab_k(D))\mapsto (\alpha(x),y)$ this is an element of $(\wh{p_D}(\Gamma))_j$. As the latter group equals $\wh{p_D}(\Gamma_j)$ we may choose $\gamma\in \Gamma_j$ such that $\alpha= \wh{(p_D)_{k-1}}(\eta_{k-1}(\gamma))$ and the result follows.
\end{proof}

As a step towards the announced connection between weak splitting with orthogonality, we establish the following equivalent formulation of weak splitting.
\begin{proposition}\label{prop:weak-ortho-eq}
Let $\ns$ be a $k$-step \textsc{cfr} nilspace. The following properties are equivalent.
\setlength{\leftmargini}{0.8cm}
\begin{enumerate}
    \item\label{it:wsdef} The nilspace $\ns$ is weakly-splitting.
    \item\label{it:wsequiv} For every $i\in [k]$, letting $B_i$ be the finite abelian group such that $\ab_i(\ns)=\mb{T}^{n_i}\times B_i$, there is a (continuous) fibration $q_i:\ns_i\to \mc{D}_i(B_i)$.
\end{enumerate}
\end{proposition}

\begin{remark}
Property \eqref{it:wsequiv} is  equivalent to the existence of a fibration $\ns\to \prod_{i=1}^k \mc{D}_i(B_i)$.
\end{remark}

\begin{proof} To prove that  \eqref{it:wsdef}  implies \eqref{it:wsequiv}, if $\ns_i=\nss_i\times\mc{D}_i(B_i)$ then we can take $q_i$ simply to be the projection to the second component $\nss_i\times\mc{D}_i(B_i)\to \mc{D}_i(B_i)$. 

To see the converse, let $\varphi_i$ denote the map $\ns_i\to\ns_i/B_i$ which takes the quotient by $B_i$ in each $\pi_{i-1}$-fiber. By \cite[Proposition A.19]{CGSS-p-hom}, this map is a fibration making $\ns_i$ a degree $i$ extension of $\ns_i/B_i$. Now we define
\[
\psi_i: \ns_i\to (\ns_i/B_i)\times B_i, \; x\mapsto (\varphi_i(x), q_i(x)).
\]
Since $\varphi_i$ and $q_i$ are both morphisms, we have that $\psi_i$ is a nilspace morphism. We claim that in fact $\psi_i$ is a nilspace isomorphism. To prove this, we claim that
\begin{equation}\label{eq:wsequivkey}
\forall n\geq 0, \forall\q\in \cu^n(\ns_i), q_i^{\db{n}}\text{ restricts to a bijection }\q+\cu^n\!\big(\mc{D}_i(\{0\}\!\times\! B_i)\big)\to \cu^n\!\big(\mc{D}_i(B_i)\big).
\end{equation}
To prove this, it suffices to show that the restriction in question is surjective (since the two sets in \eqref{eq:wsequivkey} have equal cardinality). To see the surjectivity, let $z=q_i^{\db{n}}(\q)\in \cu^n\!\big(\mc{D}_i(B_i)\big)$ and fix any $t\in \cu^n\!\big(\mc{D}_i(B_i)\big)$. Then by the fiber-surjectivity of $q_i$, there is $(\tau,b)\in \cu^n(\mc{D}_i(\mb{T}^{n_i}\times B_i))$ such that $q_i^{\db{n}}(\q+(\tau,b))=z+t$. Now the problem is that $\tau$ may not be the 0-cube. However, the continuity of $q_i$ implies continuity of $q_i^{\db{n}}$, hence preservation of connectedness, so that all of $\q+(\cu^n(\mc{D}_i(\mb{T}^{n_i}))\times\{b\})$ is mapped to $z+t$ (not just $\q+(\tau,b)$), whence in particular $q_i^{\db{n}}(\q+(0,b))=z+t$. As $t$ was arbitrary, this proves the surjectivity, and  \eqref{eq:wsequivkey} follows.

Now, from the case $n=0$ of \eqref{eq:wsequivkey} we obtain that $\psi_i$ is bijective, and from the general case of \eqref{eq:wsequivkey} we obtain that $\psi_i$ is cube-surjective (in fact cube-bijective), and it follows that the inverse map $\psi_i^{-1}$ is also a morphism, so  $\psi_i$ is a nilspace isomorphism.

The desired weak splitting at level $i$ now follows upon setting $\nss_i=(\ns_i/B_i)$, which has $i$-th structure group $\mb{T}^{n_i}$ by \cite[Proposition A.19]{CGSS-p-hom}. \end{proof}

We shall use the following fact about extensions to relate weak splitting and orthogonality.

\begin{proposition}[Extensions seen via quotients of free nilspaces]\label{prop:ext-as-quo-of-free}
Let $\ns$ be a Lie-fibered $k$-step nilspace isomorphic to $F/\Gamma$ where $F$ is a $k$-step free nilspace, i.e.\ $F=\prod_{i=1}^k\mc{D}_i(\mb{Z}^{a_i}\times \mb{R}^{b_i})$, and $\Gamma$ is a fiber-transitive, fiber-discrete subgroup of $\tran(F)$. Let $\nss$ be a (nilspace) extension of $\ns$ by $\mc{D}_\ell(\mb{T})$. Then $\nss\cong F\times \mc{D}_\ell(\mb{R})/\Gamma'$ where $\Gamma'$ is a fiber-transitive, fiber-discrete group acting on $F\times \mc{D}_\ell(\mb{R})$ generated, for $j\in[k]$, by elements of $\tran_j(F\times\mc{D}_\ell(\mb{R}))$ given by $(\underline{x},z)\in F\times \mc{D}_\ell(\mb{R})\mapsto (\gamma(\underline{x}),z+\gamma'(\underline{x}))$ where $\gamma\in \Gamma_j\subset  \Gamma$ and $\gamma'=\gamma'_\gamma\in \hom(F,\mc{D}_{\ell-j}(\mb{R}))$ and the element in $\tran_\ell(F\times\mc{D}_\ell(\mb{R}))$ given by $(\underline{x},z)\mapsto (\underline{x},z+1)$ .
\end{proposition}

\begin{proof}
Consider the following commutative diagram, where our extension is given by $\psi:\nss\to \ns\cong F/\Gamma$,
\begin{equation}
\begin{aligned}[c]
\begin{tikzpicture}
  \matrix (m) [matrix of math nodes,row sep=2em,column sep=4em,minimum width=2em]
  {\nss\times_{F/\Gamma} F & F  \\
     \nss & F/\Gamma. \\};
  \path[-stealth]
    (m-1-1) edge node [above] {$p_2$} (m-1-2)
    (m-1-1) edge node [right] {$p_1$} (m-2-1)
    (m-1-2) edge node [right] {$\pi_\Gamma$} (m-2-2)
    (m-2-1) edge node [above] {$\psi$} (m-2-2);
\end{tikzpicture}
\end{aligned}
\end{equation}
By \cite[Lemma 2.18]{CGSS-doucos}, it follows that $\nss\times_{F/\Gamma} F$ is an extension of $F$ by $\mc{D}_\ell(\mb{T})$. In fact, we know that it is a $\max(k,\ell)$-step Lie-fibered nilspace. Note that we have an action of $\mc{D}_\ell(\mb{T})$ by addition of the first coordinate and it is easy to see (using that $\psi$ is an extension) that with this action the map $p_2$ is a (continuous) morphism that satisfies the conditions of \cite[Definition 3.3.13]{Cand:Notes1}.

Hence, by \cite[Theorem 4.1]{CGSS-doucos} we have that such extensions split (the cited theorem applies for top-degree extensions of a free nilspace; if $\ell <k$, first take the $\ell$-step factor, split the extension over the free factor $F_\ell$, and pull back). Thus, there exists a morphism $\iota:F\to \nss\times_{F/\Gamma} F$ such that $p_2\co\iota=\id$, and in particular $\nss\times_{F/\Gamma} F\cong F\times \mc{D}_\ell(\mb{T})$.

Now, we construct a fiber-discrete and fiber-transitive group $\Gamma^*$ acting on $\nss\times_{F/\Gamma}F$ such that $(\nss\times_{F/\Gamma}F)/\Gamma^*\cong \nss$. Given any $\gamma\in \Gamma_j$ (for some $j\in[k]$), we may let $\gamma^*\in \tran_j(\nss\times_{F/\Gamma} F)$ be the map $(y,f)\mapsto (y,\gamma(f))$. First, let us check that this is indeed a translation on $\nss\times_{F/\Gamma} F$. To check that this map is well-defined note that for any $(y,f)\in \nss\times_{F/\Gamma} F$ we have that $\psi(y)=\pi_\Gamma(f)$ which clearly implies that $\psi(y)=\pi_\Gamma(\gamma(f))$. The fact that this is a translation follows easily from the definitions. Consider $\Gamma^*:=\{(\id,\gamma):\gamma\in \Gamma\}\subset \tran(\nss\times_{F/\Gamma} F)$. We claim that this group is fiber-discrete and fiber-transitive on $\nss\times_{F/\Gamma} F$ and that $(\nss\times_{F/\Gamma} F)/\Gamma^*\cong \nss$.

The fiber-transitive property follows from that of $\Gamma$, as indeed $\Gamma^*$ only acts on the second component of $\nss\times_{F/\Gamma} F$. The fiber-discrete property follows from the description of the structure groups of $\nss\times_{F/\Gamma} F$ given by \cite[Lemma 2.18]{CGSS-doucos}. Indeed, the only non-trivial point to check would be the case of the $\ell$-th structure group. But in that case, it follows easily that $\ab_\ell(\nss\times_{F/\Gamma} F)\cong \ab_\ell(F)\times \mb{T}$ and the action of $\Gamma^*_\ell$ on such a group is by addition on $\ab_\ell(F)$ (and hence the property follows from that of $\Gamma$). Finally note that $p_1$ factors through $\Gamma^*$, i.e.\ there is a fibration $(\nss\times_{F/\Gamma} F)/\Gamma^*\to \nss$. Moreover, using the description of the structure groups of $\nss\times_{F/\Gamma} F$ and the fact that $\psi$ is an extension (and hence, a fibration by \cite[Proposition A.17]{CGSS-p-hom}) we can conclude that it is in fact a nilspace isomorphism.

Finally, we want to see $\nss$ as a quotient of a free nilspace. First of all, note that as we have an isomorphism $\nss\times_{F/\Gamma} F\cong F\times \mc{D}_\ell(\mb{T})$, any $\gamma^*=(\id,\gamma)\in\Gamma^*_j$, transported
through this isomorphism, has an expression of the form $(f,z)\in F\times \mc{D}_\ell(\mb{T})\mapsto (\gamma(f),z+\tilde{\gamma}(f))$ for some $\tilde{\gamma}\in \hom(F,\mc{D}_{\ell-j}(\mb{T}))$, by \cite[Theorem 3.15]{CGSS-doucos}, since the splitting
isomorphism is over $F$ and
$p_2\circ\gamma^*=\gamma\circ p_2$, so the first coordinate
of the transported translation is $\gamma(f)$. By \cite[Lemma 3.7]{CGSS-doucos} note that $\tilde{\gamma}$ can be regarded as a polynomial map $\gamma'\in\hom(F,\mc{D}_{\ell-j}(\mb{R})) $ composed with the quotient map $\mb{R}\to\mb{T}$. Moreover, clearly $F\times \mc{D}_\ell(\mb{T})=F\times \mc{D}_\ell(\mb{R})/\langle (f,z)\mapsto(f,z+1)\rangle$ and the transformation $(f,z)\in F\times\mc{D}_\ell(\mb{R})\mapsto(f,z+1)$ clearly commutes with any transformation of the form $(f,z)\in F\times\mc{D}_\ell(\mb{R})\mapsto (\gamma(f),z+\gamma'(f))$. If we denote by $\Gamma'$ the group of translations in $F\times\mc{D}_\ell(\mb{R})$ generated by $(f,z)\mapsto(\gamma(f),z+\gamma'(f))$ and $(f,z)\mapsto(f,z+1)$ it is easy to see that $\Gamma'$ is fiber-discrete and fiber-transitive and that $(F\times\mc{D}_\ell(\mb{R}))/\Gamma'\cong \nss$.
\end{proof}

\begin{remark}
Although we stated the previous result for $\mc{D}_\ell(\mb{T})$-extensions, the same argument works for any compact abelian Lie group, with the natural adaptations.
\end{remark}

We can now establish the relationship between being weakly-splitting and having a representation with an orthogonal lattice.

\medskip

\begin{proposition}\label{prop:ws<=>nr}
Let $\ns$ be a $k$-step \textsc{cfr} nilspace. Then the following holds.
\setlength{\leftmargini}{0.7cm}
\begin{enumerate}
    \item\label{it:ortho-implies-wesp} If $\ns\cong F/\Gamma$ where $\Gamma\le \tran(F)$ is an orthogonal translation lattice\footnote{Note that, for $\ns$ to be a ported nilspace, the only property of $\Gamma$ missing here is pureness.}, then $\ns$ is weakly-splitting.
    \item\label{it:wesp-implies-ortho} If $\ns$ is weakly-splitting, then there exist a $k$-step free nilspace $F$ and an orthogonal translation lattice $\Gamma\le \tran(F)$ such that $\ns\cong F/\Gamma$.
    \end{enumerate}
\end{proposition}

\begin{proof}
We begin by proving \eqref{it:ortho-implies-wesp}. For this, it suffices to prove that the nilspace $\ns=F/\Gamma$ is weakly-splitting. We will use the characterization given in Proposition \ref{prop:weak-ortho-eq} of weak-splitting and find a fibration to all structure groups. By Lemma \ref{lem:orth-and-factors} note that by induction it suffices to find a fibration to the last structure group.

Let us write $F=\mc{D}_1(\mb{R}^{a_1}\times \mb{Z}^{b_1}) \times \cdots\times \mc{D}_k(\mb{R}^{a_k}\times \mb{Z}^{b_k})$  where $a_i,b_i\in \mb{Z}_{\geq 0}$. Considering the coordinate index set of $F$, namely $\big[\sum_{i=1}^k a_i+b_i\big]$, let $J$ denote the subset of this index set corresponding to the $\mc{D}_k(\mb{Z}^{b_k})$ component and $J'$ the index set corresponding to any $\mc{D}_i(\mb{Z}^{b_i})$ coordinates (in particular $J\subset J'$). Let $F_J$ (resp. $F_{J'}$) denote the free nilspace obtained from $F$ by deleting any coordinate whose index is not in $J$ (resp. $J'$). Thus $F_J\cong \mc{D}_k(\mb{Z}^{b_k})$, and let $P_J:F\to F_J$ be the corresponding coordinate projection. Note that $F_{J'}=D$, the discrete part $F$ and $p_D:F\to F_{J'}$ is the projection to the discrete part (see Definition \ref{def:free-nil-intro}).

We now claim that the subgroup $\Gamma\leq \tran(F)$ is consistent\footnote{Recall that given a morphism $m:\ns\to \nss$, its set of \emph{consistent translations} is $\{\alpha\in \tran(\ns):\forall x,y\in \ns\;\text{if}\;m(x)=m(y)\;\text{then}\;m(\alpha(x))=m(\alpha(y))\}$. A translation $\alpha\in \tran(\ns)$ is \emph{consistent} if it belongs to such set.} with $P_J$, and to prove this we shall use the orthogonality property. Let $x,y\in F$ and $\gamma\in \Gamma$ be such that $P_J(x)=P_J(y)$. We want to prove that $P_J(\gamma x)=P_J(\gamma y)$, i.e.\ that every element of $\Gamma$ is \emph{consistent} with $P_J$. By Lemma \ref{lem:Dpartproj}, every $g\in \tran(F)$ is consistent with $p_D$. Hence, if we let $q:F_{J'}=D\to F_J$, then $P_J=q\co p_D$ and the set of consistent translations of $P_J$ is the pre-image under $\wh{p_D}$ of the set of consistent translations of $q$. Thus, it suffices to prove that $\wh{p_D}(\Gamma)$ is consistent with $q$. By orthogonality, we have that $\wh{p_D}(\gamma)$ is an element of a lattice in $D=\prod_{i=1}^k \mb{Z}^{b_i}$ generated by multiples of the standard generators of $D$. Therefore, it is immediate that if $q(p_D(x))=q(p_D(y))$, then $q(\wh{p_D}(\gamma)p_D(x))=q(\wh{p_D}(\gamma)p_D(y))$, simply because on the coordinates indexed by $J$ the map $\wh{p_D}(\gamma)$ is adding a constant element of $\mb{Z}^{b_k}$. Thus $\wh{p_D}(\Gamma)$ is consistent with $q$ and our claim follows.

It follows from consistency (see \cite[Lemma 1.5]{CGSS}) that we have a group homomorphism $\wh{P_J}:\Gamma\to \tran(F_J)=\tran(\mc{D}_k(\mb{Z}^{b_k}))$, well-defined as the map sending $\gamma$ to the translation $x\in F_J\mapsto P_J(\gamma x')$ for any $x'\in F$ with $P_J(x')=x$.

The image $\Gamma_J:=\wh{P_J}(\Gamma)$ is a subgroup of the abelian group $\tran(F_J)\cong \mb{Z}^{b_k}$. Note now that $\Gamma_J=(\Gamma_J)_k=(\wh{P_J}(\Gamma))_k = (\wh{q}\co \wh{p_D}(\Gamma))_k = \wh{q}\big(( \wh{p_D}(\Gamma))_k\big)$ where in the first equality we have used that $F_J$ is of the form $\mc{D}_k(\mb{Z}^{b_k})$ and in the last one that $\wh{p_D}(\Gamma)$ is a lattice generated by \emph{elementary generators} and $\wh{q}$ is simply projecting to the coordinates of degree $k$. By the orthogonality property, we have furthermore that $( \wh{p_D}(\Gamma))_k = \wh{p_D}(\Gamma_k)$ and therefore $\Gamma_J = \wh{q}\co \wh{p_D}(\Gamma_k)= \wh{P_J}(\Gamma_k)$. In particular, the finite part $B_k$ of the $k$-th structure group of $\ns$, which equals $\mb{Z}^{b_k}/\wh{P_J}(\Gamma_k)$, is in turn $\mb{Z}^{b_k}/\Gamma_J$. 

Now note that the map $P_J$ is a nilspace fibration $F\to F_J$, and that the projection $\pi_J:F_J\to F_J/\Gamma_J =\mc{D}_k(B_k)$ is also a nilspace fibration. Hence $Q_J:=\pi_J\co P_J$ is a nilspace fibration from $F$ to $\mc{D}_k(B_k)$, and it now only remains to show that $Q_J$ factors through quotienting by $\Gamma$ on $F$, yielding a fibration $\ns\to \mc{D}_k(B_k)$. But this factoring indeed holds, because projecting by $P_J$ and then quotienting with $\pi_J$ is the same as first quotienting by $\Gamma$ and then projecting to $J$. Indeed, let $x,y\in F$ such that $y=\gamma(x)$ for some $\gamma\in \Gamma$. Then  $P_J(y)=\wh{P_J}(\gamma)P_J(x)$. As $\wh{P_J}(\gamma)\in \Gamma_J$ we clearly have that $P_J(x)\Gamma_J=P_J(y)\Gamma_J$ and thus $\pi_J\co P_J$ factors through $\pi_\Gamma$ as desired.

Let us prove now \eqref{it:wesp-implies-ortho}. We are going to prove it by induction on $k$ with the case $k=0$ being trivial. If $\ns$ is now a $k$-step weakly-splitting \textsc{cfr} nilspace, we have that $\ns\cong \nss\times \mc{D}_k(B_k)$ where $B_k$ is a finite abelian group and $\nss$ is a $k$-step extension of $\ns_{k-1}$ by $\mc{D}_k(\mb{T}^n)$ for some $n\in \mb{N}$. Moreover, we know that $\ns_{k-1}\cong F/\Gamma$ where $F$ is a $k-1$-step free nilspace and $\Gamma\le \tran(F)$ is a fiber-transitive, fiber-discrete, fiber-cocompact, orthogonal action.

By Proposition \ref{prop:ext-as-quo-of-free} (iterated $n$ times), we know that $\nss\cong F\times \mc{D}_k(\mb{R}^n)/\Gamma'$ where $\Gamma'$ acts as follows. For $(x,y)\in F\times \mc{D}_k(\mb{R}^n)$ the group $\Gamma'$ is generated by translations of the form $(x,y)\mapsto(\gamma(x),y+\gamma^*(x))$ where $\gamma\in \Gamma$ and $\gamma^*\in \hom(F,\mc{D}_k(\mb{R}^n))$ and $(x,y)\mapsto(x,y+e_i)$ where $e_i\in \mb{R}^n$ is the vector with 1 in the $i$-th coordinate and 0 in the rest. Hence, it is clear that $\wh{p_D}(\Gamma')=\wh{p_D}(\Gamma)$ is generated by translations.\footnote{Here we are abusing the notation as $p_D$ is used as a function from both $F$ and $F\times\mc{D}(\mb{R}^n)$. However, by definition $p_D$ removes the $\mb{R}^n$ component and thus the previous equality holds.} Moreover, we have that, for every $i\in[k]$, if $\alpha\in (\wh{p_D}(\Gamma))_i$, then first find some $\gamma\in \Gamma_i$ such that $\alpha=\wh{p_D}(\gamma)$ and by Proposition \ref{prop:ext-as-quo-of-free} we find some $\gamma^*$ that lifts it to $\Gamma'_i$. This shows that $(\wh{p_D}(\Gamma))_i\subset \wh{p_D}(\Gamma_i')$ and thus we have an equality as the other inclusion is trivial.

Write $B_k\cong\prod_{j=1}^m\mb{Z}/b_j\mb{Z}$ and let
$\varphi:\mb{Z}^m\to B_k$ be the coordinate quotient map, so that
$\ker(\varphi)=b_1\mb{Z}\times\cdots\times b_m\mb{Z}$.
On $F\times \mc{D}_k(\mb{R}^n)\times \mc{D}_k(\mb{Z}^m)$
we define $\Gamma''$ as follows: it is generated by transformations
of the form $(x,y,z)\mapsto(\gamma(x),y+\gamma^*(x),z)$ where
$(\gamma(\cdot),\cdot+\gamma^*(\cdot))\in \Gamma'$ and by the
translations $(x,y,z)\mapsto(x,y,z+a)$ where $a\in \ker(\varphi)$.
Note that these latter translations commute with each other and
with the transformations coming from $\Gamma'$, and, since
$\ker(\varphi)$ is generated by elementary coordinate shifts,
it is easy to see that $\Gamma''$ is an orthogonal translation
lattice such that
$\ns\cong
\bigl(F\times \mc{D}_k(\mb{R}^n)\times
\mc{D}_k(\mb{Z}^m)\bigr)/\Gamma''$.
\end{proof}

\begin{remark}
Note that being weakly-splitting is an intrinsic property of a nilspace (in the sense that it is preserved under nilspace isomorphisms). In contrast, orthogonality is a property relative to a chosen representation of the nilspace as $F/\Gamma$. Property \eqref{it:wesp-implies-ortho} shows that, for weakly-splitting nilspaces, there exists at least one representation $F/\Gamma$ in which $\Gamma$ is orthogonal. However, in general it is not true that \emph{every} representation of a weakly-splitting nilspace is orthogonal (we leave it as an exercise to obtain concrete examples).
\end{remark}

\subsection{Free nilspaces as Lie-group torsors}\label{subsec:LieParam}\hfill\\
Here we collect some basic facts about simply-connected nilpotent Lie groups, and use them to establish a useful result: given a connected free nilspace $F$ acted upon by a translation lattice $\Gamma$, we can identify $F$ as a torsor (principal homogeneous space) of a connected and simply-connected nilpotent Lie group. We refer to this identification as the \emph{Lie parametrization} of $F$ (see Theorem \ref{thm:flierep}). We then extend this construction to the case where $F$ is not necessarily connected, as long as $\Gamma$ is a \emph{porting group} as per Definition \ref{def:nicerep}. This yields our main result in this subsection, Proposition \ref{prop:Lieparamgam}.

\medskip

We recall the following classical result from Lie theory \cite[Theorem 1.2.1]{C&G}.
\begin{theorem}\label{thm:corresp-lie-gr-alg}
Let $G$ be a connected and simply-connected nilpotent Lie group and let $\mk{g}$ be its Lie algebra. Then the exponential map $\exp:\mk{g}\to G$ is a diffeomorphism, with inverse map $\log:G\to \mk{g}$. Moreover, the Baker--Campbell--Hausdorff formula holds for all elements of $\mk{g}$.
\end{theorem}

\noindent Using this, we can define 1-parameter subgroups: each element $g\in G$ lies in the 1-parameter subgroup $\{\exp(t \log g):t\in \mb{R}\}$, which we denote by $g^\mathbb{R}$. These subgroups correspond to 1-dimensional subspaces in the Lie algebra. 

Important examples of connected and simply-connected Lie groups are given by the translation groups of connected free nilspaces.

\begin{lemma}\label{lem:confreetranssimpcon}
Let $F$ be a $k$-step connected free nilspace. Then $\tran(F)$ is a connected and simply-connected Lie group.
\end{lemma}

\begin{proof}
By \cite[Lemma 3.6]{CGSS-doucos} and \cite[Theorem 3.15]{CGSS-doucos}, we have an explicit description of the Lie group $\tran(F)$ in terms of polynomials of (graded) degree at most $k$. In particular, this establishes a diffeomorphism (not necessarily a group homomorphism) between $\tran(F)$ and $\mb{R}^n$ for some $n\in \mathbb{N}$. Since the property of connectedness and simple-connectedness is preserved by homeomorphisms, the result follows.
\end{proof}

\begin{lemma}\label{lem:free-conn-action}
Let $F$ be a connected free nilspace and let $\Gamma\leq\tran(F)$ be a fiber-transitive and fiber-discrete group. Then $\Gamma$ acts freely on $F$, i.e.\ if $\gamma(x)=x$ for some $x\in F$, then $\gamma=\id$.
\end{lemma}

\begin{proof}
Note that $F=\prod_{i=1}^k\mc{D}_i(\mb{R}^{a_i})$ for some $a_i\in \mb{Z}_{\geq 0}$. Let $\gamma\in \Gamma$ and let us prove that if $\gamma\not=\id$ then it cannot fix any element of $F$. By \cite[Lemma 3.6]{CGSS-doucos} and \cite[Theorem 3.15]{CGSS-doucos} we know that $\gamma(\underline{x_1},\ldots,\underline{x_k})=(\underline{x_1},\ldots,\underline{x_k})+(0,\ldots,0,T_i(\underline{x_1},\ldots,\underline{x_{i-1}}),\ldots)$ for some $i\in[k]$ where $T_i\not=0$. We claim that $T_i$ is a constant map. Indeed, projecting onto the $i$-th factor we can assume without loss of generality that $i=k$. By the fiber-transitive property the polynomial $T_k(\underline{x_1},\ldots,\underline{x_{k-1}})$ should take values in the lattice $\Gamma_k$, which is discrete by assumption. This is possible only if $T_k$ is a constant, since $T_k$ is continuous and $F$ is connected. We have thus proved our claim and since $T_i$ is thus a non-zero constant, it is clear that $\gamma$ does not fix any element of $F$.
\end{proof}

\begin{corollary}\label{cor:freeconpure}
Let $F$ be a connected free nilspace and let $\Gamma\subset \tran(F)$ be a fiber-transitive and fiber-discrete group. Then $\Gamma$ is pure, i.e., for any $i\in[k]$ we have $\ker(\eta_{i-1})\cap \Gamma = \Gamma_i$.
\end{corollary}

\begin{proof}
This follows from \cite[Proposition 5.42]{CGSS-doucos} combined with Lemma \ref{lem:free-conn-action}.
\end{proof}

We will use the following classical result (see \cite[Proposition 2.5, p.\ 31] {Rag} and also \cite[Theorem 5.4.3, p.\ 217]{C&G}).

\begin{theorem}[Zariski closure of a subgroup]\label{thm:lieconclos}
Let $G$ be a connected simply-connected nilpotent Lie group and let $H$ be a subgroup of $G$. Then there is a unique minimal (relative to inclusion) connected closed subgroup $\tilde{H}$ of $G$ such that $\tilde{H}\supset H$, namely $\tilde{H}$ is the intersection of all connected closed subgroups of $G$ that include $H$. If $H$ is closed, then $\tilde{H}/H$ is compact. Moreover $\tilde{H}$ is simply-connected. We call $\tilde{H}$ the \emph{Zariski closure} of $H$. 
\end{theorem}
\noindent For a proof, see \cite[Proposition 2.5]{Rag}. The next result will enable us to conclude that certain subgroups are simply-connected.

\begin{lemma}\label{lem:subgpsimpcon}
Let $G$ be a nilpotent connected and simply-connected Lie group. Then any connected closed subgroup is simply-connected.
\end{lemma}

\begin{proof}
By the correspondence between Lie subalgebras and connected Lie subgroups \cite[Theorem 5.20]{Hall}, any such subgroup $H$ of $G$ corresponds to a Lie subalgebra, which is in particular a vector subspace, hence simply-connected. Since the exponential map is a diffeomorphism, it follows that $H$ is simply-connected.
\end{proof}

\noindent The term ``Zariski closure" is used for instance in \cite[Remark 2.6 and Theorem 2.10]{Rag}, where it is noted that, invoking Ado's Theorem to realize $H$ as a matrix group, one sees that $\tilde{H}$ is indeed a closure in the Zariski topology.

Given a connected free nilspace $F$, we know that $\tran(F)$ is a filtered nilpotent Lie group. However, for many purposes this group is too large. The concept of Zariski closure will play a key role by enabling us, given a discrete subgroup $\Gamma$ of $\tran(F)$ acting on $F$ in a sufficiently nice way, to use the smaller connected and simply-connected subgroup $\tilde{\Gamma}\le\tran(F)$ to analyze the compact nilspace $F/\Gamma$. 

We recall the following central notion in Lie group theory, see \cite[\S 3.1]{HKbook}.

\begin{defn}[Rational subgroups, connected case]
Let $G/\Gamma$ be a nilmanifold (see \cite[Definition 1.1]{GTorb}), where $G$ is connected and simply-connected. A subgroup $G'\leq G$ is said to be \emph{rational} (relative to $\Gamma$) if it is a closed, connected subgroup of $G$ such that $G'\cap \Gamma$ is cocompact in $G'$ (see \cite[Definition 1.10]{GTorb}). 
\end{defn}

We will use a slight generalization of the previous definition, in which $G$ is not required to be connected and simply-connected. Instead, we let $G$ be a nilpotent Lie group and let $\Gamma$ be a discrete cocompact subgroup. We say that $G'\le G$ is \emph{rational} if it is a closed subgroup of $G$ and $G'\cap \Gamma$ is cocompact in $G'$.

In the connected case, this notion of rationality is equivalent to the original notion involving rational structure constants of the Lie algebra  \cite[Theorem 5.1.11]{C&G}. In the more general (not necessarily connected) setting, we just define rationality in terms of cocompactness as above. Note that by the proof of \cite[Proposition 1.1.2]{Cand:Notes2}, rationality in this sense is necessary and sufficient for the associated coset nilspace to be a compact nilspace (that is, for each cube set to be appropriately closed). 

It turns out that, in the setting of connected free nilspaces, a translation lattice $\Gamma$ is automatically pure, and its Zariski closure $\tilde{\Gamma}$ is then rational relative to $\Gamma$.

\begin{lemma}\label{lem:ratsubgpsintransF}
Let $F$ be a $k$-step connected free nilspace and let $\Gamma\le \tran(F)$ be a translation lattice. Then $\Gamma$ and $\wt{\Gamma}$ are both pure. Moreover, for every $i\in [k]$ the group $\Gamma\cap \tran_i(F)=\Gamma\cap \ker(\eta_{i-1})$ is a cocompact subgroup of $\tilde{\Gamma}\cap \tran_i(F)=\tilde{\Gamma}\cap \ker(\eta_{i-1})$ (i.e.\footnote{The equivalence between $H$ being a rational subgroup of $G$ relative to a lattice $\Gamma\leq G$ and cocompactness of $H\cap \Gamma$ in $H$ follows from \cite[Theorem 5.1.11]{C&G}.} the group $\wt{\Gamma}\cap \tran_i(F)$ is a rational subgroup of $\tilde{\Gamma}$ relative to $\Gamma$).
\end{lemma}
\noindent To prove this crucial result, we use the following fact about pure subgroups of $\tran(F)$ (recall  Definition \ref{def:pure}).

\begin{lemma}\label{lem:PureZariski}
Let $F$ be a $k$-step connected free nilspace and let $\Gamma$ be a pure fiber-discrete subgroup of $\tran(F)$. Then
\begin{equation}\label{eq:PZclaim}
g\in \tilde{\Gamma}\cap\ker(\eta_i) \quad \Longrightarrow \quad g=\gamma_1^{t_1}\cdots \gamma_r^{t_r}\text{ for some }\gamma_j\in \Gamma\cap\tran_{i+1}(F)\text{ and }t_j\in \mb{R}.
\end{equation}
In particular $\wt{\Gamma}$ is pure, i.e.\ we have
\begin{equation}\label{eq:PZclaim2}
\wt{\Gamma}\cap\ker(\eta_i) = \wt{\Gamma\cap\tran_{i+1}}(F) = \wt{\Gamma}\cap\tran_{i+1}(F),
\end{equation}
and so $\tilde{\Gamma}$ is also fiber-transitive and free.
\end{lemma}
\begin{proof}
We prove \eqref{eq:PZclaim} by induction on $i\in [0,k]$. To prove the case $i=0$, by \cite[Definition 2.1 and the remarks below it]{GTorb}, we have that $\tilde{\Gamma}/\Gamma$ admits a Mal'cev basis with respect to the lower central series of $\tilde{\Gamma}$. Hence, we can apply statement $(iii)$ of that definition to get \eqref{eq:PZclaim}.

Now suppose that $i>0$ and that  \eqref{eq:PZclaim} holds for $i-1$. We want to prove that  \eqref{eq:PZclaim} holds for $i$, so suppose that $g\in \tilde{\Gamma}\cap\ker(\eta_i)$. Then $g\in \ker(\eta_{i-1})$, so by the case $i-1$ we have
\begin{equation}
g=\gamma_1^{s_1}\cdots \gamma_r^{s_r}\text{ where }\gamma_j\in \Gamma\cap\tran_i(F)\text{ and }s_j\in\mb{R}.
\end{equation}
Suppose that there are among these factors some $\gamma_j$ that are in $\tran_i(F)\setminus\tran_{i+1}(F)$ (if there are none then we are already done). Then, by commuting towards the left in this product every such element $\gamma_j$, and in each such commuting operation using formula \cite[(C.2)]{GTarit}, we can rewrite $g$ as a product $\gamma_1^{t_1}\cdots \gamma_u^{t_u}\tau_{u+1}^{t_{u+1}}\cdots \tau_w^{t_w}$, where each $\gamma_j$ is in $\tran_i(F)\setminus \tran_{i+1}(F)$  and every $\tau_k\in\Gamma\cap\tran_{i+1}(F)$. Throughout these rearrangements, every higher-order correction is a finite product of real powers of elements of $\Gamma \cap \tran_{i+1}(F)$. Now, by pureness of the $\gamma_j$, for each such element there is a constant $\alpha^j$ such that $\eta_i(\gamma_j)$ is a translation that just adds the constant $\alpha^j$ in the last fiber of $F_i$. By definition of $F$, such a fiber is isomorphic to $\mc{D}_i(\mb{R}^\ell)$ for some $\ell\ge 0$. Moreover, we can assume without loss of generality (just by composing with an isomorphism) that $\eta_i(\Gamma_i)$ generates a subgroup of the lattice $\mb{Z}^{\ell}\subset\mb{R}^\ell$.\footnote{If $\Gamma$ is fiber-cocompact then we may assume that the lattice is in fact $\mb{Z}^\ell$. Otherwise we just know it is some sub-lattice possibly non cocompact.}

Select a maximal linearly independent subset of $\{\alpha^1,\ldots,\alpha^u\}$. Then, by a similar commutation argument as above using \cite[(C.2)]{GTarit}, we can assume that these linearly independent $\alpha^j$ are the first $v$ ones, i.e.\ we have
\begin{equation}\label{eq:zariski-is-pure-2}
g= \gamma_1^{t_1}\cdots \gamma_v^{t_v}\gamma_{v+1}^{t_{v+1}}\cdots \gamma_u^{t_u}\tau_{u+1}^{t_{u+1}}\cdots \tau_w^{t_w}
\end{equation}
where $\alpha^1,\ldots,\alpha^v$ are linearly independent and each $\alpha^j$ with $j\in [v+1,u]$ is a linear combination of the $\alpha^j$ with $j\in[v]$. We want to prove that either $v=u$ or we can write $g$ with a similar expression as above but with at most $u-1$ terms which lie in $\tran_i(F)\setminus\tran_{i+1}(F)$.

To prove the latter statement, recall that $g\in \ker(\eta_i)$ and therefore we have the system of equations $t_1\alpha^1+\cdots+t_u\alpha^u=0\in \mb{R}^\ell$. If we write $\alpha^j=(\alpha^j_1,\ldots,\alpha^j_\ell)$ for $j\in [u]$, we find that the $t_j$ satisfy the system of equations
\[
\begin{psmallmatrix} \alpha_1^1 & \ldots & \alpha^u_1\\[0.1em]\cdot  &  & \cdot\\[0.1em] \alpha_\ell^1 & \cdots & \alpha^u_\ell \end{psmallmatrix} \begin{psmallmatrix} t_1 & \\[0.1em]\cdot \\[0.1em] t_u  \end{psmallmatrix} = \begin{psmallmatrix}0 & \\[0.1em]\cdot \\[0.1em] 0 \end{psmallmatrix}.
\]
By assumption the first $\alpha^1,\ldots,\alpha^v$ columns are a maximally large linearly independent subset. Hence, by the Gaussian elimination algorithm, there exists an invertible matrix $E\in M_{\ell\times \ell}(\mb{Q})$ such that $E\begin{psmallmatrix} \alpha_1^1 & \ldots & \alpha^u_1\\[0.1em]\cdot  &  & \cdot\\[0.1em] \alpha_\ell^1 & \cdots & \alpha^u_\ell \end{psmallmatrix} = \begin{psmallmatrix} \Id_{v\times v} &  -R\\  0_{(\ell-v)\times v} &  0_{(\ell-v)\times (u-v)} \end{psmallmatrix}$ where $R$ is now a $v\times (u-v)$ matrix with entries in $\mb{Q}$. Therefore, any solution of this system of equations must satisfy \[\begin{psmallmatrix} R \\[0.1cm] \text{Id}_{(u-v)\times (u-v)}   \end{psmallmatrix} \begin{psmallmatrix} t_{v+1} & \\[0.1em]\cdot \\[0.1em] t_u  \end{psmallmatrix} = \begin{psmallmatrix} t_1 & \\[0.1em]\cdot \\[0.1em] t_u  \end{psmallmatrix}\] 

In particular, we have that $t_{v}=r_{v+1}t_{v+1}+\cdots+r_ut_u$ where $r_j\in \mb{Q}$. Hence, there exist integers $q\neq 0,n_{v+1},\ldots,n_u$ such that $t_{v}=\tfrac{1}{q}(n_{v+1}t_{v+1}+\cdots+n_ut_u)$. Now, in \eqref{eq:zariski-is-pure-2}, we want to merge $\gamma_v^{t_v}$ with the terms $\gamma_{v+1}^{t_{v+1}},\ldots,\gamma_u^{t_u}$ (times possible additional factors of higher order) to reduce the number of elements in $\tran_i(F)\backslash\tran_{i+1}(F)$. To do so, firstly, using again \cite[(C.2)]{GTarit} we can distribute $\gamma_{v}^{t_{v}} = \gamma_{v}^{(n_{v+1}t_{v+1})/q}\cdots \gamma_{v}^{(n_ut_u)/q}$ in the expression \eqref{eq:zariski-is-pure-2} so as to have
\[
g = \gamma_1^{t_1}\cdots \gamma_{v-1}^{t_{v-1}}(\gamma_{v+1}^{t_{v+1}}\gamma_v^{(n_{v+1}t_{v+1})/q})\cdots (\gamma_{u}^{t_{u}}\gamma_v^{(n_{u}t_{u})/q})\tau
\]
where $\tau$ is a product of elements from one-parameter subgroups of higher-order terms (as in \eqref{eq:zariski-is-pure-2}). The key point now is that, by \cite[(C.1)]{GTarit}, we have $\gamma_j^{t_j}\gamma_{v}^{(n_jt_j)/q} = (\gamma_j^{q})^{t_j/q}(\gamma_{v}^{n_j})^{t_j/q} = (\gamma_j^q\gamma_{v}^{n_j})^{t_j/q}\nu$ where $\nu$ are one-parameter subgroups of higher-order commutators and $\gamma_j^q\gamma_{v}^{n_j}\in \Gamma$. Thus we have reduced the number of factors in  \eqref{eq:zariski-is-pure-2} by one.

Repeating the above process, we eventually end up with no terms in $\tran_i(F)\setminus\tran_{i+1}(F)$ (and thus are done) or we have an expression of the form
\[
g= \gamma_1^{t_1''}\cdots \gamma_v^{t_v''}{\tau''}_{v+1}^{t''_{v+1}}\cdots {\tau''}_w^{t_w''}
\]
where the ${\tau''}_j$ are in $\tran_{i+1}(F)$ and the $\alpha^1,\ldots,\alpha^v$ are linearly independent. But now, precisely by this linear independence, the only way that $t_1''\alpha^1+\cdots+t_v''\alpha^v$ can be 0 (as it should be since $\eta_i(g)=\id$) is that each $t_i''$ is 0, and therefore we conclude that   
$g=(\tau''_{v+1})^{t''_{v+1}}\cdots(\tau''_w)^{t''_w}$ with $\tau''_j\in\Gamma\cap\tran_{i+1}(F)$ and $t''_j\in\mb{R}$. This is the product representation required by \eqref{eq:PZclaim}, completing the inductive step.

Now we prove \eqref{eq:PZclaim2}. Note that it suffices to prove the chain of inclusions $\wt{\Gamma}\cap\ker(\eta_i) \subset \wt{\Gamma\cap\tran_{i+1}}(F) \subset \wt{\Gamma}\cap\tran_{i+1}(F)$ as $\wt{\Gamma}\cap\tran_{i+1}(F)\subset \wt{\Gamma}\cap\ker(\eta_i)$ is trivial. To see the first inclusion, note that by \eqref{eq:PZclaim}, an element $g\in \tilde{\Gamma}\cap\ker(\eta_i)$ is a product of elements of the form $\gamma_j^{t_j}$ where $\gamma_j\in \Gamma\cap \tran_{i+1}(F)$, and so $\gamma_j^{t_j}\in \wt{\Gamma\cap \tran_{i+1}}(F)$, and since the latter is a group we conclude that $g\in \wt{\Gamma\cap \tran_{i+1}}(F)$, which proves the first claimed inclusion. The second inclusion  is clear because each of $\tilde{\Gamma}$ and $\tran_{i+1}(F)$ are connected, so by \cite[Lemma 2.4, p.\ 31]{Rag} so is their intersection, and so this intersection includes  $\wt{\Gamma\cap\tran_{i+1}}(F)$ by minimality of the latter closure.
\end{proof}

\begin{proof}[Proof of Lemma \ref{lem:ratsubgpsintransF}]
By the fiber-discrete property and \cite[Lemma 5.24]{CGSS-doucos}, the group $\Gamma$ is discrete in $\tran(F)$. By Corollary \ref{cor:freeconpure} we have that $\Gamma$ is pure, and then by Lemma \ref{lem:PureZariski} so is $\tilde{\Gamma}$.

Now for each $i\in [k]$, by Theorem \ref{thm:lieconclos} we have that $\tilde{\Gamma}$ is closed connected and furthermore, as $\Gamma$ is closed, by that same theorem $\tilde{\Gamma}/\Gamma$ is compact.

To prove more generally that $\Gamma\cap \tran_i(F)$ is cocompact in $\tilde{\Gamma}\cap \tran_i(F)$, first note that, since the former group is certainly cocompact in $\wt{\Gamma\cap \tran_i}(F)$, it suffices to prove that $\tilde{\Gamma} \cap \tran_i(F) = \wt{ \Gamma \cap \tran_i(F)}$, but this is given by \eqref{eq:PZclaim2}. This completes the proof.
\end{proof}

\begin{remark}
Note that \eqref{eq:PZclaim2} is quite a delicate fact. In particular it is \emph{not} true in general that if $\Gamma$ is a discrete cocompact subgroup of a connected simply-connected Lie group $G$, and $H$ is a connected subgroup of $G$, then the analogue of \eqref{eq:PZclaim2} holds, namely the equality $\tilde{\Gamma}\cap H=\wt{\Gamma\cap H}$. Indeed, let for instance $G=\mb{R}^2$, let $\Gamma$ be some rank 2 lattice therein, and let $H$ be a 1-dimensional subspace of $G$ having trivial intersection with $\Gamma$. Then since $\tilde{\Gamma}=G$ we have $H=\tilde{\Gamma}\cap H\supsetneq \wt{\Gamma\cap H}=\{0\}$. In fact, the equality $\tilde{\Gamma}\cap H=\wt{\Gamma\cap H}$ holds precisely when $H\cap \Gamma$ is cocompact in $\tilde{\Gamma}\cap H$.
\end{remark}
\begin{remark}\label{rem:commtildeta}
The above results also imply that $\eta_{k-1}(\tilde{\Gamma})=\wt{\eta_{k-1}(\Gamma)}$ in $\tran(F_{k-1})$. Indeed, by \cite[Lemma 5.1.4 (a), p.\ 196]{C&G} applied with $G=\tilde{\Gamma}$,  $H=\ker(\eta_{k-1})\cap G$,  and $\pi=\eta_{k-1}$, we have that $\eta_{k-1}(\Gamma)$ is uniform in $G/H\cong \eta_{k-1}(\tilde{\Gamma})$, and so by uniqueness this latter group must be the Zariski closure of $\eta_{k-1}(\Gamma)$ as claimed.
\end{remark}
The following result, which is central to this subsection, indicates more precisely the relevance of the Zariski closure. It tells us that the connected free nilspace $F$ is a principal homogeneous space (or \emph{torsor}) of the Lie group $\tilde{\Gamma}$.

\begin{theorem}[Lie parametrization of a connected free nilspace]\label{thm:flierep} Let $F$ be a connected free nilspace and let $\Gamma\le \tran(F)$ be a translation lattice. Then the Zariski closure $\tilde{\Gamma}$ acts transitively and freely on $F$. In particular there is a unique bijection $\beta:F\to \tilde{\Gamma}$ such that $\beta(0)=1_{\tilde{\Gamma}}$ and for every $\gamma\in \wt{\Gamma}$ and $x\in F$ we have $\beta(\gamma(x))=\gamma\beta(x)$. Moreover $\beta$ is a nilspace isomorphism \textup{(}relative to the group nilspace structure on $\tilde{\Gamma}$ associated with the filtration $\tilde{\Gamma}_i:=\tilde{\Gamma}\cap \tran_i(F)$\textup{)}.
\end{theorem}

This map $\beta$ is what we call the \emph{Lie parametrization} of $F$. Note that $\beta$ induces an embedding of $\Gamma$ as the subset $\beta^{-1}(\Gamma)$ of $F$.

\begin{proof}
The existence and bijectivity of the map $\beta$ will follow from the fact that $\tilde{\Gamma}$ defines a free transitive action on $F$. Assuming this holds, note that for every $x\in F$ there exists a unique $h\in \tilde{\Gamma}$ such that $h(0_F)=x$, whence we shall  define $\beta(x):=h$. Note that it then follows that $\gamma\beta(x)=\gamma h=\beta(\gamma(x))$, since $\gamma h$ is an element of $\tilde{\Gamma}$ (hence the unique element) mapping $0_F$ to $\gamma(x)$. Our first main task is thus to prove that $\tilde{\Gamma}$ acts freely and transitively on $F$. 

Since by Lemma \ref{lem:ratsubgpsintransF} the group $\wt{\Gamma}$ is pure, we have by \cite[Proposition 5.42]{CGSS-doucos} that $\wt{\Gamma}$ acts freely, and is also fiber-transitive, i.e.\ the following holds: 
\begin{equation}\label{eq:strongtrans}
\forall x,y\in F, s\geq 0, \text{ if } y\in \wt{\Gamma}x\text{ and } \pi_s(x)=\pi_s(y), \text{ then }  \exists h\in \wt{\Gamma}\cap \tran_{s+1}(F) \text{ s.t. }h(x)=y. 
\end{equation}
Let us now prove that $\wt{\Gamma}$ acts transitively as required. We argue by induction on the step $k$ of the nilspace. For $k=1$ the claim is clear since the translation lattice is in this case a cocompact subgroup of the abelian group $F$, so $\wt{\Gamma}=F$ in this case and so $\wt{\Gamma}$ clearly acts transitively on $F$. For $k>1$, fixing any $x,y\in F$, note that the projected group $\eta_{k-1}(\Gamma)$ is a translation lattice on $F_{k-1}$, so by induction and Remark \ref{rem:commtildeta}, there is $\gamma\in \wt{\Gamma}$ such that $\eta_{k-1}(\gamma)\pi_{k-1}(x)=\pi_{k-1}(y)$, i.e.\ $\pi_{k-1}(\gamma x)=\pi_{k-1}(y)$, so there exists $z\in \ab_k(F)$ such that $\gamma x+z=y$. But by fiber-cocompactness of $\Gamma$, we have that $\Gamma_k$ is cocompact in $\ab_k(F)$, whence $\wt{\Gamma_k}=\ab_k(F)$. By \eqref{eq:PZclaim2}, we have $\wt{\Gamma_k}=\wt{\Gamma}_k$, so $z\in \ab_k(F)= \wt{\Gamma_k}=\wt{\Gamma}_k\subset \wt{\Gamma}$, and thus $\gamma z$ is an element of $\wt{\Gamma}$ mapping $x$ to $y$ as required.

Now it only remains to prove that $\beta$ is a nilspace isomorphism. First let us prove the simple fact that $\beta^{-1}$ is a morphism. By definition recall that $\beta(x)$ is the unique element $h\in \tilde{\Gamma}$ such that $h(0)=x$. Hence $\beta^{-1}(h)=h(0)$. Therefore $\beta^{-1}$ is simply the evaluation in $0$. Given now $\q\in \cu^n(\tilde{\Gamma})$, note that $\beta^{-1}\co\q=\q\co 0^{\db{n}}$ (where $0^{\db{n}}$ is the constant cube in $\cu^n(F)$ with value $0$) which is clearly a cube in $F$ by the simple fact that $\cu^n(\tilde{\Gamma})\subset \cu^n(\tran(F))$.

We now prove that $\beta$ is a morphism. To this end, fix some $n\in\mb{N}$ and order the vertices $v_1,\ldots,v_{2^n}$ of $\db{n}$ so that whenever $v_i\neq v_j$ and $v_i(t)\leq v_j(t)$ for every $t\in [n]$, we have $i<j$. In other words, for any $j\in[2^n]$ the set $\{v_1,\ldots,v_j\}$ is a simplicial subset of $\db{n}$, see \cite[Definition 3.1.4]{Cand:Notes1}. Moreover, for any $j\in[2^n]$ let $F_j^\ell:=\{v_i\in\db{n}:v_i\sbr{t}\le v_j\sbr{t},\;\forall t\in[n]\}$ and $F_j^u:=\{v_i\in\db{n}:v_i\sbr{t}\ge v_j\sbr{t},\;\forall t\in[n]\}$.

Now, given any $\q\in \cu^n(F)$, we shall construct iteratively a sequence $\q'_j\in \cu^n(\tilde{\Gamma})$ such that $\q'_j(v_i)=\beta\co\q(v_i)$ for all $i\le j$, so that when we reach $j=2^n$ the desired claim, i.e.\ that $\beta\co\q$ is a cube, will be proved. For $j=1$ we simply define $\q_1'=\beta(\q(0^n))$ (the constant cube) which is of course an element of $\cu^n(\tilde{\Gamma})$ and agrees with $\beta\co\q$ on $v_1=0^{n}$.

Assume now that we have defined $\q_j'$, and let $m=\dim(F_{j+1}^\ell)$. By the inductive assumption, $\beta\co\q|_{F_{j+1}^\ell}$ and $\q_j'|_{F_{j+1}^\ell}$ agree on all vertices except maybe on $v_{j+1}$. Since \(\beta^{-1}\) is a morphism, the maps \(\q|_{F^\ell_{j+1}}\) and \(\beta^{-1}\circ(\q'_j|_{F^\ell_{j+1}})\) are \(m\)-cubes in \(F\), and they agree at every vertex except possibly \(v_{j+1}\). Here the second map sends \(v\) to \(\q'_j(v)(0_F)\). If $m>k$, cube uniqueness gives $\q(v_{j+1})=\q_j'(v_{j+1})(0_F)$, and we set $\q_{j+1}'=\q_j'$. Otherwise, $\q(v_{j+1})\sim_{m-1}\q_j'(v_{j+1})(0_F)$. By \eqref{eq:strongtrans} there exists $h\in \tilde{\Gamma}\cap \tran_m(F)$ such that 
\[
h(\q'_{j}(v_{j+1})(0_F)) = \q(v_{j+1}).
\]
Then $\q_{j+1}':=h^{F^u_{j+1}}\q_{j}'\in \cu^n(\tilde{\Gamma})$ satisfies the desired properties. Continuing this way, we conclude that $\beta\co\q = \q'_{2^n}\in \cu^n(\tilde{\Gamma})$, and it follows that $\beta$ is a morphism. Hence $\beta$ is indeed a nilspace isomorphism.
\end{proof}

\begin{corollary}\label{cor:Lie-param-conn}
Under the assumptions of Theorem \ref{thm:flierep}, the isomorphism $\beta:F\to \tilde{\Gamma}$ induces a nilspace isomorphism $\overline{\beta}:F/\Gamma\to \Gamma\backslash \tilde{\Gamma}$ where $F/\Gamma$ is the quotient of a free nilspace by a fiber-transitive action and $\Gamma\backslash\tilde{\Gamma}$ is the (left) coset nilspace. Equivalently, by composing $\beta$ with the inversion map $\inv:\tilde{\Gamma}\to \tilde{\Gamma}$ given by $\gamma\mapsto\gamma^{-1}$ we have that $F/\Gamma\cong \tilde{\Gamma}/\Gamma$.
\end{corollary}

\begin{proof}
Note that $\beta$ defines a nilspace isomorphism and thus, any quotient of $F$  translates into a quotient of $\tilde{\Gamma}$. By the equivariance property, the action of $\Gamma$ is given precisely by left multiplication, and the result follows.
\end{proof}

\begin{remark}
Corollary \ref{cor:Lie-param-conn} can be used to give an alternative proof of \cite[Theorem 2.9.17]{Cand:Notes2}, i.e., that any toral nilspace is a coset nilspace.
\end{remark}
\noindent Having developed the Lie parametrization in the case of connected free nilspaces, we now begin to work toward an extension of this machinery valid for general (possibly disconnected) free nilspaces acted upon by \emph{porting} translation lattices. To this end, we begin with the following notion.

\begin{defn}\label{def:conn-closure}
The \emph{connected closure} of a free nilspace $F$ is obtained from $F$ by replacing all discrete components $\mc{D}_i(\mb{Z})$ of $F$ by continuous ones $\mc{D}_i(\mb{R})$. We denote this connected closure by $F_\mb{R}$, and we let $\iota$ denote the embedding $F\to F_\mb{R}$ given by $x\mapsto x$. Since every translation in $\tran(F)$ is given by polynomials (see \cite[Theorem 3.15]{CGSS-doucos}), we can use the same polynomials to define a translation on $F_\mb{R}$. We thus obtain an embedding homomorphism $\wh{\iota}:\tran(F)\to\tran(F_{\mb{R}})$.
\end{defn}

\begin{lemma}\label{lem:pure-i} Let $F$ be a free nilspace and $g\in\tran(F)$. Then $g$ is pure if and only if $\wh{\iota}(g)$ is pure.
\end{lemma}
\begin{proof}
By \cite[Proposition 5.42]{CGSS-doucos}, whether a translation $\alpha\in \tran(F)$ belongs to a subgroup $\tran_i(F)$ or $\ker(\eta_{i-1})$ is completely characterized by the form of the polynomials $T_i$ in the description of $\alpha$. Since $\wh{\iota}$ does not change these polynomials, the result follows.
\end{proof}

\begin{lemma}\label{lem:embednice}
Let $\ns$ be a $k$-step \textsc{cfr} nilspace of the form $\ns=F/\Gamma$ where $\Gamma\le\tran(F)$ is a pure translation lattice. Then $\wh{\iota}(\Gamma)$ is also a pure translation lattice in $\tran(F_{\mb{R}})$. Moreover, if $\Gamma$ was orthogonal then so is $\wh{\iota}(\Gamma)$.
\end{lemma}

\begin{proof}
By Lemma \ref{lem:pure-i} we have that $\wh{\iota}(\Gamma)$ is also pure, and in particular, fiber-transitive. It follows also that $\wh{\iota}(\Gamma)$ is also fiber-discrete and fiber-cocompact (simply because the lattice generated by $\eta_i(\wh{\iota}(\Gamma_{i}))\le \mb{Z}^{a_i}\times\mb{R}^{b_i}$ is clearly still discrete and cocompact when we embed $\mb{Z}^{a_i}\times\mb{R}^{b_i}\xhookrightarrow{}\mb{R}^{a_i}\times\mb{R}^{b_i}$). Finally, note that $\wh{\iota}(\Gamma)$ is trivially orthogonal in $ \tran(F_{\mb{R}})$ since there are no discrete components in $F_{\mb{R}}$.
\end{proof}

\noindent Our aim now is to establish a generalization of the Lie parametrization, valid even when $F$ may include discrete coordinates. As in the previous lemma, let $\Gamma$ be a pure translation lattice. Equipped with Lemma \ref{lem:embednice} there are two natural approaches:
\setlength{\leftmargini}{0.6cm}
\begin{enumerate}[label*=\arabic*.]
\item Consider the embedding $\iota:F\to F_{\mb{R}}$, use the Lie parametrization $\beta:F_{\mb{R}}\to L$ (where $L$ is the Zariski closure of $\wh{\iota}(\Gamma)$), then take $\beta(\iota(F))\le L$ as a candidate ambient group, hoping that $F/\Gamma\cong \beta(\iota(F))/\Gamma$.
\item Consider the stabilizer $\stab_L(\iota(F)):=\{h\in L:h(\iota(F))=\iota(F)\}$ and hope that $F/\Gamma\cong \stab_L(\iota(F))/\Gamma$.
\end{enumerate}
\noindent However, these approaches cannot work in general, as one can find examples of nilspaces $F/\Gamma$ where $\Gamma\le\tran(F)$ is a pure translation lattice, and yet $\stab_L(\iota(F))\subsetneq \beta(\iota(F))$ and $\beta(\iota(F))$ is not a subgroup of $L$.\footnote{Fix any prime $p$, let $F:=\mc{D}_1(\mb{Z})\times\mc{D}_2(\mb{Z})$, and let $\Gamma:=\langle \gamma_1,\gamma_2\rangle$ where $\gamma_1(x,y):=(x+p,y+x)$ and $\gamma_2(x,y):=(x,y+p)$. We omit the details.} 
Fortunately, there is a property that makes both approaches work and yield the same result: lattice orthogonality (see Definition \ref{def:ortlattice}). More precisely, for this approach to work it suffices if $\Gamma$ is a \emph{porting} subgroup of $\tran(F)$ (recall Definition \ref{def:nicerep}), which includes the \emph{orthogonality property}. This property is useful because it enables the following ``correction" operation.

\begin{lemma}
Let $\ns=F/\Gamma$ be a $k$-step ported \textsc{cfr} nilspace. Let $F_{\mb{R}}$ be the connected closure of the free nilspace $F$, with embedding $\iota:F\to F_\mb{R}$. Let $L$ be the Zariski closure of $\wh{\iota}(\Gamma)$ in $\tran(F_{\mb{R}})$, with filtration $\big(L_i:=L\cap\tran_i(F_\mb{R})\big)_{i\geq 0}$, and let $\beta:F_\mb{R}\to L$ be the corresponding Lie parametrization. Let $\stab_L(\iota(F))=\{h\in L: h(\iota(F))=\iota(F)\}$. Then, for every $i\in [k]$ and $h\in L \cap \tran_i(F_{\mb{R}})$, the following holds:
\begin{multline}\label{eq:L'-match}
\eta_i(h)(\iota(F_i))=\iota(F_i)\; \Rightarrow \exists\, \alpha_i\in \stab_L(\iota(F))\cap L_i\text{ such that }\alpha_i^{-1}h \in\ker(\eta_i).
\end{multline}
\end{lemma}

\begin{proof}
Suppose that $h\in L\cap \tran_i(F_{\mb{R}})$ and $\eta_i(h)$
stabilizes $\iota(F_i)$. Let $D$ denote the discrete part of $F$.
We identify $\Gamma$ with $\wh{\iota}(\Gamma)$, and write
$\Gamma_j=\Gamma\cap\tran_j(F_{\mb{R}})$, with
$\Gamma_{k+1}=\{\id\}$. Since $h\in\tran_i(F_{\mb{R}})$, we have
$\eta_{i-1}(h)=\id$. Applying \eqref{eq:PZclaim}
with $i-1$ in place of $i$, we obtain
$h=\gamma_1^{t_1}\cdots\gamma_r^{t_r}$ where
$\gamma_j\in\Gamma_i$ and $t_j\in\mb{R}$. 

By the orthogonality assumption, for each $\ell\in [r]$ we can
find elements $g_{\ell,1},\ldots,g_{\ell,s_\ell}\in\Gamma_i$
whose images under $\wh{p_D}$ are elementary generators of
$\wh{p_D}(\Gamma)$, such that the action of each such element
on the discrete components of $F$ is non-trivial
only\footnote{That is, in a discrete component
$\mc{D}_j(\mb{Z})$ with $j\neq i$, each of these elements
just shifts by 0.} in discrete components of degree $i$,
and such that for some integers
$n_{\ell,1},\ldots,n_{\ell,s_\ell}$ we have that the element
$\zeta_\ell:=
g_{\ell,1}^{n_{\ell,1}}\cdots
g_{\ell,s_\ell}^{n_{\ell,s_\ell}}\in\Gamma_i$
satisfies that $\eta_i(\zeta_\ell),\eta_i(\gamma_\ell)$ act
as the same constant shift in each discrete component of $F_i$.
Here the choice of the lifts in $\Gamma_i$ uses the equality
$\wh{p_D}(\Gamma_i)=\wh{p_D}(\Gamma)\cap\tran_i(D)$.
The advantage of the elements $\zeta_\ell$ is that, by our
choice of the elements $g_{\ell,j}$, on each discrete component
of $F$ of degree higher than $i$, these elements $\zeta_\ell$
just shift by 0.

We now modify the elements $\zeta_\ell$ without changing their
action on the discrete coordinates, so that they also agree
with $\gamma_\ell$ on the continuous coordinates of $F_i$.
For each $\ell\in[r]$, the translation
$\wh{p_D}(\zeta_\ell^{-1}\gamma_\ell)$ fixes all discrete
coordinates of degree at most $i$. By orthogonality, this translation is a constant shift on $D$.
Since its displacement vanishes in every coordinate of degree
at most $i$, it belongs to
$\wh{p_D}(\Gamma)\cap\tran_{i+1}(D)$.
If $i<k$, the orthogonality assumption therefore gives an
element $\delta_\ell\in\Gamma_{i+1}$ such that
\[
\wh{p_D}(\delta_\ell)
=\wh{p_D}(\zeta_\ell^{-1}\gamma_\ell).
\]
If $i=k$, the right-hand side is the identity, and we take
$\delta_\ell=\id$.
Replacing $\zeta_\ell$ by $\gamma_\ell\delta_\ell^{-1}$,
we retain its discrete action, because
\[
\wh{p_D}(\gamma_\ell\delta_\ell^{-1})
=\wh{p_D}(\zeta_\ell).
\]
Moreover, the new $\zeta_\ell$ still belongs to $\Gamma_i$
and satisfies $\eta_i(\zeta_\ell)=\eta_i(\gamma_\ell)$ 
because $\delta_\ell\in\Gamma_{i+1}\subset\ker(\eta_i)$.

Now let $\alpha_i:=\zeta_1^{t_1}\cdots\zeta_r^{t_r}$,
and note that $\alpha_i\in L_i$. Moreover, by construction
we have that
\[
\eta_i(\alpha_i)
=\eta_i(\zeta_1)^{t_1}\cdots\eta_i(\zeta_r)^{t_r}
=\eta_i(\gamma_1)^{t_1}\cdots\eta_i(\gamma_r)^{t_r}
=\eta_i(h),
\]
which implies that $\alpha_i$ stabilizes the discrete part
of $\iota(F)$ of degree at most $i$, and also that
$\alpha_i^{-1}h\in\ker(\eta_i)$.
Since $\alpha_i$ also acts trivially on the discrete components
of $\iota(F)$ of degree greater than $i$ (because so do the
$\zeta_\ell$, and therefore so too do the $\zeta_\ell^{t_\ell}$,
as can be seen from the expression of such translations --
and their 1-parameter subgroups -- from
\cite[Theorem 3.15]{CGSS-doucos}), we conclude that $\alpha_i$
stabilizes $\iota(F)$. This proves \eqref{eq:L'-match}.
\end{proof}

We are almost ready to prove the main result. We just need the following technical lemma which follows from \cite[Lemma B.2]{CGSS-doucos}.

\begin{lemma}\label{lem:closed-submanifolds}
Let $G$ be an \textsc{lch}\footnote{This means \emph{locally-compact, Hausdorff, second-countable topological space}, see \cite[Definition 2.1]{CGSS-doucos}.} group and let $\Gamma\le G$ be a closed subgroup. Then, for any closed subgroup $H\le G$ such that $\Gamma\le H$, we have that $H/\Gamma$ is closed inside $G/\Gamma$.
\end{lemma}

\begin{proof}
Let us apply \cite[Lemma B.2]{CGSS-doucos} with the pair $(X,G):=(G,\Gamma)$ and the action being right multiplication by elements of $\Gamma$. Since $G$ is a metric space, to check that $\{(g,g\gamma):g\in G,\gamma\in \Gamma\}$ is closed it suffices to show that for any convergent sequence $(g_n,g_n\gamma_n)\to (g,t)$ we have $t=g\gamma'$ for some $\gamma'\in \Gamma$. To see this claim, note that since $g_n\to g$ we have $g_n^{-1}g_n\gamma_n=\gamma_n\to g^{-1}t$, and since $\Gamma$ is closed we have $g^{-1}t\in \Gamma$, so the claim follows with $\gamma':=g^{-1}t$. Now, to prove that $H/\Gamma$ is closed, let $h_n\Gamma\to s\Gamma$ be a convergent sequence where $h_n\in H$. Using that $G/\Gamma$ is a metric space, it suffices to prove that $s\Gamma\in H/\Gamma$. By \cite[Lemma B.2]{CGSS-doucos} there exists $r_n\in \Gamma$ such that $h_nr_n\to s$. But as $\Gamma\le H$ and $H$ is closed we have that $s\in H$ and the result follows.
\end{proof}

\begin{proposition}[Lie parametrization for ported nilspaces]\label{prop:Lieparamgam}
Let $\ns=F/\Gamma$ be a $k$-step \textsc{cfr} ported nilspace. Let $F_{\mb{R}}$ be the connected closure of $F$ with  embedding $\iota:F\to F_\mb{R}$. Let $L$ be the Zariski closure of $\wh{\iota}(\Gamma)$ in $\tran(F_{\mb{R}})$ and let $\beta:F_\mb{R}\to L$ be the  Lie parametrization of $F_\mb{R}$. Then $G:=\beta(\iota(F))$ is a closed subgroup of $L$ which includes $\wh{\iota}(\Gamma)$, and such that $G_i:=G\cap \tran_i(F_\mb{R})$ is rational\footnote{Note that we need this rationality in order for $(G/\Gamma,(G_i)_i)$ to be a proper filtered nilmanifold.} relative to $\wh{\iota}(\Gamma)$. Moreover $\beta\co\iota$ is a nilspace isomorphism from $F$ to the group nilspace associated with $(G,(G_i)_{i\geq 0})$.
\end{proposition}

\begin{proof}
Recall that for each $x\in F_{\mb{R}}$ we define $\beta(x)$ to be the unique element of $L=\wt{\wh{\iota}(\Gamma)}$ such that $\beta(x) (0)=x$.

The inclusion $\beta(\iota(F))\supset \wh{\iota}(\Gamma)$ is clear, since for any $\gamma\in \Gamma$ viewed as a translation in $\tran(F_{\mb{R}})$, this $\gamma$ must be the element $\beta(\gamma(0_F))$.

We now turn to the main claim in the proposition, namely that $\beta(\iota(F))$ is a subgroup of $L$. Recall that $\stab_L(\iota(F)):=\{h\in L: h(\iota(F))=\iota(F)\}$ is the set-wise stabilizer of $\iota(F)$ in $L$. By basic group theory, this is a subgroup of $L$. We claim that
\begin{equation}\label{eq:invprop}
\beta(\iota(F))=\stab_L(\iota(F)).
\end{equation}
It is clear that $\stab_L(\iota(F))\subset \beta(\iota(F))$, since if $h\in \stab_L(\iota(F))$ then, letting $x=h(0_{F_\mb{R}})$ (which is in $\iota(F)$ since $0_{F_\mb{R}}$ is), we have $h=\beta(x)\in \beta(\iota(F))$.

To prove that $\beta(\iota(F))\subset \stab_L(\iota(F))$, the idea is to use  \eqref{eq:L'-match} iteratively to obtain a factorization of the following kind:
\begin{equation}\label{eq:L'factorization}
\forall\, h \in \beta(\iota(F)),\quad h = \alpha_1 \alpha_2 \cdots \alpha_k\text{ where }\alpha_i\in \stab_L(\iota(F))\cap \tran_i(F).
\end{equation}
Since $\stab_L(\iota(F))$ is a subgroup, it follows from this that $h\in \stab_L(\iota(F))$, which proves the desired inclusion $\beta(\iota(F))\subset \stab_L(\iota(F))$.

To deduce \eqref{eq:L'factorization} from \eqref{eq:L'-match}, we start with any $h\in \beta(\iota(F))$, thus $h\in L$ and $h(0)\in \iota(F)$. Then $\eta_1(h)$ is a constant shift on $\pi_1(F_{\mb{R}})$ which sends $0$ into $\pi_1(\iota(F))$, so it must act on discrete components by integer shifts. Therefore $\eta_1(h)(\iota(F_1))=\iota(F_1)$. Applying \eqref{eq:L'-match}, we find $\alpha_1\in \stab_L(\iota(F))\cap\tran_1(F)$ such that $\alpha_1^{-1}h\in \ker(\eta_1)$. Moreover, as $\alpha_1,h$ are both in $L$, so is $\alpha_1^{-1}h$ and so this is pure (by Lemma \ref{lem:PureZariski}), whence $\alpha_1^{-1}h\in L\cap \tran_2(F_{\mb{R}})$. Furthermore, as $h(0)\in \iota(F)$ and $\alpha_1^{-1}$ stabilizes $\iota(F)$, we have also $\alpha_1^{-1}h(0)\in \iota(F)$, and we can now run the same argument with $\alpha_1^{-1}h$. For general $i$, we start with $h_i:=\alpha_{i-1}^{-1}\cdots\alpha_1^{-1}h\in L$ which maps $0$ into $\iota(F)$, and such that $\eta_{i-1}(h_i)=\id$. Then by pureness of every element in $L$ (Lemma \ref{lem:PureZariski}), we have $h_i\in L_i$ and in particular $\eta_i(h_i)$ is a constant shift. Moreover $\eta_i(h_i)(\pi_i(0))\in \pi_i(\iota(F))$, so $\eta_i(h_i)$ must act by integer-shifts in discrete components, and this implies that $\eta_i(h_i)(\iota(F_i))=\iota(F_i)$. Now, by \eqref{eq:L'-match}, there is $\alpha_i\in \stab_L(\iota(F))$ such that $\alpha_i^{-1}h_i\in \ker\eta_i$. Since both $\alpha_i^{-1}$ and $h_i$ are in $L$, so is $\alpha_i^{-1}h_i$, whence $\alpha_i^{-1}h_i$ is pure. Thus, defining $h_{i+1}:=\alpha_i^{-1}h_i$, we have $h_{i+1}\in L\cap \ker(\eta_i)$. By pureness, we have that $h_{i+1}\in L_{i+1}$. Moreover, since $h_i$ sends $0$ into $\iota(F)$ and $\alpha_i^{-1}$ stabilizes $\iota(F)$, we have $\alpha_i^{-1}h_i(0)\in \iota(F)$, so we can continue the process. This proves  \eqref{eq:L'factorization}.

Note that then $\stab_L(\iota(F))$ is a closed subgroup, because $\iota(F)$ is a closed subset of $F_\mb{R}$ and $\beta:F_\mb{R}\to L$ is a homeomorphism, which maps $\iota(F)$ onto this subgroup by \eqref{eq:L'factorization}.

To see the rationality of each group $G_i$, i.e.\ that $L_i\cap\wh{\iota}(\Gamma)$ is cocompact in $G_i$, note that $G_i$ is a closed subgroup of $L_i$, that $L_i\cap\wh{\iota}(\Gamma)$ is cocompact in $L_i$ by rationality of $L_i$, and that $L_i\cap\wh{\iota}(\Gamma)\subset L_i\cap G=G_i$. Therefore, by Lemma \ref{lem:closed-submanifolds}, the quotient space $G_i/(L_i\cap\wh{\iota}(\Gamma))$ is a closed subspace of the compact space $L_i/(L_i\cap\wh{\iota}(\Gamma))$, and is therefore compact as required.

Finally, we prove that $\beta\co\iota$ is a nilspace isomorphism. We already have by Theorem \ref{thm:flierep} that $\beta$ is a nilspace isomorphism from $F_\mb{R}$ to the group nilspace $L:=\wt{\iota(\Gamma)}$. Clearly $\beta\co\iota$ is a bijection because $\beta$ is and $\iota$ is a bijection between $F$ and $\iota(F)$. Thus, we only need to prove that $\beta\co\iota$ and its inverse are morphisms. Both $\iota$ and $\beta$ are morphisms (the latter by Theorem \ref{thm:flierep}), hence so is $\beta\co\iota$. To prove that the inverse is a morphism, consider the obvious map $\nu:\iota(F)\to F$ such that $\nu\co\iota=\id_F$. Then $(\beta\co\iota)^{-1}=\nu\co\beta^{-1}|_{\beta(\iota(F))}$. As these maps are both morphisms, the result follows.
\end{proof}
\noindent Let us give a first application of the Lie parametrization for ported nilspaces.

\begin{proposition}\label{prop:nicerep=>coset}
Every ported nilspace is a coset nilspace.
\end{proposition}

\begin{proof}
By Proposition \ref{prop:Lieparamgam}, the Lie parametrization $\beta$ yields a nilspace isomorphism $\beta\co\iota:F\to G$ where $(G,(G_i)_i)$ is as in that proposition. Moreover, arguing as in the proof of Corollary \ref{cor:Lie-param-conn}, we have $(\beta\co\iota)(\gamma(x))=\gamma(\beta\co\iota)(x)$. Hence $\beta\co\iota$ is equivariant and thus it induces an isomorphism between $F/\Gamma$ and $\wh{\iota}(\Gamma)\backslash G \cong  G/\wh{\iota}(\Gamma)$.\end{proof}

\begin{remark}
The proofs of Propositions  \ref{prop:Lieparamgam} and \ref{prop:nicerep=>coset} also give us that $\ns$ is a coset nilspace in which the ambient group $G$ is a nilpotent Lie group such that its factor by the connected component of the identity is a free abelian group (indeed this abelian group corresponds to the discrete part of $F$ and is free by the orthogonality condition).
\end{remark}

To close this subsection, we establish an inverse of the Lie parametrization map.

\begin{proposition}\label{prop:invLieparam}
Let $L$ be a Lie group with a degree-$k$ filtration $(L_i)_{i\geq 0}$ of closed subgroups, such that $L_i/L_{i+1}$ is torsion-free and compactly generated for every $i\in [k]$. Then $L$, as a group nilspace, is isomorphic to a $k$-step free nilspace.
\end{proposition}
\begin{proof}
Every quotient $L_i/L_{i+1}$ is a torsion-free compactly-generated abelian Lie group, hence\footnote{A compactly generated abelian Lie group is isomorphic to $\mb{R}^c\times\mb{Z}^d\times K$, where $K$ is a compact abelian Lie group, by \cite[Theorem 7.57(ii)]{H&M-Cpct}. If the group is torsion-free, then \(K\) is trivial.} of the form $\mb{R}^{c_i}\times\mb{Z}^{d_i}$ for some non-negative integers $c_i,d_i$. The $i$-th structure group of the group nilspace $L$ is $L_i/L_{i+1}$, and then by applying \cite[Theorem 4.1]{CGSS-doucos} iteratively starting from the free 1-step nilspace $L/L_2$, we deduce inductively for each $i\in [2,k]$ that $L/L_{i+1}$ (the $i$-step nilspace factor of the group nilspace $L$) is isomorphic to the product nilspace of $\mc{D}_i(L_i/L_{i+1})$ with the free nilspace $L/L_i$. We thus conclude that the group nilspace $L$ is isomorphic to the free nilspace $\prod_{i\in [k]} \mc{D}_i(L_i/L_{i+1})$.
\end{proof}

\subsection{Extending $k$-step universal nilspaces to $(k+1)$-step universal nilspaces}\label{subsec:thm1}\hfill\\
\noindent We prove Theorem \ref{thm1} in this subsection. To do so, we will combine the tools developed in the previous subsections with some consequences of classical rigidity results in nilpotent Lie groups. One such result is the following version of a theorem originally due to Mal'cev \cite[Theorem 5]{Mal'cev:nil}.
\begin{theorem}\label{thm:rigidity1}
Let $G$ and $G'$ be connected simply-connected nilpotent Lie groups and let $\Gamma$ and $\Gamma'$ be discrete cocompact subgroups of $G$ and $G'$ respectively. Then any homomorphism $\pi:\Gamma\to \Gamma'$ can be extended uniquely to a continuous homomorphism $\wt{\pi}:G\to G'$.
\end{theorem}
\begin{proof}
See \cite[Theorem 5]{Mal'cev:nil}, or Theorem 2.11 in Chapter II of the book \cite{Rag}. 
\end{proof}
\noindent The following associated lemma will be useful.

\begin{lemma}\label{lem:surj&ker}
Let $G$, $G'$, $\Gamma$, $\Gamma'$, $\pi$, $\wt{\pi}$  be as in Theorem \ref{thm:rigidity1}. If $\pi$ is surjective, then so is $\wt{\pi}$. Moreover $\ker\wt{\pi}=\wt{\ker\pi}$.   
\end{lemma}
\begin{proof}
The surjectivity claim can be established by an inspection of the proof of Theorem \ref{thm:rigidity1} (e.g.\ in \cite[p. 33, proof of Theorem 2.11]{Rag}). To see the second claim, put $K=\ker\wt{\pi}$. Since $\pi(\Gamma)=\Gamma'$, the fiber over the identity coset of the induced map $G/\Gamma\to G'/\Gamma'$ is naturally homeomorphic\footnote{More precisely, considering the  induced map $\overline\pi:G/\Gamma\to G'/\Gamma'$,
 $g\Gamma\mapsto\wt\pi(g)\Gamma'$ (well-defined because $\wt\pi(\Gamma)=\pi(\Gamma)=\Gamma'$), the identity coset in $G'/\Gamma'$ is $1_{G'}\Gamma'=\Gamma'$, and its fiber is $
\overline\pi^{-1}(\Gamma')
=\{g\Gamma:\widetilde\pi(g)\in\Gamma'\}$. 
The claimed homeomorphism is the map $K/\ker\pi\to\overline\pi^{-1}(\Gamma')$, $h\ker\pi\mapsto h\Gamma$. To see the bijectivity, note that if $g\Gamma$ lies in the fiber, surjectivity of $\pi:\Gamma\to\Gamma'$ yields $\gamma\in\Gamma$ with $
\pi(\gamma)=\widetilde\pi(g)$. Then $h=g\gamma^{-1}$ belongs to $K$, and $h\Gamma=g\Gamma$. To see injectivity, note that $h_1,h_2\in K$ determine the same coset precisely when $h_2^{-1}h_1\in\Gamma\cap K=\ker\pi$. The fact that this bijective map is also a homeomorphism is left to the reader.} to $K/\ker(\pi)$. This fiber is closed in the compact space $G/\Gamma$, so $\ker(\pi)$ is a cocompact lattice in $K$. Since $K$ is connected (as verified below), and hence is a connected simply-connected nilpotent Lie group, its cocompact lattice $\ker(\pi)$ has Zariski closure $K$. Thus $\ker\wt{\pi}\subset \wt{\ker(\pi)}$. To prove the opposite inclusion, note that $\ker\wt{\pi}$ is connected and closed. Indeed, the exponential map is a homeomorphism and, using that $\exp\co\ud\tilde{\pi}=\tilde{\pi}\co \exp$, we have that $\ker\wt{\pi}=\exp(\ker\ud\wt{\pi})$, and clearly $\ker\ud\wt{\pi}$ is connected and closed. Moreover $\ker\wt{\pi}$ also includes $\ker\pi$, so it must include $\wt{\ker\pi}$ by minimality of the latter.
\end{proof}

\begin{lemma}\label{lem:G-tor-free}
Let $(\Gamma,\Gamma_\bullet)$ be a discrete finitely-generated filtered group of degree $k$. Suppose that for all $i\in[k]$ the group $\Gamma_i/\Gamma_{i+1}$ is torsion-free. Let $\wt{\Gamma}$ be a connected simply-connected Lie group including $\Gamma$ as a discrete co-compact subgroup,\footnote{As provided by Mal'cev's existence theorem; see \cite[Theorem 2.18]{Rag}.} let $G\ge \Gamma$ be a closed subgroup of $\wt{\Gamma}$ and for $i\in[k]$ let $G_i:= G\cap \wt{\Gamma_i}$. Then, for all $i\in[k]$ the group $G_i\wt{\Gamma_{i+1}}$ is a closed subgroup of $\wt{\Gamma_{i}}$ and $G_i/G_{i+1}$ is torsion-free.
\end{lemma}

\begin{proof}
First note that the assumptions imply that $\Gamma$ is torsion-free. Indeed, using the 3-rd isomorphism theorem we have that $(\Gamma_i/\Gamma_k)/(\Gamma_{i+1}/\Gamma_k)\cong \Gamma_i/\Gamma_{i+1}$ is torsion-free for each $i\in [k-1]$, so by induction $\Gamma/\Gamma_k$ is torsion-free. But then $\Gamma$ is a central extension of the torsion-free group $\Gamma/\Gamma_k$ by the torsion-free group $\Gamma_k$, hence $\Gamma$ is torsion-free. This justifies our use of Mal'cev's theorem providing the group $\wt{\Gamma}$ (this theorem requires $\Gamma$ to be torsion-free).

Now fix any $i\in[k]$ and consider the quotient homomorphism $\pi:\Gamma_i\to \Gamma_i/\Gamma_{i+1}$. By Theorem \ref{thm:rigidity1}, this extends (uniquely) to a continuous surjective homomorphism $\wt{\pi}:\wt{\Gamma_i}\to\wt{\Gamma_i/\Gamma_{i+1}}$. By Lemma \ref{lem:surj&ker}, we have that $\ker(\wt{\pi})=\wt{\ker(\pi)}$ and hence $\wt{\Gamma_i/\Gamma_{i+1}}\cong \wt{\Gamma_i}/\wt{\Gamma_{i+1}}$. Since $\wt{\Gamma_i/\Gamma_{i+1}}$ is torsion-free, then so is $\wt{\Gamma_i}/\wt{\Gamma_{i+1}}$. Note that, as $\wt{\Gamma_{i+1}}/\Gamma_{i+1}$ is compact, there exists a compact set $K\subset \wt{\Gamma_{i+1}}$ such that $\wt{\Gamma_{i+1}} = K\Gamma_{i+1}$. In particular $KG_i=K\Gamma_{i+1}G_i=\wt{\Gamma_{i+1}}G_i$. Now letting $\pi_{G_i}:\wt{\Gamma_{i}}\to \wt{\Gamma_{i}}/G_i$ be the quotient map, note that the image of $K\subset \wt{\Gamma_{i}}$ is compact, hence closed in $\wt{\Gamma_{i}}/G_i$. But this image is simply $KG_i/G_i=\wt{\Gamma_{i+1}}G_i/G_i$ which is thus closed. Therefore, its pre-image under $\pi_{G_i}$ is also closed, and this is precisely $\wt{\Gamma_{i+1}}G_i$. By the second isomorphism theorem, note that  $G_i\wt{\Gamma_{i+1}}/\wt{\Gamma_{i+1}}\cong G_i/G_{i+1}$. As $G_i\wt{\Gamma_{i+1}}/\wt{\Gamma_{i+1}}\le \wt{\Gamma_i}/\wt{\Gamma_{i+1}}$ and the latter is torsion-free, we have that $G_i/G_{i+1}$ is torsion-free and the result follows.
\end{proof}

We can now proceed to the main proof in this subsection.

\begin{proof}[Proof of Theorem \ref{thm1}]
By assumption, we have $\ns=F/\Gamma$ where $\Gamma$ is a porting subgroup of $\tran(F)$ which is a free graded $k$-step nilpotent group and such that $\Gamma_i=\Gamma\cap \tran_i(F)$ for all $i\in[k]$. Let $F=C\times D$ be the decomposition of $F$ into its continuous part $C$ and its discrete part $D$ as per Definition \ref{def:free-nil-intro}, let $t\in \mb{Z}_{\geq 0}$ be such that $D=\mb{Z}^t$ as a set (thus $t=a_1+\cdots+a_k$), and let $\wh{p_D}:\Gamma\to\tran(D)$ be the homomorphism given by Lemma \ref{lem:Dpartproj}. By the orthogonality assumption, the image of $\wh{p_D}$ is a subgroup $A$ of $\mb{Z}^t$ generated by the images under $\wh{p_D}$ of the elementary generators in $\Gamma$. Note that by assumption on the elementary generators, we can suppose that $A$ has full rank in $\mb{Z}^t$, in particular the Zariski closure of $A$ is the same as that of $\mb{Z}^t$, namely $\mb{R}^t$. By Theorem \ref{thm:rigidity1}, there is a unique continuous surjective homomorphism $\wt{\wh{p_D}}:\wt{\Gamma}\to \mb{R}^t$ that extends $\wh{p_D}$.  Letting $G$ be the Lie group parametrizing $F$ as per Proposition \ref{prop:Lieparamgam}, we have
\begin{equation}
G=\wt{\wh{p_D}}^{-1}(\mb{Z}^t).
\end{equation}
Indeed the inclusion $G\subset\wt{\wh{p_D}}^{-1}(\mb{Z}^t)$ follows immediately from the fact that $\wt{\wh{p_D}}(G)=\mb{Z}^t$, which in turn holds since, by Proposition \ref{prop:Lieparamgam}, the free nilspace $F$ is the orbit of $0$ in $F_\mb{R}$ under the free action of $G$. The opposite inclusion $G\supset\wt{\wh{p_D}}^{-1}(\mb{Z}^t)$ follows from the fact that, inside $\wt{\wh{p_D}}^{-1}(\mb{R}^t)=\wt{\Gamma}$, if a translation is in $\wt{\wh{p_D}}^{-1}(\mb{Z}^t)$ then it is in the setwise stabilizer of $F$, which is precisely the group $G$ as established in \eqref{eq:invprop}. 

Let $\Gamma'$ be the free graded $(k+1)$-step nilpotent group with same graded generators as $\Gamma$. Thus $\Gamma'$ has no generators of degree $k+1$ and is obtained from $\Gamma$ just by deleting, from the set of relations defining $\Gamma$, the relations involving commutators of degree exactly $k+1$, leaving only the relations involving commutators of degree at least $k+2$. This definition of $\Gamma'$ immediately yields a filtered\footnote{Meaning that $\pi(\Gamma'_i)=\Gamma_i$ for all $i\in[k+1]$.} surjective homomorphism $\pi:\Gamma'\to \Gamma$, namely $\pi$ is the identity map on the set of generators, extended to a homomorphism the natural way (see e.g.\ \cite[Corollary 1.1.3]{MKS76}).  Let us specify also the filtration of degree $k+1$ on $\Gamma'$ that we shall use, as this will determine a $(k+1)$-step nilspace structure. By assumption on $\ns$ (i.e.\ Definition \ref{def:univCFRns}), the given filtration on $\Gamma$ has $i$-th term $\Gamma_i=\Gamma\cap \tran_i(F)$. We then define the degree $k+1$ filtration $(\Gamma'_i)_{i\geq 0}$ where $\Gamma'_i=\pi^{-1}(\Gamma_i)$ for $i\in [0,k+1]$ and $\Gamma_i'=\{1\}$ for $i\geq k+2$. In particular $\Gamma'_{k+1}=\ker(\pi)$ is the subgroup of $\Gamma'$ generated by basic commutators of degree $k+1$ in the generators (see \cite[\S 2.3]{MKS76}).

Now we make a key use of Mal'cev's theorem \cite[Theorem 6]{Mal'cev} (see also \cite[Theorem 2.18, page 40]{Rag}). Since $\Gamma'$ is torsion-free and finitely generated, by this theorem there is a connected simply-connected Lie group $L'$, which is compactly generated (since it is the Mal'cev completion of a finitely generated group $\Gamma'$), such that $\Gamma'$ is a discrete co-compact subgroup of $L'$. Moreover, define $L'_i:=\wt{\Gamma'_i}$.
These form a degree-$(k+1)$ filtration of $L'$, since $[L'_i,L'_j]
=\wt{[\Gamma'_i,\Gamma'_j]}
\subseteq L'_{i+j}$ by \cite[Theorem 6.5 (b)]{Wil}.
By the weighted Hall normal form, the quotients
$\Gamma'/\Gamma'_i$ and $\Gamma'_i/\Gamma'_{i+1}$
are torsion-free. Applying Lemma \ref{lem:surj&ker}
to the corresponding quotient homomorphisms gives
$\Gamma'\cap L'_i=\Gamma'_i$ and identifies
$L'_i/L'_{i+1}$ with the connected real Mal'cev
completion of $\Gamma'_i/\Gamma'_{i+1}$.
In particular,
$\Gamma'_iL'_{i+1}/L'_{i+1}$ is a discrete cocompact
lattice in $L'_i/L'_{i+1}$.
Also, $L'_i$ is rational relative to $\Gamma'$,
since $\Gamma'\cap L'_i=\Gamma'_i$ is cocompact
in $L'_i$.

By Theorem \ref{thm:rigidity1}, the homomorphism $\pi$ extends uniquely to a continuous surjective homomorphism $\wt{\pi}:L'\to \wt{\Gamma}$. Let $q=\wt{\wh{p_D}}\co\wt{\pi}$, a continuous surjective homomorphism $L'\to \mb{R}^t$ extending $
\wh{p_D}\co\pi: \Gamma'\to A$. (Note that by uniqueness in Theorem \ref{thm:rigidity1}, the homomorphism $q$ is the unique such extension of $\wh{p_D}\co\pi$). Using Lemma \ref{lem:surj&ker}, we have
\begin{equation}\label{eq:kerext1}
\ker(\wt{\pi})=\wt{\ker(\pi)}=\wt{\Gamma'_{k+1}}=L'_{k+1}.
\end{equation}

We can now define the desired nilspace $\nss$. Let $G':= q^{-1}(\mb{Z}^t)=\wt{\pi}^{-1}(G)\leq L'$, and equip $G'$ with the natural degree-$(k+1)$ filtration $G'_\bullet=(G'_i:=L'_i\cap G')_{i\geq 0}$. For every $i\in[k+1]$, we have $\wt{\pi}(L'_i)=\wt{\Gamma_i}$, since $\pi(\Gamma'_i)=\Gamma_i$. By Proposition \ref{prop:Lieparamgam} and Lemma \ref{lem:PureZariski}, $G_i=G\cap\wt{\Gamma_i}$. Thus $G'=\wt{\pi}^{-1}(G)$ implies
$\wt{\pi}(G'_i)=G_i$ for every $i\in[k+1]$. Note that $\Gamma'\subseteq G'$, since $q(\Gamma')=A\subseteq\mb{Z}^t$. For every $i\in[k+1]$,
the subgroup $\Gamma'_i$ is finitely generated and is
a cocompact lattice in $L'_i=\wt{\Gamma'_i}$.
Since $G'_i$ is a closed subgroup of $L'_i$ containing
$\Gamma'_i$, the space $G'_i/\Gamma'_i$ is a closed
subspace of the compact space $L'_i/\Gamma'_i$, and is
therefore compact. Thus there is a compact set
$K_i\subseteq G'_i$ such that $G'_i=K_i\Gamma'_i$.
Together with a finite generating set for $\Gamma'_i$,
the set $K_i$ generates $G'_i$. Hence every $G'_i$,
and consequently every quotient $G'_i/G'_{i+1}$,
is compactly generated. We can thus define the coset nilspace
\begin{equation}
\nss:=G'/\Gamma'\text{ equipped with the Host--Kra cube-sets } \cu^n(G'_\bullet), n\geq 0.
\end{equation}
We first claim that
\begin{equation}\label{eq:kerext2}
\nss_k\cong\ns.
\end{equation}
To see this, first recall that in the $(k+1)$-step coset nilspace $G'/\Gamma'$, the $k$-step factor $\nss_k$ is $(G'/G'_{k+1})/[(\Gamma'G'_{k+1})/G'_{k+1}]$ (see \cite{Cand:Notes1}). By \eqref{eq:kerext1} we have
\begin{equation}\label{eq:kerext3}
G_{k+1}'=L_{k+1}'\cap G'=\wt{\pi}^{-1}(\{1\})\cap \wt{\pi}^{-1}(G)=\wt{\pi}^{-1}(\{1\}\cap G)=\wt{\pi}^{-1}(\{1\})=\ker(\wt{\pi}).
\end{equation}
It follows that $G'/G'_{k+1}\cong G$, since $G$ is the image of $G'$ under the continuous homomorphism $\wt{\pi}|_{G'}$, which has kernel $G'_{k+1}$ by \eqref{eq:kerext3}. Also, by standard results we have $(\Gamma'G'_{k+1})/G'_{k+1}\cong \Gamma'/(\Gamma'\cap G'_{k+1})=\Gamma'/\Gamma'_{k+1}$, and this is isomorphic to $\Gamma$, since (as noted above) $\Gamma'_{k+1}$ is the subgroup generated by commutators of degree exactly $k+1$ in the generators. This completes the proof of \eqref{eq:kerext2}.

Next, in order to obtain a ported presentation $\nss\cong F'/\Gamma'$ as needed in Definition \ref{def:univCFRns}, we apply Proposition \ref{prop:invLieparam}. We can use this result because the quotient groups $\Gamma_i'/\Gamma_{i+1}'$ are torsion-free, see \cite[\S 5.11]{MKS76}. By Lemma \ref{lem:G-tor-free}, the quotients $G'_i/G'_{i+1}$ are also torsion-free and we have already seen that they are compactly generated. Thus, applying Proposition \ref{prop:invLieparam}, we obtain a $(k+1)$-step free nilspace $F'$ isomorphic to the group nilspace $(G',G'_\bullet)$. Moreover, by combining the isomorphisms $F'\to G'$, $G\to F$, and the nilspace fibration (surjective filtered-group homomorphism) $\wt{\pi}:G'\to G$, we obtain a nilspace fibration $F'\to F$, so by \cite[Lemma 8.9]{CGSS-doucos} we can view $F$ as the $k$-step direct-component (and $k$-step nilspace factor) of $F'$. This view will now enable us to see that the action of $\Gamma'$ on $F'$ has all the required properties. Under this action, right multiplication by an element
of $\Gamma'$ has height $i$ precisely when that element
belongs to $G'_i$. Hence the induced filtration is $\Gamma'\cap G'_i=\Gamma'\cap L'_i=\Gamma'_i$, which is the weighted free filtration required by Definition \ref{def:univCFRns}.

Firstly, we can see that $\Gamma'$ is orthogonal on $F'$. Indeed, this follows straightforwardly from the assumed orthogonality of $\Gamma$ on $F$ if we can prove that the degree $k+1$ component of $F'$ is connected (i.e.\ a power of $\mb{R}$), as then the action of $\Gamma'$ on the discrete part of $F'$ is the same as that of $\Gamma$ on $F$. But from construction of $F'$ (Proposition \ref{prop:invLieparam}), the degree $k+1$ component of $F'$ is connected because by \eqref{eq:kerext3} the group $G'_{k+1}$ is connected. Moreover, $\pi(\Gamma'_i)=\Gamma_i$ implies that
the discrete images of these filtration subgroups
coincide at every level. Thus the height-lifting
condition in the definition of orthogonality is
also inherited from $\Gamma$. This completes the proof of orthogonality of $\Gamma'$ on $F'$. Secondly, the action of $\Gamma'$ on $F'$ is free because by  Proposition \ref{prop:Lieparamgam}, the action of the whole of $G'$ on $F'$ is free, so in particular the subgroup $\Gamma'\leq G'$ acts freely on $F'$. Next, note that via the nilspace isomorphism $F'\to G'$ and the fact that $\Gamma'$ is acting on the group nilspace $G'$ by right multiplication, the action is fiber-transitive by \cite[Remark 5.9]{CGSS-doucos}. By \cite[Proposition 5.42]{CGSS-doucos}, it follows that the action of $\Gamma'$ is pure. That $\Gamma'$ is fiber-discrete and fiber-cocompact (recall \cite[Definition 1.4]{CGSS-doucos}) clearly follows in a similar way from these properties for $\Gamma'$ on $G'$. Indeed, for every $i\in[k+1]$, we must show that $\Gamma'_i G'_{i+1}/G'_{i+1}$ is a discrete cocompact
subgroup of $G'_i/G'_{i+1}$. The natural homomorphism 
$G'_i/G'_{i+1}\to L'_i/L'_{i+1}$ is continuous and injective, and the subgroup in question is the preimage of the discrete lattice
$\Gamma'_iL'_{i+1}/L'_{i+1}$ established above. It is therefore discrete. Its quotient is a continuous image of the compact
space $G'_i/\Gamma'_i$, so it is also cocompact. The result follows.
\end{proof}

To conclude this subsection, let us note that the \emph{converse} of Theorem \ref{thm1} holds.

\begin{lemma}\label{lem:k-1-univ-is-univ}
Let $\ns$ be a $k$-step universal nilspace. Then its $(k-1)$-step factor is also universal.
\end{lemma}

\begin{proof}
This follows by tracking carefully the conditions in Definition \ref{def:univCFRns}, we omit the details.
\end{proof}

\subsection{Universal nilspaces are toral splitting}\label{subsec:thm2}\hfill\\
Recall the notions of \emph{universal \textsc{cfr} nilspaces} from Definition \ref{def:univCFRns}  and toral-splitting nilspaces from Definition \ref{def:tor-split}. The main result in this subsection is the proof of Theorem \ref{thm2}, which we recall here for convenience.
\begin{theorem}\label{thm:thm2}
Every $k$-step universal \textsc{cfr} nilspace is $k$-toral splitting.
\end{theorem}
To prove this we shall use the following consequence of rigidity results. 

\begin{lemma}\label{lem:prodquotsplit}
Let $\ns$ be a coset nilspace of the form $L/\Gamma$, where $L$ is a connected simply-connected nilpotent Lie group with a filtration  $L_\bullet$ of degree $k$ consisting of connected simply-connected closed subgroups, and $\Gamma$ is a discrete cocompact subgroup of $L$, and for every $j\in [k]$ the subgroup $\Gamma_j=\Gamma\cap L_j$ is cocompact in $L_j$. Let $K^{(1)},K^{(2)}$ be filtered subgroups of $\Gamma$ such that there exists an isomorphism $\phi:\Gamma\to  K^{(1)}\times K^{(2)}$, and for $i=1,2$ let $\wt{K}^{(i)}$ be the Zariski closure of $K^{(i)}$ in $L$. Assume that, for all $j\in[k]$, the restriction map $\phi|_{\Gamma_j}:\Gamma_j\to K^{(1)}_j\times K^{(2)}_j$ is also an isomorphism.

Then there exists a group isomorphism $\varphi: L\to \wt{K}^{(1)}\times \wt{K}^{(2)}$. Moreover, for $i=1,2$, if $\wt{K}^{(i)}/K^{(i)}$ denotes the coset nilspace associated to the filtration $(\wt{K}^{(i)}_j)_{j\in[k]}$, we have a nilspace isomorphism $\ns\cong (\wt{K}^{(1)}/K^{(1)})\times (\wt{K}^{(2)}/K^{(2)})$.
\end{lemma}

\begin{proof}
For any $i=1,2$ and $j\in[k]$, since $K^{(i)}$ is co-compact in $\wt{K}^{(i)}$, we have that $K^{(1)}_j\times K^{(2)}_j$ is cocompact in $\wt{K}^{(1)}_j\times \wt{K}^{(2)}_j$. By \cite[p. 33 Theorem 2.11]{Rag}, there is a unique extension $\varphi:L_1\to \wt{K}^{(1)}_1\times \wt{K}^{(2)}_1$ of $\phi$. Moreover, by \cite[p. 34 Corollary 2]{Rag}, the map $\varphi$ is an isomorphism. Note that for every $j\in[k]$, the map $\phi|_{\Gamma_j}$ extends to a unique map $L_j\to \wt{\Gamma_j}\cong \wt{K}^{(1)}_j\times \wt{K}^{(2)}_j$. By uniqueness in \cite[p. 33 Theorem 2.11]{Rag}, we must have that such extension is precisely $\varphi|_{L_j}$. Hence $\varphi$ is a filtered group isomorphism between $(L,L_\bullet)$ and $(\wt{K}^{(1)}_1\times \wt{K}^{(2)}_1,\wt{K}^{(1)}_\bullet\times \wt{K}^{(2)}_\bullet)$ and therefore it induces a nilspace isomorphism between $\ns$ and $(\wt{K}^{(1)}/K^{(1)})\times (\wt{K}^{(2)}/K^{(2)})$ as claimed.
\end{proof}

\begin{proof}[Proof of Theorem \ref{thm:thm2}]
Let $\ns$ be a $k$-step universal \textsc{cfr} nilspace. Thus $\ns=F/\Gamma$ where $\Gamma$ is a porting subgroup of $\tran(F)$ which is also a free graded $k$-step nilpotent group. Let $\nss$ be a $\mc{D}_\ell(\mb{T})$-extension of $\ns$. We need to prove that this is a split extension. By an argument analogous to \cite[Lemma 4.4]{CGSS-p-hom} we may assume that $\ell=k$.\footnote{The proof is by induction on $k$, with the case $k=1$ being trivial by standard abelian Lie group theory. Assuming that Theorem \ref{thm:thm2} holds up to step $k-1$, we may repeat the arguments of \cite[Lemma 4.4]{CGSS-p-hom} with the nilspace $\mc{Q}_{p,k}$ therein replaced by a universal $k$-step nilspace (using that the $k-1$-step factor of such a nilspace is $k-1$-step universal, by Lemma \ref{lem:k-1-univ-is-univ}) to prove that a $\mc{D}_\ell(\mb{T})$ extension of a $k$-step universal nilspace splits.}

The first step consists in proving the following claim.
\begin{equation}\label{eq:ts1}
\text{For }F':=F\times\mc{D}_k(\mb{R}), \text{ there is }\Gamma'\leq \tran(F')\text{ so that } \nss\cong F'/\Gamma' \text{ and } \Gamma'\cong\Gamma\times \mb{Z}.
\end{equation}
To prove this, we start by applying Proposition \ref{prop:ext-as-quo-of-free}. This  tells us that $\nss\cong F'/\Gamma''$ where $\Gamma''$ is a fiber-transitive, fiber-discrete subgroup of $\tran(F')$ generated by the translation $(x,z)\mapsto (x,z+1)$ together with translations of the form $(x,z)\in F\times\mc{D}_k(\mb{R}) \mapsto (\gamma(x),z+\gamma'(x))$, where $\gamma\in \Gamma_j$ (for some $j\in [k]$) and $\gamma'\in \hom\big(F,\mc{D}_{k-j}(\mb{R})\big)$. In particular, there exists a surjective homomorphism $\Gamma''\to \Gamma$. We now use the assumption that $\Gamma$ is \emph{free} filtered of degree $k$ to let $s:\Gamma\to \Gamma''$ be a cross-section (an injective homomorphism such that composing it with the previous homomorphism gives the identity map) defined as follows: first for each of the finitely many generators $\gamma_1,\ldots,\gamma_r\in \Gamma$, we let $s(\gamma_i)\in \Gamma''$ be a translation $(x,z)\mapsto (\gamma_i(x),z+\gamma_i'(x))$ given by the above application of Proposition \ref{prop:ext-as-quo-of-free}, and then for a general element $\gamma=\prod_{j=1}^n\gamma_{i_j}^{a_j}$, we let $s(\gamma):=\prod_{j=1}^ns(\gamma_{i_j})^{a_j}$. It follows from the assumptions on $\Gamma$ that $s$ is thus a well-defined map, and it is then clear that $s$ is an injective homomorphism. More precisely, from group theory (see e.g.\ \cite[Corollary 1.1.2]{MKS76}), we know that for $s$ to be well-defined it suffices to ensure that the following property holds:
\begin{equation}\label{eq:ts2}
\text{For every relation $\gamma_{i_1}^{n_1}\cdots\gamma_{i_w}^{n_w}=\id$ in $\Gamma$, we have $s(\gamma_{i_1})^{n_1}\cdots s(\gamma_{i_w})^{n_w}=\id$.}
\end{equation}
By definition of the free filtered group $\Gamma$ of degree $k$, the only such relations are commutator expressions of degree at least $k+1$. Since $\ell\leq k$, the group $\tran(F')$ (of which $\Gamma''$ is a subgroup), is also of degree $k$, so the same commutator in the elements $s(\gamma_{i_j})$ must indeed be the identity. This proves \eqref{eq:ts2}.

Now let $Z$ be the subgroup of $\Gamma''$ generated by the map $\gamma^*:(x,z)\mapsto (x,z+1)$ (thus $Z\cong \mb{Z}$). Clearly $Z$ is in the center of $\tran(F')$. Let $\Gamma':=s(\Gamma)Z$. We claim that $\Gamma'$ satisfies \eqref{eq:ts1} and furthermore, letting its filtration be $\Gamma'_i:=\Gamma'\cap \tran_i(F')$, we have that
\begin{equation}\label{eq:tsextra}
\text{
$\Gamma_i'=s(\Gamma_i)Z$ and $\Gamma_i'\cong \Gamma_i\times \mb{Z}$ for $i\in[k]$.} 
\end{equation}
\noindent First, note that as $Z$ is central we have that $\Gamma'\le \tran(F')$. To prove that $\nss\cong F'/\Gamma'$, it suffices to see that $\Gamma'$ and $\Gamma''$ define equivalent filtrations (see \cite[Definition 5.15]{CGSS-doucos}). In turn, this follows from the fact that for any $\alpha\in \Gamma''$, which is of the form $(x,z)\mapsto(\gamma(x),z+\gamma'(x))$, we have that $\alpha^{-1}s(\gamma)\in \Gamma''$ is an element that only modifies the $\mc{D}_\ell(\mb{R})$ component of $F'$. By the fiber-transitive property and the definition of $\Gamma''$ given by Proposition \ref{prop:ext-as-quo-of-free}, we have that for any $(x,z)\in F'$ there must exist some $d=d_{\alpha,(x,z)}\in \mb{Z}$ such that $\alpha^{-1}s(\gamma)(x,z)=(x,z+d)$ and thus $\alpha(x,z)=s(\gamma)(x,z-d)$. Using this observation it is clear that $\Gamma''$ and $\Gamma'$ define equivalent filtrations and thus $\nss\cong F'/\Gamma'$. To prove the filtration condition, note that for any $s(\gamma)(\gamma^*)^n\in \Gamma'\cap \tran_i(F')$ we can project to $\Gamma$ and deduce that $\gamma\in\Gamma_i$ as claimed. Finally, if we prove that $s(\Gamma)\cap Z=\{\id\}$, then by the centrality of $Z$ in $\Gamma'$, we will have $\Gamma'\cong\Gamma\times\mb{Z}$ (and similarly for all filtration terms). Let us now prove that $s(\Gamma)\cap Z=\{\id\}$. Suppose for a contradiction that  $s(\gamma)=(\gamma^*)^n$ for some $\gamma\in \Gamma$ and some non-zero $n\in \mb{Z}$. Since $\Gamma$ is pure on $F$, if $\gamma\not=\id$ then $\gamma$ acts non-trivially on some component of $F$, by \cite[Proposition 5.42]{CGSS-doucos}. However, the equality $s(\gamma)=(\gamma^*)^n$ implies precisely that $\gamma$ acts trivially on $F$. Hence $\gamma=\id$, contradicting that $n\neq 0$. 

We now move on to the second main step in the proof. Note that the first step was just about $\Gamma$ and $\Gamma'$, but ultimately we want to prove that the nilspace $\nss$ itself splits. To this end, we use the Lie parametrization $\beta:F'\to \beta(\iota(F'))$ given by Proposition \ref{prop:Lieparamgam}. To apply it, we need to prove that $F'/\Gamma'$ is a porting presentation. To prove that $\Gamma'$ is pure, note that if $s(\gamma)(\gamma^*)^n\in \Gamma'$ is in $\ker(\eta_{i-1})$ for $i\in[k]$ then in particular $\gamma\in \ker(\eta_{i-1})$ (projecting to $F$) and by pureness of $\Gamma$ we have $\gamma\in \Gamma_i$. Hence, $s(\gamma)(\gamma^*)^n\in \Gamma_i'$. The orthogonality property follows immediately from that of $\Gamma$ as $F'$ is only adding a continuous component. Finally, as $\Gamma'$ is equivalent to $\Gamma''$ and the latter is a translation lattice, so is $\Gamma'$ (see \cite[\S 5.1]{CGSS-doucos}).

Now recall the notion of connected closure (see Definition \ref{def:conn-closure}) and recall that by Theorem \ref{thm:flierep} we have that $\beta(F'_{\mb{R}})$ is the Zariski closure of $\Gamma'$. Note that Theorem \ref{thm:flierep} gives also a filtration of the form $\wt{\Gamma'}\cap \tran_i(F'_{\mb{R}})$ which is thus a connected simply-connected filtration. The same argument applies to the filtered groups $\beta(F_{\mb{R}})$ and $Z=\mb{Z}$. As $\Gamma'_j\cong \Gamma_j\times \mb{Z}$ for all $j\in[k]$ (by \eqref{eq:tsextra}), by Lemma \ref{lem:prodquotsplit}, we have that $\Gamma'\cong \Gamma\times\mb{Z}$ and there exists a filtered group isomorphism $\varphi:\beta(F'_{\mb{R}})\to  \beta(F_{\mb{R}})\times \mb{R}$. In particular, we have an isomorphism $\beta(F'_{\mb{R}})\cong \beta(F_{\mb{R}})\times \mb{R}$ as group nilspaces.

Next, we show that the splitting of the connected group $\beta(F'_\mb{R})$ via $\varphi$ just proved implies a corresponding splitting
\begin{equation}\label{eq:ts3}
\beta(\iota(F'))\cong \beta(\iota(F)) \times \mb{R}
\end{equation}
via the same map $\varphi$. Indeed, we know, by Proposition \ref{prop:Lieparamgam}, that $\beta(\iota(F'))$ is a closed subgroup and it includes, as closed subgroups, both $\beta(\iota(F))$ and $\wt{Z}\cong \mb{R}$. To see that it includes $\beta(\iota(F))$  note that by \cite[p. 33, Theorem 2.11]{Rag} the homomorphism $s:\Gamma\to\Gamma'$ extends uniquely to a homomorphism $\wt{s}:\wt{\Gamma}\to\wt{\Gamma'}$, and by uniqueness $\wt{s}(h)=\varphi^{-1}(h,0)$ for every $h\in \wt{\Gamma}$. The fact that $\wt{Z}$ is a closed subgroup of $\wt{\Gamma'}$ is shown similarly. Now, since $\wt{\Gamma}$ and $Z$ are already known to have trivial intersection in $\wt{\Gamma'}$, we also have 
\[
\beta(\iota(F)) \cap \wt{Z} = \{\id\}.
\]
As $\wt{Z}$ is central in $\wt{\Gamma'}$, to prove \eqref{eq:ts3} it now suffices to show that $\beta(\iota(F'))\subset \beta(\iota(F)) \cdot \wt{Z}$. This can again be deduced from the transitivity of the action of $\beta(\iota(F'))$ on $\iota(F')$. Therefore $\varphi|_{\beta(\iota(F'))}:\beta(\iota(F'))\to \beta(\iota(F))\times\mb{R}$ is an isomorphism. By similar arguments, we have analogues of \eqref{eq:ts3} for all other terms in the filtration on $\beta(\iota(F'))$.

By Proposition \ref{prop:nicerep=>coset}, the nilspace $\nss$ is isomorphic (as a compact nilspace) to $\beta(\iota(F'))/\Gamma'$. But we already know that the isomorphism $\varphi$, restricted to $\beta(\iota(F'))$ and $\Gamma'$, induces isomorphisms $\beta(\iota(F'))\cong \beta(\iota(F))\times\mb{R}$ and $\Gamma'\cong \Gamma\times \mb{Z}$. Thus $\nss\cong \ns \times \mc{D}_k(\mb{T})$ (note that by Theorem \ref{thm:flierep} the filtration corresponding to $Z$ is precisely $\mc{D}_k(Z)$ and thus the corresponding filtration on $\wt{Z}\cong \mb{R}$ is also $\mc{D}_k(\mb{R})$). To conclude that this extension is precisely the original extension, note that we may apply Proposition \ref{prop:Lieparamgam} to $F/\Gamma$ and by uniqueness of the Mal'cev completion, the corresponding Lie parametrizations of $\nss$ and $\ns$ are precisely the ones given by the projection to the first coordinate $\Gamma'\cong \Gamma\times\mb{Z}\to\Gamma$. Therefore, via these Lie parametrizations, the original extension $\nss\to\ns$ corresponds to $\ns\times\mc{D}_k(\mb{T})\to\ns$ via the isomorphism $\varphi$.
\end{proof}

\section{An inverse theorem for all finite abelian groups in terms of nilmanifolds}\label{sec:projinv} 
\noindent Given a map $f$ from $\mb{Z}^n$ into a set $X$, if there is an integer $M$ such that for every $i\in [n]$ we have $f(x+Me_i)=f(x)$ for all $x\in\mb{Z}^n$, then we say that $f$ is $M$-\emph{periodic} (or just \emph{periodic}, if we do not need to specify the period $M$).

In this section we prove Theorem \ref{thm:projinv}. To motivate the ingredients that we shall use, let us consider the situation we obtain when we combine the general inverse theorem \cite[Theorem 1.6]{CSinverse} with Theorem \ref{thm:cfr-are-factor-of-nilmanifolds}. Given the initial finite abelian group $\ab$, the general inverse theorem gives us a nilspace morphism $\varphi:\ab\to\ns$ for some bounded-complexity \textsc{cfr} nilspace $\ns$. Then Theorem \ref{thm:cfr-are-factor-of-nilmanifolds} provides a fibration $\psi:\nss\to\ns$ where $\nss$ is a filtered nilmanifold. If we could lift the morphism $\varphi$ directly to a morphism $\ab\to\nss$, then the Jamneshan--Tao conjecture would be proved. Currently, we are unable to do this in general. However, by allowing $\ab$ to be extended to a boundedly-larger group $\ab'$, we are able to obtain a morphism $\varphi':\ab'\to \nss$, and this is what leads to the notion of projected nilsequence and to Theorem \ref{thm:projinv}. The main task is thus to obtain such a morphism $\varphi'$, and this will be achieved with Theorem \ref{thm:CFRpermorphlift} below. The proof of Theorem \ref{thm:projinv} will then be readily completed in Subsection \ref{subsec:PfInvThm}. 

To obtain the desired morphism $\varphi':\ab'\to\nss$, we first view $\ab$ the standard way as the image of a surjective homomorphism on $\mb{Z}^n$ (where $n$ is the rank of $\ab$), and note that this homomorphism composed with $\varphi:\ab\to \ns$ yields a nilspace morphism $f:\mc{D}_1(\mb{Z}^n)\to \ns$ which has the additional property of being $M$-periodic, where $M$ is the exponent of $\ab$. Hence, it suffices to show that $f$ can be lifted to another periodic morphism $f':\mc{D}_1(\mb{Z}^n)\to \nss$, of period $M'$ being boundedly larger than $M$, as this periodicity would immediately imply that $f'$ factors through a morphism $\varphi':\ab'\to\nss$ of the desired type (where $\ab'=\mb{Z}^n/(M'\mb{Z}^n)$). This is precisely what we establish in the following theorem, which is the main result of this section.

\begin{theorem}\label{thm:CFRpermorphlift}
Let $\ns$ be a $k$-step \textsc{cfr} nilspace, let $\nss$ be a universal $k$-step \textsc{cfr} nilspace, and let $\psi:\nss\to\ns$ be a fibration. Then, for any $M$-periodic morphism $f\in \hom(\mc{D}_1(\mb{Z}^n),\ns)$, there is an $M'$-periodic morphism $f'\in \hom(\mc{D}_1(\mb{Z}^n),\nss)$ such that $\psi\co f'=f$, where $M'=CM^C$ for $C=O_{k,\psi}(1)$.
\end{theorem}

\begin{remark}
When we write that an implicit variable depends on a map, e.g.\ on $\psi:\nss\to\ns$, we mean that it depends on the map itself and its domain and image.
\end{remark}

We shall prove Theorem \ref{thm:CFRpermorphlift} by induction on $k$, using the following two results.
\begin{proposition}[Lifting multivariable periodic polynomials]\label{prop:nilmaniperiodlift}
Let $(G/\Gamma,G_\bullet)$ be a degree-$k$ filtered nilmanifold (not necessarily connected), let $\nss$ be the associated compact nilspace, and let $f:\mb{Z}^n\to \nss_{k-1}$ be an $M$-periodic morphism. Then there exists an $M'$-periodic morphism $\wt{f}:\mb{Z}^n\to \nss$ such that $\pi_{k-1}\co\wt{f}=f$, where $M'=CM^C$ for $C=O_{k,G_\bullet,\Gamma}(1)$.
\end{proposition}

\begin{lemma}[Abelian case of Theorem  \ref{thm:CFRpermorphlift}]\label{lem:abcaseproj}
Let $\ab,\ab'$ be compact abelian Lie groups such that there exists a continuous surjective homomorphism $\eta:\ab'\to \ab$, let $k\in \mb{N}$, and let $m$ be an $M$-periodic morphism $\mc{D}_1(\mb{Z}^n)\to \mc{D}_k(\ab)$ (i.e.\ a polynomial $\mb{Z}^n\to \ab$ of degree $k$). Then there exists an $M'$-periodic morphism $m':\mc{D}_1(\mb{Z}^n)\to \mc{D}_k(\ab')$ such that $\eta\co m' = m$, where $M'=CM^C$ with $C=O_{k,\eta}(1)$.
\end{lemma} 
\noindent Before proving these results, let us assume them and complete the picture of the overall strategy.

\begin{proof}[Proof of Theorem \ref{thm:CFRpermorphlift}]
The morphism $f_{k-1}:=\pi_{k-1}\co f: \mc{D}_1(\mb{Z}^n) \to \ns_{k-1}$ is $M$-periodic. Note that as $\nss$ is universal so is $\nss_{k-1}$. Hence, letting $\psi_{k-1}:\nss_{k-1}\to\ns_{k-1}$ be the induced fibration (satisfying $\pi_{k-1,\ns}\co\psi=\psi_{k-1}\co\pi_{k-1,\nss}$), we can suppose by induction (recall that the factor of a universal nilspace is universal by Lemma \ref{lem:k-1-univ-is-univ}) that there is an $O_{k-1,\psi_{k-1}}(M^{O_{k-1,\psi_{k-1}}(1)})$-periodic morphism $\wt{f}_{k-1}:\mb{Z}^n \to \nss_{k-1}$ such that $\psi_{k-1}\co \wt{f}_{k-1}=f_{k-1}$. By Proposition \ref{prop:nilmaniperiodlift}, there is an $O_{k,\nss}(M^{O_{k,\nss}(1)})$-periodic morphism $\wt{f}: \mb{Z}^n \to \nss$ such that $\pi_{k-1,\nss}\co \wt{f}=\wt{f}_{k-1}$. We then have
\[
\pi_{k-1,\ns} \co \psi\co\wt{f} = \psi_{k-1}\co \pi_{k-1,\nss}  \co\wt{f} = \psi_{k-1}\co \wt{f}_{k-1} = \pi_{k-1,\ns}\co f.
\]
Hence $\psi\co\wt{f} + m = f$, where $m$ is an $O_{k,\nss}(M^{O_{k,\nss}(1)})$-periodic morphism from $\mb{Z}^n$ to the $k$-th structure group $\ab_k(\ns)$. We thus reduce the problem to the abelian case of Lemma \ref{lem:abcaseproj}. Letting $\phi_k:\ab_k(\nss)\to\ab_k(\ns)$ be the $k$-th structure homomorphism of $\psi$ (see \cite[Definition  3.3.1]{Cand:Notes1}), Lemma \ref{lem:abcaseproj} provides an $O_{k,\phi_{k}}(M^{O_{k,\phi_{k}}(1)})$-periodic morphism (polynomial) $m':\mb{Z}^n\to\ab_k(\nss)$ such that $\phi_k\co m'=m$. We then have $\psi\co (\wt{f}+m') =  \psi\co\wt{f} + \phi_k\co m' = \psi\co\wt{f} + m = f$. Hence $\wt{f}+m'$ is a valid $O_{k,\psi}(M^{O_{k,\psi}(1)})$-periodic lift of $f$.
\end{proof}

\begin{remark}
Note that, if the constant $M'$ in Theorem \ref{thm:CFRpermorphlift} was exactly $M$ (or, more precisely, if for each $i\in [n]$ we could get as a period in the $i$-th coordinate the order of the $i$-th generator of $\ab$), then the lift $f'$ would essentially be a morphism on $\ab$ itself again, and the Jamneshan--Tao conjecture would follow. However, Theorem \ref{thm:CFRpermorphlift} does not hold in general with such a small $M'$. For example, let $G:=\begin{psmallmatrix} 1 & \mb{R} & \mb{R}\\[0.1em]0  & 1 & \mb{R}\\[0.1em] 0 & 0 & 1 \end{psmallmatrix}$ be the Heisenberg group and $\Gamma:=\begin{psmallmatrix} 1 & \mb{Z} & \mb{Z}\\[0.1em]0  & 1 & \mb{Z}\\[0.1em] 0 & 0 & 1 \end{psmallmatrix}$. Let $g_1:=\begin{psmallmatrix} 1 & 1/N & 0\\[0.1em]0  & 1 & 0\\[0.1em] 0 & 0 & 1 \end{psmallmatrix}$ and $g_2:=\begin{psmallmatrix} 1 & 0 & 0\\[0.1em]0  & 1 & 1/M\\[0.1em] 0 & 0 & 1 \end{psmallmatrix}$ for positive integers $N,M$. Let $\nss:=G/\Gamma$, let $\ns$ be the 1-step factor $\nss_1$ (isomorphic to $\mb{T}^2$), and let $f:\mb{Z}^2\to \ns$ be the morphism $(n,m)\mapsto\pi_1\co (g_1^n g_2^m)$. Clearly $f$ is $(N,M)$-periodic (i.e.\ with period $N$ in the first coordinate and $M$ in the second). However, in general there is no lift of $f$ to $\nss$ which is $(N,M)$ periodic. In fact, it can be proved that there exists an $(N,M)$ periodic lift to $\nss$ if and only if $N$ and $M$ are coprime (we omit the details). Note that, in the latter case, our morphism $f$ can be regarded as a morphism $\mb{Z}_{NM}\to \ns$, and thus as a periodic \emph{univariate} polynomial map (i.e.\ defined on $\mb{Z}$). In the univariate case of polynomial maps on $\mb{Z}$, it is already known from \cite[proof of Proposition 6.1]{CandSis} that for any degree-$k$ nilmanifold $\nss:=G/\Gamma$ where the last structure group is connected, any $M$-periodic polynomial map $f:\mb{Z}\to\nss_{k-1}$ can be lifted to an $M$-periodic polynomial map $\mb{Z}\to \nss$.
\end{remark}
\noindent We now turn to the proofs of Proposition \ref{prop:nilmaniperiodlift}, Lemma \ref{lem:abcaseproj}, and various related ingredients. 

Recall from \cite[Lemma A.1]{GTarit} that given a filtered group $(G,G_\bullet)$ of degree $k$, a polynomial map $f:\mb{Z}^n\to G_\bullet$ has a so-called \emph{Taylor expansion} of the following form:
\[
f(x)=\prod_{\mf{t}\in \mb{Z}_{\geq 0}^n} a_{\mf{t}}^{\binom{x}{\mf{t}}},
\]
where $a_{\mf{t}}\in G_{|\mf{t}|}$ (where $|\mf{t}|=t_1+\cdots +t_n$) and $\binom{x}{\mf{t}}=\binom{x_1}{t_1}\binom{x_2}{t_2}\cdots \binom{x_n}{t_n}$. In particular, the coefficients $a_{\mf{t}}$ are nontrivial only for $|\mf{t}|\leq k$, so there are at most $O_{k,n}(1)$ such coefficients.

We begin with a useful connection between periodicity of polynomial maps and the \emph{rationality} of their Taylor coefficients. Recall from \cite{LeibRat} that given a nilpotent Lie group $G$ with a subgroup $\Gamma\leq G$, an element $g$ is said to be \emph{rational} with respect to $\Gamma$ if $g^q\in \Gamma$ for some positive integer $q$.

We shall use the following fact.
\begin{lemma}\label{lem:ratcoefflift}
Let $(G/\Gamma,G_\bullet)$ be a degree-$k$ filtered nilmanifold (not necessarily connected). For any $i\in[k]$ and $g\in G_i$ let $gG_k$ be a rational element of the group $G/G_k$ with respect to $\Gamma G_k/G_k$. Then there exists $r\in G_k$ such that $gr\in G_i$ is a rational element in $G$ with respect to $\Gamma$. Moreover, if $g^qG_k\subset \Gamma G_k$ then $(gr)^{qq'}\in \Gamma$ where $q'$ is the exponent of the finite abelian group $H/H^0$ for $H:=G_k/(G_k\cap \Gamma)$.
\end{lemma}

\begin{proof}
If $i=k$, then we simply let $r=g^{-1}$ and the element $gr=\id$ satisfies the claims.

For $i<k$, suppose that $g^qG_k\subset \Gamma G_k$, so that $g^q=h\gamma$ for some $h\in G_k$ and $\gamma\in \Gamma$. Note that $H:=G_k/(G_k\cap \Gamma)$ is a compact abelian Lie group and thus we can identify it $H=\mb{T}^\ell\times A$ for some $\ell\in \mb{Z}_{\ge0}$ and some finite abelian group $A$. Then $h^{q'}=(t,0)\mod G_k\cap\Gamma$ where $t\in \mb{T}^\ell$ and $q'$ is the exponent of $H/H^0\cong A$. As the torus is divisible, there exists $r\in G_k$ such that $r^{qq'}=h^{-q'}\mod G_k\cap\Gamma$, that is, there exists $\gamma'\in G_k\cap \Gamma$ such that $r^{qq'}=h^{-q'}\gamma'$. Therefore, using that the elements $h,r,\gamma'\in G_k$ are in the center of $G$, we have
\[
(gr)^{qq'}=g^{qq'}r^{qq'}=(g^q)^{q'}r^{qq'}=(h\gamma)^{q'}r^{qq'}=h^{q'}\gamma^{q'}h^{-q'}\gamma'\in\Gamma.\qedhere
\]
\end{proof}

\begin{proposition}\label{prop:perequivrat}
Let $(G/\Gamma,G_\bullet)$ be a filtered nilmanifold of degree $k$ and let $f:\mb{Z}^n\to G$ be a polynomial map with Taylor expansion $f(x)=\prod_{\mf{t}\in \mb{Z}_{\geq 0}^n} a_{\mf{t}}^{\binom{x}{\mf{t}}}$ and $f(0)=\id_G$. Then the following statements are equivalent.
\begin{enumerate}
\item Every coefficient $a_{\mf{t}}$ is rational with respect to $\Gamma$.
\item The map $\mb{Z}^n\to G/\Gamma$ defined as $x\mapsto f(x)\Gamma$ is periodic.
\end{enumerate}
Moreover, if $(i)$ holds with $a_{\mf{t}}^q\in \Gamma$ for every $\mf{t}$, then $(ii)$ holds with $M$-periodicity for $M=q^{1+k^2(k+1)^2/4}k!$, and if $(ii)$ holds with $M$-periodicity, then $(i)$ holds with $a_{\mf{t}}^q\in \Gamma$ for every $\mf{t}$, for $q=CM^C$ where $C=O_{k,G_\bullet,\Gamma}(1)$.\footnote{More precisely, the dependence of $C$ is as follows. For every $i\in[2,k]$ letting $H_i:=G_i/\big((\Gamma\cap G_i)G_{i+1}\big)$, the dependence of $C$ on $k$ and the exponents of the groups $H_i/H_i^0$ for $i\in[2,k]$. In particular, if $G$ is abelian with nilspace structure $\mc{D}_k(G)$ then the dependence is only on $k$.}
\end{proposition}
\noindent We divide the proof into two parts. For the implication $(i)\Rightarrow (ii)$, we shall use the following adaptation of a result of Leibman.

\begin{proposition}[Uniform rationality; see Lemma 1.3 in \cite{LeibRat}]\label{prop:adaLeib}
Let $G/\Gamma$ be a degree-$k$ nilmanifold, let $q\in \mb{N}$, and let $a_1,\ldots,a_n\in G$ be such that $a_i^q\in \Gamma$ for every $i\in [n]$. Then, letting $H:=\langle a_1,\ldots,a_n\rangle$, we have that every element $b\in H$ satisfies $b^N\in \Gamma$ for $N=q^{k(k+1)/2}$. 
\end{proposition}
\begin{proof}
We adapt the proof of \cite[Lemma 1.3]{LeibRat}, giving full details here for completeness. Letting $L:=\Gamma\cap H$, we want to show that $H^{N}:=\{c^N:c\in H\}\subset L$. Let $H=H_1\ge H_2\ge\cdots\ge H_k\ge \{\id\}$ be the lower central series of $H$. We first prove, by induction on $i$, that 
\begin{equation}\label{eq:Leib1}
H_i^{q^i}\subset H_{i+1}L\text{ for every }i\in [k].
\end{equation}
For the case $i=1$ of the induction, note that $H_2L$ is a normal subgroup of $H$ and that $H/H_2L$ is a finite abelian group of exponent dividing $q$ (since every generator $a_i$ of $H$ satisfies $a_i^q\in L$). Hence we have $H^q\subset H_2 L$ as required.

Now, suppose by induction that $H_i^{q^i} \subseteq H_{i+1}L$. Then $ H_i^{q^i} \subseteq H_{i+1}(L \cap H_i)$. It follows that
\begin{equation}\label{eq:inclusions}
H_{i+1}^{q^{i+1}} 
\subseteq [H^q, H_i^{q^{i}}] H_{i+2}\subseteq [H_2 L, H_{i+1}(L \cap H_i)] H_{i+2}\subseteq H_{i+2} L.
\end{equation}
Indeed, to see the first inclusion, recall the following standard fact:
\begin{equation}\label{eq:comfact}
\text{for all $a\in H_r$, $b\in H_s$, and $j,\ell\in \mb{N}$ we have $[a,b]^{j\ell}=[a^{j},b^{\ell}]\!\mod H_{r+s+1}$.}    
\end{equation}
(This follows for instance from \cite[Lemmas 1.12 and 1.13]{Clement&al} since $[H_r,H_s]$ is central in $H$ mod $H_{r+s+1}$). Then since $H_{i+1}=[H,H_i]$, given any element of $[H,H_i]$, i.e.\ any finite product of commutators $[x,y]$ with $x\in H$ and $y\in H_i$, we can first apply the H--P formula \eqref{eq:H-P} to view the $q^{i+1}$-th power of this product as the product of powers $[x,y]^{q^{i+1}}$ mod $H_{i+2}$, and then apply \eqref{eq:comfact} to each of these powers to deduce that it is indeed in $[H^q, H_i^{q^{i}}] H_{i+2}$. 

To prove the third inclusion in \eqref{eq:inclusions}, let $h_2\in H_2$, $g\in L$, $h_{i+1}\in H_{i+1}$, and $h_i\in L\cap H_i$. Then $[h_2g,h_{i+1}h_i]=g^{-1}h_2^{-1}h_i^{-1}h_{i+1}^{-1}h_2gh_{i+1}h_i = g^{-1}h_2^{-1}h_i^{-1}h_2gh_i\mod H_{i+2}$ (where this last equality follows since elements of $H_{i+1}$ commute with every element modulo $H_{i+2}$). Since we (similarly) have $h_2^{-1}h_i^{-1}=h_i^{-1}h_2^{-1}\mod H_{i+2}$, we conclude that $[h_2g,h_{i+1}h_i]=g^{-1}h_i^{-1}gh_i\mod H_{i+2}$ and using that $g^{-1}h_i^{-1}gh_i\in L$, we deduce \eqref{eq:Leib1}.

Now, to deduce the main claim, we prove by induction that
$H^{q^{i(i+1)/2}} \subseteq H_{i+1} L $ for every $i\in [k]$. We already have the case $i=1$. Assuming then that the inclusion holds for a given $i$, we have that $H^{q^{(i+1)(i+2)/2}}$ equals
\[
\left( H^{q^{i(i+1)/2}} \right)^{q^{i+1}} \subseteq (H_{i+1} L)^{q^{i+1}} \subseteq H_{i+1}^{q^{i+1}} L^{q^{i+1}} H_{i+2}\subseteq (H_{i+2}L) L^{q^{i+1}} H_{i+2}\subseteq H_{i+2} L,
\]
where the second inclusion here follows by the H--P formula \eqref{eq:H-P}, the third one follows from \eqref{eq:Leib1}, and the last one by normality of $H_{i+2}$. This completes the induction, and now the main claim follows since we have $H^{q^{k(k+1)/2}} \subseteq L$.
\end{proof}

\noindent We shall also use the following multivariable generalization of Lemma A.12 from \cite{GTorb}. In \cite{GTorb}, Green and Tao mention that the proof of their Lemma A.12 (stated for polynomial maps on $\mb{Z}$) is easily generalized to the multivariable setting (i.e.\ for maps on $\mb{Z}^n$). We give below an alternative proof with an explicit bound on the period and without using Mal'cev coordinates.

\begin{lemma}\label{lem:multiGTlem} 
Let $(G/\Gamma,G_\bullet)$ be a filtered (possibly disconnected) nilmanifold of degree $k$ and let $g:\mb{Z}^n\to G_\bullet$ be a polynomial map such that $g(0)=\id_G$ and every Taylor coefficient of $g$ has its $q$-th power in $\Gamma$. Then $g(x)\Gamma$ is $M$-periodic for $M=q^{1+k^2(k+1)^2/4}k!$.
\end{lemma}

\begin{proof}
We have $g(x)=\prod_{|\mf{t}|\le k} g_{\mf{t}}^{\binom{x}{\mf{t}}}$ where the order of the indices $\mf{t}$ in the product is any such that for $\mf{t},\mf{t}'$, if $t_i\le t_i'$ for all $i\in[n]$ then $\mf{t}\le \mf{t}'$.

For any $a\in \mb{Z}$ and $b,r\in \mb{Z}_{\ge 0}$, the Chu-Vandermonde formula yields $\binom{a+r}{b}-\binom{a}{b}=\sum_{j=1}^b\binom{r}{j}\binom{a}{b-j}$, so $\binom{a+r}{b}=\binom{a}{b}+r\tfrac{c}{b!}$ for some integer $c$. Hence $g_{\mf{t}}^{\binom{x+re_i}{\mf{t}}}=g_{\mf{t}}^{\binom{x}{\mf{t}}+r\tfrac{c}{t_i!}}$ where $c\in \mb{Z}$ may vary in each appearance, even within the same formula, and may depend on $x,r,\mf{t},i$.

We now write $g(x+re_i)\Gamma=\prod_{|\mf{t}|\le k} g_{\mf{t}}^{\binom{x}{\mf{t}}+r\tfrac{c}{k!}}\Gamma$ (note that in the exponent we should have $r\tfrac{c}{t_i!}$ but as $t_i!$ always divides $k!$, we may indeed write $k!$ by modifying the integer $c$) and check that this is equal to $g(x)\Gamma$ for the claimed value of $r$. To check this we gradually make each term of the form $g_{\mf{t}}^{r\tfrac{c}{k!}}$ commute rightward past each other term in the product, until we obtain that $g(x+re_i)\Gamma=g(x)\Gamma$. To this end we use the following fact.
\begin{equation}\label{eq:commuclaim}
\text{Let $h\in G$ with $h^m\in \Gamma$, and $\gamma\in \Gamma$. Then, if $m^{k(k+1)/2}\mid m'$, we have $h^{-1}\gamma^{m'} h\in \Gamma$.}
\end{equation}
Indeed $h^{-1}\gamma^{m'} h=(h^{-1}\gamma h)^{m'}$ and the claim then follows by Proposition \ref{prop:adaLeib}.

Now, to make the term $g_{\mf{t}}^{r\tfrac{c}{k!}}$ commute rightward past the term $\prod_{\mf{t}'>\mf{t}}g_{\mf{t}'}^{\binom{x+re_i}{\mf{t}'}}$, note that since each element $g_{\mf{t}'}^q$ is in $\Gamma$, so is any integer power of this element. Therefore, by Proposition \ref{prop:adaLeib} we have $\Big(\prod_{\mf{t}'>\mf{t}}g_{\mf{t}'}^{\binom{x+re_i}{\mf{t}'}}\Big)^{q^{k(k+1)/2}}\in \Gamma$ (note that this is independent of $x$, $n$, and $i$ because all that matters in Proposition \ref{prop:adaLeib} is that the coefficients $g_{\mf{t}'}$ have $q$-th power in $\Gamma$). 

Now let $m=q^{k(k+1)/2}$, let $r=m^{k(k+1)/2}k!q$, and let us apply \eqref{eq:commuclaim} with $h=\prod_{\mf{t}'>\mf{t}}g_{\mf{t}'}^{\binom{x+re_i}{\mf{t}'}}$ and $\gamma=g_{\mf{t}}^{(k!q)\tfrac{c}{k!}}\in \Gamma$. We deduce that $\gamma^{m^{k(k+1)/2}} h\Gamma =h\Gamma$, that is, that we have the desired commuting $g_{\mf{t}}^{r\tfrac{c}{k!}}\big(\prod_{\mf{t}'>\mf{t}}g_{\mf{t}'}^{\binom{x+re_i}{\mf{t}'}}\big)\Gamma= \big(\prod_{\mf{t}'>\mf{t}}g_{\mf{t}'}^{\binom{x+re_i}{\mf{t}'}}\big)\Gamma$. Note that this process holds for any $\mf{t}$ with $|\mf{t}|\le k$, and thus the result follows.
\end{proof}

\begin{proof}[Proof of Proposition \ref{prop:perequivrat}]
The implication $(i)\Rightarrow (ii)$ follows by Lemma \ref{lem:multiGTlem}.

We now prove the implication $(ii)\Rightarrow (i)$. We shall argue by induction on $k$, and for this we shall first prove the case where $G$ is abelian, where $G/\Gamma$ is thus a compact abelian Lie group.

Thus, given a compact abelian Lie group $\ab$ (and using thus additive notation) and an $M$-periodic polynomial map $m:\mb{Z}^n\to\ab$ of degree $k$, with $m(0)=0$ and with Taylor expansion of the form $m(x)=\sum_{\mf{t}\in\mb{Z}_{\geq 0}^n:|\mf{t}|\leq k}a_{\mf{t}} \binom{x}{\mf{t}}$ for unique coefficients $a_{\mf{t}}\in\ab$, we want to prove that for some $q=CM^C$ where $C=O_k(1)$, we have $q a_{\mf{t}}=0$ for all $\mf{t}$. We argue by induction on $k$. For $k=1$ the claim is clear, in fact each coefficient satisfies $Ma_{\mf{t}}=0$. For $k>1$, fix any $\mf{t}$ with $|{\mf{t}}|=k$ and assume without loss of generality that $t_1>0$. Then, taking $t_1-1$ derivatives with respect to $e_1=(1,0,\ldots,0)$ and $t_2,\ldots,t_n$ derivatives with respect to $e_2,\ldots e_n$ respectively, we obtain $\Delta^{t_1-1}_{e_1}\Delta_{e_2}^{t_2}\cdots \Delta_{e_n}^{t_n}m(x)=a_{\mf{t}}x_1+\text{linear terms independent of }x_1$. Clearly this map is still $M$-periodic, hence $Ma_{\mf{t}}=0$. We then easily see using the Chu-Vandermonde identity that the map $x\mapsto \sum_{|\mf{t}|=k}a_{\mf{t}}\binom{x}{\mf{t}}$ is $k!M$-periodic. Hence the map $m(x)-\sum_{|\mf{t}|=k}a_{\mf{t}}\binom{x}{\mf{t}}=\sum_{\mf{t}\in\mb{Z}_{\geq 0}^n:|\mf{t}|\leq k-1}a_{\mf{t}} \binom{x}{\mf{t}}$ is $k!M$-periodic, and has degree $k-1$, so by induction its Taylor coefficients satisfy the required claim with $q=C'(k!M)^{C'}=CM^C$.

Now, for the general case, let $f:\mb{Z}^n\to G$ be a polynomial map relative to $G_\bullet$ such that $f(x)\Gamma$ is $M$-periodic and $f(0)=\id_G$. To prove that every Taylor coefficient $a_{\mf{t}}$ of $f$ satisfies $a_{\mf{t}}^q\in \Gamma$, we proceed by induction on the degree $k$ of the filtration $G_\bullet$. For $k=1$, the group $G$ is abelian and $f$ is linear, so each Taylor coefficient $a_{\mf{t}}$ of $f$ in fact satisfies $a_{\mf{t}}^M\in\Gamma$. Thus, assume that the result holds for degree up to $k-1$ and let $f(x)=\prod_{|\mf{t}|\le k} a_{\mf{t}}^{\binom{x}{\mf{t}}}$. Then clearly $\pi_{k-1}\co (f(x)\Gamma)$ is an $M$-periodic map into the degree-$(k-1)$ nilmanifold $(G/G_k)/(G_k\Gamma/G_k)$, and therefore, by induction, for some $q=CM^C$ we have that $a_{\mf{t}}^q G_k\subset \Gamma G_k$ for all $\mf{t}$. By Lemma \ref{lem:ratcoefflift}, we can find elements $r_{\mf{t}}\in G_k$ such that $\wt{f}(x):=\prod_{|\mf{t}|\le k} (a_{\mf{t}}r_{\mf{t}})^{\binom{x}{\mf{t}}}$ is a polynomial whose Taylor coefficients all have $qq'$-th power in $\Gamma$. By Lemma \ref{lem:multiGTlem}, the map $\wt{f}\Gamma$ is $M'$-periodic for $M'=C_k(qq')^{C_k}$. Since $f(x)G_k=\wt{f}(x)G_k$ for all $x$, we can define the $G_k/(G_k\cap \Gamma)$-valued map $f-\wt{f}$, which is polynomial of degree $k$ and is $M''=\lcm(M,M')$-periodic. By the case proved in the previous paragraph, all Taylor coefficients of $f-\wt{f}$ are $C{M''}^C$-rational. Therefore $f=\wt{f}+(f-\wt{f})$ has all its Taylor coefficients having $C'M^{C'}$-th power in $\Gamma$ as required.
\end{proof}

\begin{proof}[Proof of Lemma \ref{lem:abcaseproj}]
The given $M$-periodic polynomial map $m$ has a Taylor expansion $m(x)=\sum_{{\mf{t}}\in\mb{Z}_{\geq 0}^n:|{\mf{t}}|\leq k}a_{\mf{t}} \binom{x}{{\mf{t}}}$, for unique coefficients $a_{\mf{t}}\in\ab$. By Proposition \ref{prop:perequivrat}, we have $qa_{\mf{t}}=0$ for $q=CM^C$, for every ${\mf{t}}\neq 0$. Let $\ab'_\bullet$ be the filtration of degree 2 on $\ab'$ with $\ab'_2= \ker(\eta)$. We now apply Lemma \ref{lem:ratcoefflift} with $(G,G_\bullet):=(\ab',\ab'_\bullet)$, $\Gamma=\{\id\}$, and $a_{\mf{t}}\in \ab\cong\ab'/\ker(\eta)$ for\footnote{Applying the first isomorphism theorem to identify $\ab$ with $\ab'/\ker(\eta)$.} $\mf{t}\not=0$. We obtain that there exists $b_{\mf{t}}\in \ab'$ such that $\eta(b_{\mf{t}})=a_{\mf{t}}$ and $qq' b_{\mf{t}}=0$ for $q'=O_{\eta}(1)$. Choosing now any element $b_0$ such that $\eta(b_0)=a_0$, we can define the polynomial map $m':\mb{Z}^n\to \ab'$ by the formula $m'(x)=\sum_{{\mf{t}}\in\mb{Z}_{\geq 0}^n:|{\mf{t}}|\leq k}b_{\mf{t}} \binom{x}{{\mf{t}}}$. Clearly, we have $\eta\co m'=m$, and since $qq' b_{\mf{t}}=0$ for every $\mf{t}\neq 0$, it follows from Proposition \ref{prop:perequivrat} that $m'$ is $CM^C$-periodic where $C=O_{\eta,k}(1)$.
\end{proof}

\noindent We now proceed to the proof of Proposition \ref{prop:nilmaniperiodlift} and thus complete the proof of Theorem \ref{thm:CFRpermorphlift}.
\begin{proof}[Proof of Proposition \ref{prop:nilmaniperiodlift}]
We begin with a polynomial map $f:\mb{Z}^n\to (G/G_k)/(\Gamma G_k/G_k)$ which is $M$-periodic, with Taylor expansion $f(x)=\prod_{\mf{t}\in \mb{Z}_{\geq 0}^n} (a_{\mf{t}}G_k)^{\binom{x}{\mf{t}}} \Gamma G_k$, and we want to lift this to a periodic polynomial map $\mb{Z}^n\to\nss$. By the implication $(ii)\Rightarrow (i)$ in Proposition \ref{prop:perequivrat}, for every $\mf{t}\neq 0$ we have $(a_{\mf{t}}G_k)^q\subset \Gamma G_k$ for some $q=CM^C$ for $C=O_{k,G_\bullet,\Gamma}(1)$. By Lemma \ref{lem:ratcoefflift}, for every $\mf{t}\neq 0$ there is $b_{\mf{t}}\in G_{|\mf{t}|}$ such that $b_{\mf{t}}^{qq'}\in\Gamma$ where $q'=O_{G_\bullet,\Gamma}(1)$ and $b_{\mf{t}}G_k=a_{\mf{t}}G_k$. Using these elements as Taylor coefficients (where the constant Taylor coefficient is also lifted arbitrarily), we obtain a new polynomial map $\wt{f}:\mb{Z}^n \to G/\Gamma$ such that $\pi_{k-1}\co\wt{f}=f$ and which, by Lemma \ref{lem:multiGTlem}, is $M'$-periodic for $M'=CM^C$ where $C=O_{k,G_\bullet,\Gamma}(1)$.
\end{proof}

\subsection{Proof of the inverse theorem}\label{subsec:PfInvThm}\hfill\\
We have now all the necessary ingredients to prove Theorem \ref{thm:projinv}, which we recall here.
\begin{theorem}\label{thm:projinv2}
For any $k\in\mb{N}$ and $\delta>0$, there exist $C>0$ and $\varepsilon>0$ such that the following holds. For any finite abelian group $\ab$ and any 1-bounded function $f:\ab\to \mb{C}$ with $\|f\|_{U^{k+1}}\ge \delta$, there exists a projected $k$-step nilsequence $\phi_{*\tau}$ on $\ab$ of complexity at most $C$ such that $|\mb{E}_{x\in\ab} f(x) \overline{\phi_{*\tau}(x)}|\ge \varepsilon$. Moreover $\phi_{*\tau}$ is rank-preserving and has torsion $C\exp_{\ab}^C$.
\end{theorem}
\begin{proof}
By \cite[Theorem 1.6]{CSinverse} there exists a constant $c'=c'_{k,\delta}$, a $k$-step \textsc{cfr} nilspace $\ns$ of complexity at most $c'$, a morphism $\varphi:\mc{D}_1(\ab)\to \ns$, and a 1-bounded $c'$-Lipschitz function $F':\ns\to\mb{C}$ such that $|\mb{E}_{x\in \ab} f(x) \overline{F'(\varphi(x))}|>\delta^{2^{k+1}}/2$. By Theorem \ref{thm:realmain} there exists a nilmanifold $\nss$ and a fibration $\psi:\nss\to\ns$. Note that by construction the complexity of $\nss$ is bounded in terms of the complexity of $\ns$ and thus it is of complexity $C_1=O_{c'}(1)$.\footnote{Implicit in this bound $O_{c'}(1)$ is a function $R(c')$ which is determined as follows: having fixed a complexity notion for \textsc{cfr} nilspaces as explained in \cite[Definition 1.2]{CSinverse}, for each \textsc{cfr} nilspace $\ns$ we fix one such extension $\psi:\nss\to\ns$ given by Theorem \ref{thm:cfr-are-factor-of-nilmanifolds}. The complexity of the nilspace $\nss$ and the Lipschitz constant of $\psi$ are at most $R(c')$ for some fixed function $R:\mb{N}\to\mb{N}$ (using \cite[Proposition A.13]{CGSS-spec} for the Lipschitz bound). Thus $R$ is determined by the original complexity notion.} 
Note also that by \cite[Proposition A.13]{CGSS-spec} we have that $\psi$ is Lipschitz with a constant bounded in terms of $c'$, and is thus $O_{c'}(1)$. Hence $F:=F'\co\psi$ is $C_2$-Lipschitz where $C_2=O_{c'}(1)$. Next, note that (by the classification of finite abelian groups) there exists a surjective homomorphism $\pi:\mb{Z}^n\to \ab$ where $n$ is the rank of $\ab$. Then $\varphi\co \pi\in \hom(\mc{D}_1(\mb{Z}^n),\ns)$ is $M$-periodic for $M=\exp_{\ab}$, so by Theorem \ref{thm:CFRpermorphlift} there is an $M'$-periodic morphism $f':\mb{Z}^n\to \nss$ such that $\psi\co f' = \varphi\co\pi$ and $M'=C_3M^{C_3}$ where $C_3=O_{c'}(1)$. By $M'$-periodicity, the morphism $f'$ factors through the finite abelian group $\ab'=\mb{Z}^n/(M'\mb{Z}^n)$, of rank $n$ and exponent $M'$, that is, there is a morphism $\varphi':\ab'\to\nss$ such that $f'(x)=\varphi'(x+M'\mb{Z}^n)$. Moreover, letting $\tau$ be the natural surjective homomorphism $\ab'\to\ab$,  we have $\varphi\co \tau=\psi\co \varphi'$. Therefore $\mb{E}_{x\in \ab} f(x) \overline{F'(\varphi(x))}=\mb{E}_{y\in \ab'} f\co\tau(y) \overline{F'\co \varphi(\tau(y))}=\mb{E}_{y\in \ab'} f(\tau(y)) \overline{F\co\varphi'(y)}$. Letting $C=\max_{i\in[3]}(C_i)$, the result follows.
\end{proof}

\subsection{Projected $k$-step nilsequences are obstructions to $U^{k+1}$-uniformity}\hfill\smallskip\\
We complete this section with the following result, establishing that $k$-step projected nilsequences are genuine obstructions to having small $U^{k+1}$-norm.

\begin{proposition}\label{prop:directthm}
For every $k\in \mb{N}$ and $\delta,C>0$, there exists $\varepsilon>0$ such that the following holds. Let $\ab$ be a finite abelian group and suppose that $f:\ab\to \mb{C}$ is a 1-bounded function satisfying $|\langle f,\phi_{*\tau}\rangle|\ge \delta$ for some projected $k$-step nilsequence $\phi_{*\tau}$ on $\ab$ of complexity at most $C$. Then $\|f\|_{U^{k+1}}\ge \varepsilon$.
\end{proposition}
The proof is very similar in spirit to that of \cite[Lemma 5.5]{CGSS-bndtor} combined with \cite[Theorem 4.26]{CGSS-spec}. As there are some adjustments to make, we give the details below.
\begin{proof}
Recall from Definition \ref{def:pronilseq} that $\phi_{*\tau}(x)=\mb{E}_{y\in \tau^{-1}(x)} F(g(y))$, for some surjective homomorphism $\tau:\ab'\to\ab$ (where $\ab'$ is another finite abelian group), some filtered nilmanifold $(G/\Gamma,G_\bullet)$ of degree $k$, some polynomial map $g:\ab'\to G/\Gamma$ relative to $G_\bullet$, and a 1-bounded Lipschitz function $F:G/\Gamma\to\mb{C}$. It then suffices to prove a lower bound $\|f\|_{U^{k+1}}\ge \varepsilon$ where $\varepsilon>0$ depends only on $\delta$, $G/\Gamma$, and $\|F\|_{\Lip}>0$.

Note that as $\tau$ is surjective, the homomorphism $\tau^{\db{k+1}}:\cu^{k+1}(\ab')\to\cu^{k+1}(\ab)$ sending a cube $\q$ to the cube $\tau\co\q$ is also surjective (indeed we have the expression $\q(v)=x+v_1 h_1+\cdots+v_{k+1}h_{k+1}$ for some elements $x,h_i\in\ab$ and then the claim follows by taking $\tau$-preimages of these elements in $\ab'$). It follows that $\|f\|_{U^{k+1}(\ab)}=\|f\co\tau\|_{U^{k+1}(\ab')}$. By \cite[Theorem 4.26]{CGSS-spec} applied to the map $F\co g$ with parameter $\delta$ (from the statement of \cite[Theorem 4.26]{CGSS-spec}) equal to $\delta/2$ here, we have that there exists $h:\ab'\to\mb{C}$ such that $\|h\|_{U^{k+1}(\ab')}^*\le N_{\delta/2,G/\Gamma,\|F\|_{\Lip}}$ (recall that we have assumed that $\|F\|_\infty\le 1$) and $\|F\co g-h\|_\infty\le \delta/2$. Hence $
\delta \le |\langle f,(F\co g)_{\tau*}\rangle_{\ab}| = |\langle f\co \tau, F\co g\rangle_{\ab'}| \le |\langle f\circ \tau, F\co g - h \rangle_{\ab'}| + |\langle f\co \tau, h\rangle_{\ab'}|\le \|f\|_{L^1(\ab)}\|F\co g - h\|_\infty+ \|f\co\tau\|_{U^{k+1}(\ab')} \|h\|_{U^{k+1}(\ab')}^*\le \delta/2 + N\|f\|_{U^{k+1}(\ab)}$. Therefore $\|f\|_{U^{k+1}}\ge \varepsilon>0$ where $\varepsilon:=\tfrac{\delta}{2N}$ and the result follows.
\end{proof}

\section{\texorpdfstring{On minimal $\mb{Z}^\omega$-systems of finite order}{On minimal Zw-systems of finite order}}\label{sec:ergapp}
\noindent In this section we prove Theorem \ref{thm:main-dynam}. We begin by recalling some background notions from dynamics. 

A \emph{topological dynamical system} (or $A$-system) $(\ns,A)$ consists of a  compact metric space $\ns$ and a countable discrete abelian group $A$ which acts on $\ns$ via homeomorphisms, that is, for each $t\in A$ we have a homeomorphism $T^t:\ns\to\ns$ and for every $t,t'\in A$ we have $T^t\co T^{t'}=T^{t+t'}$ and $(T^{t})^{-1}=T^{-t}$. Abusing the notation, instead of writing $T^t(x)$ for $t\in A$ and $x\in\ns$, we often write $tx$. The system is \emph{distal} if, given any points $x,y\in\ns$, if there exists a sequence $t_i\in A$ such that $\operatorname{dist}(t_ix,t_iy)\to 0$ as $i\to\infty$ (where $\operatorname{dist}$ is a metric on $\ns$) then $x=y$. A system is \emph{minimal} if the $A$-orbit of any point $x\in \ns$, i.e.\ the set $\{tx:t\in A\}$, is dense in $\ns$. 

The following definition of a system of order $k$ for abelian group actions was introduced by Host, Kra, and Maass.

\begin{defn}[Definition 3.2 of \cite{HKM}]\label{def:reg-prox-rel} Let $(\ns, A)$ be a topological dynamical system with $A$ abelian and $k \in \mb{N}$. The points $x, y \in \ns$ are said to be \emph{regionally proximal of order $k$}, denoted $(x, y) \in \mathrm{RP}^{[k]}(\ns)$, if there are sequences of elements $t_i^{1}, t_i^{2}, \ldots, t_i^{k} \in A$ and $x_i, y_i \in \ns$ such that for all $\boldsymbol{\epsilon} = (\epsilon_1, \epsilon_2, \ldots, \epsilon_{k}) \in \db{k}  \setminus \{0^{k}\}$:
\[ \lim_{i \to \infty} x_i = x, \quad \lim_{i \to \infty} y_i = y, \quad \lim_{i \to \infty}  \operatorname{dist}\!\big( \big( \sum_{j=1}^{k} \epsilon_j t_i^{j} \big) x_i,\; \big( \sum_{j=1}^{k} \epsilon_j t_i^{j} \big) y_i \big) = 0.
\]
\noindent If $\mathrm{RP}^{[k]}(\ns)$ is trivial (i.e.\ if it is the identity relation), then we say that $(\ns,A)$ is a \emph{system of order $k$}.
\end{defn}

\noindent It turns out that such minimal distal systems of order $k$ are compact nilspaces with cube sets $\cu^n(\ns):=\overline{\{\q x^{\db{n}}:x\in\ns,\; \q\in \cu^n(\mc{D}_1(A))\}}$, see \cite[Theorem 7.14]{GGY}. Indeed, for any topological dynamical system $(\ns,A)$, when with the previous definition we have that the cube sets $\cu^n(\ns)$ define a $k$-step nilspace structure on $\ns$, we will say that this is a (topological) \emph{nilspace system of order $k$}.

We are now ready to introduce the main class of objects which appear in Theorem \ref{thm:main-dynam}.

\begin{defn}[Generalized polynomial orbit system]\label{def:gen-poly-orb}
Let $G/\Lambda$ be a (possibly disconnected) nilmanifold and let $A$ be a countable abelian group. Let $f\in \hom(\mc{D}_1(A),G/\Lambda)$. Let $\orb(f)$ be the orbit of $f$ under the shift action, i.e.\ $\orb(f):=\{f(\cdot+t)\in \hom(\mc{D}_1(A),G/\Lambda):t\in A\}\subset (G/\Lambda)^{A}$. Then the \emph{generalized polynomial orbit system} (or \emph{$A$-polynomial orbit system} if we want to highlight the acting group) associated with $f$ is $(\overline{\orb(f)},A)$ where the action is given by $(g,t)\mapsto g(\cdot+t)$ for any $t\in A$ and $g\in \overline{\orb(f)}$.
\end{defn}

\noindent As mentioned in the introduction, this definition is a generalization of the well-known concept of \emph{nilsystem of polynomial orbits}, see \cite[Ch.\ 14, \S 2.4]{HKbook}. Such systems are generated when we take a nilmanifold $G/\Lambda$, an element $f\in \hom(\mc{D}_1(\mb{Z}),G/\Lambda)$, and similarly we consider $\overline{\orb(f)}$ with the shift action $(g,t)\mapsto g(\cdot+t)$ for any $t\in \mb{Z}$ and $g\in \overline{\orb(f)}$. In this case, it can be proved that such a construction is isomorphic to a nilsystem (and the converse also holds); see the discussion before Theorem \ref{thm:main-dynam} involving \cite[Ch.\ 14, Propositions 13 and 14]{HKbook}. This motivates asking whether more general similar arguments could similarly describe $A$-polynomial orbit systems as translational systems.

\begin{question}\label{Q:polyorbnils}
Is every polynomial orbit system as per Definition \ref{def:gen-poly-orb} isomorphic to a translational system as defined in \cite[Definition A.2]{JST1}?
\end{question}

For our purposes, we will be mainly interested in the case $A=\mb{Z}^\omega:=\bigoplus_{n\in \mb{N}}\mb{Z}$. Note that, if $A'$ is any countable abelian group, then there exists a surjective homomorphism $\mb{Z}^\omega\to A'$ and thus any $A'$-system can be turned into a $\mb{Z}^\omega$-system via such group extension. It turns out that $\mb{Z}^\omega$-polynomial orbit systems enjoy several good properties which make them suitable as general extension-systems of $\mb{Z}^\omega$ minimal systems of order $k$.

\begin{proposition}\label{prop:properties-z-omega-sys}
Let $G/\Lambda$ be a degree-$k$ nilmanifold, let $f\in\hom(\mc{D}_1(\mb{Z}^\omega),G/\Lambda)$, and let $(\overline{\orb(f)},\mb{Z}^\omega)$ be the corresponding $\mb{Z}^\omega$-polynomial orbit system. Then this system is minimal and of order $k$.\footnote{Recall that, by \cite[Remark 7.6]{GGY}, any such system is distal.}
\end{proposition}

The proof of this statement is divided into several results. First, let us prove that any element in $\hom(\mc{D}_1(\mb{Z}^\omega),G/\Lambda)$ has a Taylor expansion.

\begin{lemma}\label{lem:inf-taylor}
Let $(G,G_\bullet)$ be a nilpotent group of degree $k$, let $[0,k]:=\{0,\ldots,k\}$, and let $\Lambda$ be a subgroup of $G$. Let $s:\mb{N}\to [0,k]^\omega$ be a bijective function\footnote{Recall that $[0,k]^\omega$ is the set of elements in $[0,k]^{\mb{N}}$ with finitely many non-zero coordinates.} with the property that given any two $w,w'\in [0,k]^\omega$, if $w_i\le w_i'$ for all $i\in \mb{N}$, then $s^{-1}(w)\le s^{-1}(w')$. Then any $f\in \hom(\mc{D}_1(\mb{Z}^\omega),G/\Lambda)$ has a factorization of the form
\[
f(\underline{x})=\prod_{i=1}^\infty  g_{i}^{\binom{\underline{x}}{s(i)}}\Lambda
\]
where $g_{i}\in G_{\sum_{j=1}^\infty s(i)_j}$, the binomial term $\binom{\underline{x}}{s(i)}:=\prod_{j=1}^\infty\binom{x_{j}}{s(i)_j}$, and $\underline{x}=(x_1,x_2,\ldots)\in \mb{Z}^\omega$.
\end{lemma}
\noindent Note that this formula is always well-defined. Indeed, the exponents $\binom{\underline{x}}{s(i)}=\prod_{j=1}^\infty\binom{x_{j}}{s(i)_j}$ are always finite products because $s(i)$ has at most a finite number of non-zero coordinates. Similarly, for any fixed $\underline{x}$ note that if $x_i=0$ for all $i\ge i_0$ then the only terms $g_i^{\binom{\underline{x}}{s(i)}}$ that may be non-identity are those corresponding to $s^{-1}([0,k]^{i_0})$, which is a finite set.

\begin{proof}
By \cite[Proposition 4.3]{CGSS-abramov} note that, as $G\to G/\Lambda$ is a fibration, it is enough to prove the result assuming that $\Lambda=\{\id\}$. Indeed, if the result holds for $G$, let $f'\in \hom(\mc{D}_1(\mb{Z}^\omega),G)$ be such that $f=f'\Lambda$ and the result is clear.

The proof proceeds by induction. That is, we are going to create a sequence $g_n\in  G_{\sum_{j=1}^\infty s(n)_j}$ for $n\in \mb{N}$ such that $f$ agrees with $f_n:=\prod_{i=1}^n  g_{i}^{\binom{\underline{x}}{s(i)}}$ on the set $\{s(1),\ldots,s(n)\}$. Note that $s(1)=0^{\infty}$ and we can simply set $g_{1}:=f(0^{\infty})$.

Assuming that the result holds up to $n\ge 1$, let us find the element $g_{n+1}$. Let $r\in \mb{N}$ be such that $s(n+1)\in [0,k]^r$. Then, note that both $f$ and $f_n$ restricted to $\mb{Z}^r$ are elements of $\hom(\mc{D}_1(\mb{Z}^r),G)$. By standard results, e.g. \cite[Theorem 2.2.14]{Cand:Notes1}, both of these elements are polynomial maps, and thus there is a unique Taylor expansion. In particular  $f(s(n+1))=f_n(s(n+1))g$ for some $g\in G_{\sum_{j=1}^\infty s(n+1)_j}$. Letting $g_{n+1}:=g$ it follows that our sequence $g_1,\ldots,g_n,g_{n+1}$ satisfy the required assumptions.

To complete the proof, we claim that $f_\infty:=\prod_{i=1}^\infty  g_{i}^{\binom{\underline{x}}{s(i)}}$ equals $f$. Indeed by construction it is equal to $f$ on $[0,k]^\omega$, and then, using the fact that $(G,G_\bullet)$ is of degree $k$, it is easy to see that $f_\infty$ and $f$ must agree on all of $\mb{Z}^\omega$.
\end{proof}

\begin{lemma}\label{lem:poly-orb-distal-minimal}
Let $G/\Lambda$ be a degree-$k$ nilmanifold. Then the $\mb{Z}^\omega$-system\footnote{Where the action is given by the shift $(f,z)\in \hom(\mc{D}_1(\mb{Z}^\omega),G/\Lambda)\times \mb{Z}^\omega\mapsto f(\cdot+z)$.} $(\hom(\mc{D}_1(\mb{Z}^\omega),G/\Lambda),\mb{Z}^\omega)$ is distal. In particular, any $\mb{Z}^\omega$-polynomial orbit system is minimal and distal.
\end{lemma}

\begin{proof}
We need to prove that, given $f,g\in \hom(\mc{D}_1(\mb{Z}^\omega),G)$ and a sequence $\underline{t_n}\in \mb{Z}^\omega$ such that $\lim_{n\to\infty} f(\cdot+\underline{t_n})\Lambda=\lim_{n\to\infty} g(\cdot+\underline{t_n})\Lambda$, we have $f\Lambda=g\Lambda$. We argue by induction on the degree $k$ of the nilmanifold. The case $k=1$ follows directly from \cite[Lemma 2.7]{CGSS-abramov}.

Assuming the result holds for degree up to $k-1$, note that if $\lim f(\underline{x}+\underline{t_n}) \Lambda = \lim g(\underline{x}+\underline{t_n}) \Lambda$ then, quotienting by $G_k$ we have that $\lim f(\underline{x}+\underline{t_n})\Lambda G_k/G_k = \lim g(\underline{x}+\underline{t_n})\Lambda G_k/G_k$. Hence, by induction $f(\underline{x}) \Lambda G_k/G_k = g(\underline{x}) \Lambda G_k/G_k$ for all $\underline{x}\in \mb{Z}^\omega$. By Lemma \ref{lem:inf-taylor}, we have that $g^{-1}f(\underline{x})=\prod_{i=1}^\infty  h_{i}^{\binom{\underline{x}}{s(i)}}$ and we have just proved that $h_i G_k\in \Lambda G_k/G_k$. In particular, we may write $h_i=h_i'\lambda_i$ where $h_i'\in G_k$ and $\lambda_i\in \Lambda\cap G_{\sum_{j=1}^\infty s(i)_j}$ for all $i\in\mb{N}$. Note that, as $G_k$ is in the center of $G$, we may write $g^{-1}f=h'\lambda$ where $h'(\underline{x})=\prod_{i=1}^\infty  (h'_{i})^{\binom{\underline{x}}{s(i)}}$ and $\lambda(\underline{x})=\prod_{i=1}^\infty  \lambda_{i}^{\binom{\underline{x}}{s(i)}}$. Hence, we have that $\lim f(\cdot+\underline{t_n})\Lambda = \lim gh'\lambda(\cdot+\underline{t_n})\Lambda = \lim gh'(\cdot+\underline{t_n})\Lambda$ equals $\lim g(\cdot+\underline{t_n})\Lambda$ by hypothesis. Note now that, for any fixed $x\in \mb{Z}^\omega$, the values of $gh'(x+t_n)\Lambda$ and $g(x+t_n)\Lambda$ lie in the same $k$-fiber of the nilmanifold $G/\Lambda$ and they both converge as $n\to\infty$. Thus, as the difference map is continuous, we have that $\lim h'(x+t_n)\Lambda$ exists for all $x$ and must be equal to the difference of the limits, which is $\Lambda$. Recall now that $G_k\Lambda/\Lambda\cong G_k/(G_k\cap \Lambda)$ and this is a compact abelian Lie group (recall that in order to have this we need $G_k$ to be a rational subgroup of the nilmanifold $G\Lambda/\Lambda$). Thus, we may apply \cite[Lemma 2.7]{CGSS-abramov} which shows that $h'(x)\in \Lambda$ for all $x$. Hence $f\Lambda = gh'\lambda \Lambda=g\Lambda$ as required.

Finally, by \cite[Corollary 2.8]{CGSS-abramov} we see that the $\mb{Z}^\omega$-polynomial orbit system is minimal.\end{proof}

\begin{proposition}\label{prop:z-omega-are-nilspaces}
Let $G/\Lambda$ be a degree-$k$ nilmanifold and $f\in\hom(\mc{D}_1(\mb{Z}^\omega),G/\Lambda)$. Then the $\mb{Z}^\omega$-system $(\overline{\orb(f)},\mb{Z}^\omega)$ can be endowed with the structure of a $k$-step nilspace making the dynamical system a nilspace system of order $k$.
\end{proposition}

\begin{proof}
By Lemma \ref{lem:poly-orb-distal-minimal}, a $\mb{Z}^\omega$-polynomial orbit system is minimal and distal. By \cite[Theorem 7.14]{GGY}, it suffices to prove that $(\overline{\orb(f)},\mb{Z}^\omega)$ is a system of order $k$, i.e.\ that the regionally proximal relation of order $k$ (see Definition \ref{def:reg-prox-rel}) is trivial. We use the characterization of being regionally proximal of order $k$ from \cite[Definition 3.6]{GGY}: Two elements $g,g'\in \overline{\orb(f)}$ are regionally proximal of order $k$ if there exist sequences $g_n\in \overline{\orb(f)}$ and $t^{(n)}:=(t_1^{(n)},\ldots,t_{k+1}^{(n)}) \in (\mb{Z}^\omega)^{k+1}$ such that $\lim g_n= g$ and $\lim g_n(\cdot+v\cdot t^{(n)}) = g'$ for any $v\in \db{k+1}\setminus\{0^{k+1}\}$. Let $(g,g')\in \mathrm{RP}^{[k]}(\overline{\orb(f)})$ be two such elements. For any fixed $x\in \mb{Z}^\omega$, define a sequence of cubes $\q_n(v):=g_n(x+v\cdot t^{(n)})\in \cu^{k+1}(G/\Lambda)$ for $n\in \mb{N}$. Recall that the evaluation map $\hom(\mb{Z}^\omega,G/\Lambda)\to G/\Lambda$ given by $h\mapsto h(x)$ is continuous. Thus, we have that $\q_n(0^{k+1})\to g(x)$ and $\q_n(v)\to g'(x)$ for $v\not=0^{k+1}$ as $n\to\infty$. Since $\cu^{k+1}(G/\Lambda)$ is a closed set, we have that $\lim \q_n\in \cu^{k+1}(G/\Lambda)$, and since $G/\Lambda$ is a $k$-step nilmanifold, the completion of $k+1$ cubes is unique. Therefore $\lim\q_n(0^{k+1})=\lim\q_n(v)$ for any $v\not=0^{k+1}$, which proves that $g(x)=g'(x)$. As $x\in \mb{Z}^\omega$ was arbitrary, we have that $g=g'$, and the result follows.
\end{proof}

\begin{proof}[Proof of Proposition \ref{prop:properties-z-omega-sys}] The result follows directly from \cite[Theorem 7.14]{GGY}, Lemma \ref{lem:poly-orb-distal-minimal}, and Proposition \ref{prop:z-omega-are-nilspaces}.
\end{proof}

To prove Theorem \ref{thm:main-dynam}, we need some additional results which we prove now, the first one being a direct consequence of our main result Theorem \ref{thm:cfr-are-factor-of-nilmanifolds}.

\begin{corollary}\label{cor:cov-by-inv-lim-coset}
Let $\ns$ be a compact $k$-step nilspace. Then there exists a compact $k$-step nilspace $\nss$ which is the inverse limit (in the nilspace category) of nilmanifolds $\nss=\varprojlim G_n/\Lambda_n$ and a fibration $\varphi:\nss\to\ns$.
\end{corollary}

\begin{proof}
Recall that, by \cite[Theorem 2.7.3]{Cand:Notes2}, the nilspace $\ns$ is equal to an inverse limit of $k$-step \textsc{cfr} nilspaces, i.e.\ $\ns=\varprojlim \ns_i$. The idea now is to use Theorem \ref{thm:cfr-are-factor-of-nilmanifolds} to extend all these factors in a consistent way. Let us present the commutative diagram that summarizes the proof.

\[\begin{tikzcd}[sep=small]
	{\ns_1} & {\ns_2} & {\ns_3} & {\ns_4} & \cdots \\
	G_1/\Lambda_1 & {G_1/\Lambda_1\times_{\ns_1}\ns_2} & & \\
	 & {G_2/\Lambda_2} & {G_2/\Lambda_2\times_{\ns_2}\ns_3} & \\
	&  & G_3/\Lambda_3 & {G_3/\Lambda_3\times_{\ns_3}\ns_4} \\
	&&& G_4/\Lambda_4 & \cdots.
	\arrow[from=1-2, to=1-1]
	\arrow[from=1-3, to=1-2]
	\arrow[from=1-4, to=1-3]
	\arrow[from=2-1, to=1-1]
    \arrow[from=3-2, to=2-2]
	\arrow[from=2-2, to=1-2]
    \arrow[from=2-2, to=2-1]
	\arrow[from=3-3, to=3-2]
	\arrow[from=3-3, to=1-3]
	\arrow[from=4-3, to=3-3]
	\arrow[from=4-4, to=4-3]
	\arrow[from=5-4, to=4-4]
	\arrow[from=4-4, to=1-4]
    \arrow[from=1-5, to=1-4]
    \arrow[from=5-5, to=5-4]
\end{tikzcd}\]
We create this diagram as follows. Note that the maps $\ns_i\to\ns_{i-1}$ are the ones of the inverse limit $\ns=\varprojlim\ns_i$. Then, we apply Theorem \ref{thm:cfr-are-factor-of-nilmanifolds} to $\ns_1$ to create $G_1/\Lambda_1$ and the fibration $G_1/\Lambda_1\to \ns_1$. Then, we consider the fiber product $G_1/\Lambda_1\times_{\ns_1}\ns_2$ (see e.g. \cite[Lemma 2.18]{CGSS-doucos}), which is also a $k$-step \textsc{cfr} nilspace. We then apply Theorem \ref{thm:cfr-are-factor-of-nilmanifolds} to $G_1/\Lambda_1\times_{\ns_1}\ns_2$ and we obtain $G_2/\Lambda_2$. Continuing with this process it is clear that we obtain an inverse limit of coset nilspaces $\nss:=\varprojlim G_i/\Lambda_i$ and a fibration $\nss\to\ns$.\end{proof}

\begin{lemma}\label{lem:equality-orbits}
Let $A$ be a discrete group, let $X,Y$ be compact metric $A$-systems, and $\varphi:Y\to X$ be a continuous equivariant function. Then for any $y\in Y$ we have that $\varphi(\overline{\orb(y)})=\overline{\orb(\varphi(y))}$.
\end{lemma}

\begin{proof}
As $\varphi$ is continuous, the image of the compact set $\overline{\orb(y)}$ is compact, hence closed. As clearly $\gamma(\varphi(y))\in \varphi(\overline{\orb(y)})$ for all $\gamma\in A$ we have that $\varphi(\overline{\orb(y)})\supset\overline{\orb(\varphi(y))}$. For the converse inclusion, let $y'\in \overline{\orb(y)}$ and let $\gamma_n\in A$ be a sequence such that $y'=\lim \gamma_n(y)$. Then, by continuity of $\varphi$ and $A$ equivariance we have that $\varphi(y')=\lim \gamma_n(\varphi(y))\in \overline{\orb(\varphi(y))}$.
\end{proof}

\begin{proof}[Proof of Theorem \ref{thm:main-dynam}]
The proof is essentially the same as the one of \cite[Theorem 4.5]{CGSS-abramov} but we included it here for the convenience of the reader. By \cite[Theorem 7.14]{GGY}, a minimal $\mb{Z}^\omega$-system of order $k$ is isomorphic to a minimal nilspace system $(\ns,\mb{Z}^\omega)$ of order $k$. Hence,  any point $x_0\in \ns$ is transitive and thus $\overline{\orb(x_0)}=\overline{\{tx_0\in\ns: t\in \mb{Z}^\omega\}}=\ns$.

Let $\phi^*: \ns \to  \hom(\mc{D}_1(\mb{Z}^\omega),\ns)$ be the map $x \mapsto \phi^*(x)$ where $\phi^*(x)(t):=tx$ for any $t\in \mb{Z}^\omega$. This map is continuous, injective, and equivariant (see the proof of \cite[Theorem 4.5]{CGSS-abramov}). As $\ns$ and $\hom(\mc{D}_1(\mb{Z}^\omega),\ns)$ are compact metric spaces, we have that $\phi^*$ is a homeomorphism $\ns\to \phi^*(\ns)$ by \cite[Theorem 26.6]{Mu}. This together with the equivariance of $\phi^*$ implies that  $\ns$ and $\phi^*(\ns)$ are isomorphic as topological dynamical systems. We shall now prove that, as a \emph{topological} dynamical system, the system $\phi^*(\ns)$ is a factor of a $\mb{Z}^\omega$-nilspace system.

By minimality, fixing any  $x_0\in \ns$, we have that $\phi^*(\ns)$ is the closure of the orbit of $\phi^*(x_0)$. By Corollary \ref{cor:cov-by-inv-lim-coset}, there exists a nilspace $\nss=\varprojlim G_i/\Lambda_i$ and a fibration $\varphi:\nss\to \ns$. This induces the map $\varphi^*:\hom(\mc{D}_1(\mb{Z}^\omega),\nss) \to \hom(\mc{D}_1(\mb{Z}^\omega),\ns)$ defined as $\varphi^*(f) = \varphi\co f$. Thus, we have the following diagram:
\begin{center}
\begin{tikzpicture}
  \matrix (m) [matrix of math nodes,row sep=2em,column sep=4em,minimum width=2em]
  {
      & \hom(\mc{D}_1(\mb{Z}^\omega),\nss)  \\
     \ns & \hom(\mc{D}_1(\mb{Z}^\omega),\ns). \\};
  \path[-stealth]
    (m-1-2) edge node [right] {$\varphi^*$} (m-2-2)
    (m-2-1) edge node [above] {$\phi^*$} (m-2-2);
\end{tikzpicture}
\end{center}
Note that $\varphi^*$ is continuous and equivariant. Recall that $\phi^*(x_0)\in \hom(\mc{D}_1(\mb{Z}^\omega),\ns)$ and thus, by \cite[Proposition 4.3]{CGSS-abramov}, there exists $g\in \hom(\mc{D}_1(\mb{Z}^\omega),\nss)$ such that $\varphi^*(g)=\phi^*(x_0)$. Moreover, by Lemma \ref{lem:equality-orbits} we have that $\varphi^*(\overline{\orb(g)})=\overline{\orb(\phi^*(x_0))}\cong \ns$. As both $\phi^*$ and $\varphi^*$ are equivariant, we have proved that our original system $\ns$ is a factor (as a topological dynamical system) of $\overline{\orb(g)}$ for some $g\in \hom(\mc{D}_1(\mb{Z}^\omega),\nss)$.

To conclude the proof, it suffices to show that $\hom(\mc{D}_1(\mb{Z}^\omega),\nss)$ is isomorphic (as a topological dynamical system) to $\varprojlim \hom(\mc{D}_1(\mb{Z}^\omega),G_i/\Lambda_i)$. Clearly we can define the continuous function $\iota:\hom(\mc{D}_1(\mb{Z}^\omega),\nss)\to \varprojlim\hom(\mc{D}_1(\mb{Z}^\omega),G_i/\Lambda_i)$ given by $f\mapsto (\varphi_i\co f)_{i\in \mb{N}}$ where $\varphi_i:\nss\to G_i/\Lambda_i$ are the projections of the inverse limit $\nss=\varprojlim G_i/\Lambda_i$. Moreover, this map is a bijection, because for any $(f_i)_{i=1}^\infty\in \varprojlim \hom(\mc{D}_1(\mb{Z}^\omega),G_i/\Lambda_i)$ there is a unique $f\in \hom(\mc{D}_1(\mb{Z}^\omega),\nss)$ such that $\iota(f)=(f_i)_{i=1}^\infty$ (note that the only possible definition is to set $f(t)=(f_i(t))_{i=1}^\infty$ for every $t\in \mb{Z}^\omega$). Hence $\iota$ is a continuous bijection between compact spaces, and therefore it is a homeomorphism. In particular $\iota$ induces a homeomorphism between $\overline{\orb(g)}$ and $\varprojlim \overline{\orb(\varphi_i\co g)}$. This, combined with Proposition \ref{prop:properties-z-omega-sys}, finishes the proof. \end{proof}

\medskip

\noindent \textbf{Funding.} All authors used funding from project PID2024-156180NB-I00 funded by the MICIU/AEI and the European Union. The second-named author was funded by HORIZON-MSCA-2024-PF-01, AlgHOF 101202161 ``Funded by the European Union. Views and opinions expressed are those of the author(s) only and do not reflect those of the European Union or the European Commission. Neither the European Union nor the European Commission can be held responsible for them''. The second and third-named authors were also supported partially by the Hungarian Ministry of Innovation and Technology NRDI Office within the framework of the Artificial Intelligence National Laboratory Program (MILAB, RRF-2.3.1-21-2022-00004).

\medskip

\noindent \textbf{Disclosure on LLM use.} ChatGPT (GPT-6 Astra) was used to review and optimize the corrections made to obtain this version of the paper. This affected none of the main results nor the overall exposition of the paper.


\begin{thebibliography}{1}

\bibitem{BTZ} V. Bergelson, T. Tao and T. Ziegler, \emph{An inverse theorem for the uniformity seminorms associated with the action of $\mb{F}_p^{\infty}$}, Geom. Funct. Anal. \textbf{19} (2010), no. 6, 1539--1596.

\bibitem{CamSzeg} O. Antol\'in Camarena, B. Szegedy, \emph{Nilspaces, nilmanifolds and their morphisms}, preprint (2010), \url{http://arxiv.org/abs/1009.3825}

\bibitem{Cand:Notes1} P. Candela, \emph{Notes on nilspaces: algebraic aspects}, Discrete Anal., 2017, Paper No. 15, 59 pp.

\bibitem{Cand:Notes2} P. Candela, \emph{Notes on compact nilspaces}, Discrete Anal., 2017, Paper No. 16, 57pp.

\bibitem{CGSS} P. Candela, D. Gonz\'alez-S\'anchez, B. Szegedy, \emph{On nilspace systems and their morphisms}, Ergodic Theory Dynam. Systems \textbf{40} (2020), no. 11, 3015--3029.

\bibitem{CGSS-p-hom} P. Candela, D. Gonz\'alez-S\'anchez, B. Szegedy, \emph{On higher-order Fourier analysis in characteristic $p$}, Ergodic Theory Dynam.\ Systems \textbf{43} (2023), no.\ 12, 3971-4040.

\bibitem{CGSS-abramov} P. Candela, D. Gonz\'alez-S\'anchez, B. Szegedy, \emph{On measure-preserving $\mb{F}_p^\omega$-systems of order $k$}, J. Anal. Math. {\bf 156}, (2025), pp. 281--298.

\bibitem{CGSS-doucos} P. Candela, D. Gonz\'alez-S\'anchez, B. Szegedy, \emph{Free nilspaces, double coset nilspaces, and Gowers norms}, Math. Ann. 395, 23 (2026).

\bibitem{CGSS-bndtor} P. Candela, D. Gonz\'alez-S\'anchez, B. Szegedy, \emph{On the inverse theorem in abelian groups of bounded torsion}, to appear in Discrete Anal. (2026), \url{https://arxiv.org/abs/2311.13899}.

\bibitem{CGSS-spec} P. Candela, D. Gonz\'alez-S\'anchez, B. Szegedy,  \emph{Spectral algorithms in higher-order Fourier analysis}, Forum Math. Sigma (2026), Vol. 14:e95 1–71.

\bibitem{CandSis} P. Candela, O. Sisask, \emph{Convergence results for systems of linear forms on cyclic groups and periodic nilsequences}, SIAM J. Discrete Math. {\bf 28} (2014), no. 2, 786--810.

\bibitem{CScouplings} P. Candela, B. Szegedy, \emph{Nilspace factors for general uniformity seminorms, cubic exchangeability and limits}, Mem. Amer. Math. Soc. \textbf{287} (2023), no. 1425, v+101 pp.

\bibitem{CSinverse} P. Candela, B. Szegedy, \emph{Regularity and inverse theorems for uniformity norms on compact abelian groups and nilmanifolds}, J.\ Reine Angew.\ Math.\ \textbf{789} (2022), 1--42. 

\bibitem{Clement&al} A. Clement, S. Majewicz, M. Zyman, \emph{The theory of nilpotent groups}, Birkh\"auser$/$Springer, Cham, 2017. xvii$+$307 pp.

\bibitem{C&G} L. J. Corwin, F. P. Greenleaf,  \emph{Representations of nilpotent Lie groups and their applications. Part I. Basic theory and examples}, Cambridge Stud. Adv. Math., 18 Cambridge University Press, Cambridge, 1990. viii$+$269 pp.

\bibitem{GTarit} B. Green, T. Tao, \emph{An arithmetic regularity lemma, an associated counting lemma, and applications}. An irregular mind, 261--334,
Bolyai Soc. Math. Stud., 21, J\'anos Bolyai Math. Soc., Budapest, 2010. 

\bibitem{GT08}  B. Green, T. Tao, \emph{An inverse theorem for the Gowers $U^3$-norm}, Proc. Edinburgh Math. Soc. (1) \textbf{51} (2008), 73--153.

\bibitem{GTorb} B. Green, T. Tao, \emph{The quantitative behaviour of polynomial orbits on nilmanifolds}, Ann. of Math. \textbf{175} (2012), 465--540.

\bibitem{GTZ-U4} B. Green, T. Tao, T. Ziegler, \emph{An inverse theorem for the Gowers $U^4$-norm}, Glasgow Math. J. \textbf{53} (2011), 1--50. 

\bibitem{GTZ} B. Green, T. Tao, T. Ziegler, \emph{An inverse theorem for the Gowers $U^{s+1}[N]$-norm}, Ann. of Math. (2) \textbf{176} (2012), no. 2, 1231--1372.

\bibitem{GGY} E. Glasner, Y. Gutman, X. Ye, \emph{Higher order regionally proximal equivalence relations for general minimal group actions}, Advances in Mathematics \textbf{333} (2018), 1004--1041.

\bibitem{GMV1} Y. Gutman, F. Manners, P. P. Varj\'u, \emph{The structure theory of nilspaces I},  J. Anal. Math. \textbf{140} (2020), no. 1, 299--369.

\bibitem{GMV2} Y. Gutman, F. Manners, P. P. Varj\'u, \emph{The structure theory of nilspaces II: Representation as nilmanifolds},  Trans. Amer. Math. Soc. \textbf{371} (2019), no. 7, 4951--4992. 

\bibitem{GMV3} Y. Gutman, F. Manners, P. P. Varj\'u, \emph{The structure theory of nilspaces III: Inverse limit representations and topological dynamics}, Adv. Math. \textbf{365} (2020), 107059, 53 pp.

\bibitem{Hall} B. C. Hall, \emph{Lie Groups, Lie Algebras, and Representations: An Elementary Introduction}, Graduate Texts in Mathematics, vol. 222 (2nd ed.), Springer. 

\bibitem{H&M-Cpct} K. H. Hofmann,  S. A. Morris, \emph{The structure of compact groups. A primer for the student -- a handbook for the expert}. Third edition. De Gruyter Stud. Math., 25. De Gruyter, Berlin, 2013.

\bibitem{HK-non-conv} B. Host, B. Kra, \emph{Nonconventional ergodic averages and nilmanifolds}, Ann. of Math. (2) \textbf{161} (2005), no. 1, 397--488.

\bibitem{HK-par} B. Host, B. Kra, \emph{Parallelepipeds, nilpotent groups, and Gowers norms}, Bull. Soc. Math. France \textbf{136} (2008), 405--437.

\bibitem{HKbook} B. Host, B. Kra, \emph{Nilpotent structures in ergodic theory}, Math. Surveys Monogr., 236. American Mathematical Society, Providence, RI, 2018. 

\bibitem{HKM} B. Host, B. Kra, A. Maass, \emph{Nilsequences and a structure theorem for topological dynamical systems}, Adv. Math. {\bf 224} (1) (2010) 103--129.

\bibitem{JST1} A. Jamneshan, O. Shalom, T. Tao, \emph{The structure of totally disconnected Host--Kra--Ziegler factors, and the inverse theorem for the $U^k$  Gowers uniformity norms on finite abelian groups of bounded torsion}, preprint. \url{https://arxiv.org/abs/2303.04860}

\bibitem{JST2} A. Jamneshan, O. Shalom, T. Tao, \emph{A Host--Kra $\mb{F}_2^\omega$-system of order 5 that is not Abramov of order 5, and non-measurability of the inverse theorem for the $U^6(\mb{F}_2^n)$ norm}, Math. Ann. 394, 11 (2026).

\bibitem{JST3} A. Jamneshan, O. Shalom, T. Tao, \emph{The structure of arbitrary Conze--Lesigne systems}, Comm. Amer. Math. Soc. {\bf 4} (2024), pp. 182--229.

\bibitem{J&T} A. Jamneshan, T. Tao, \emph{The inverse theorem for the $U^3$ Gowers uniformity norm on arbitrary finite abelian groups: Fourier-analytic and ergodic approaches}, Discrete Anal., 2023:11, 48 pp.

\bibitem{Khukhro}  E. I. Khukhro, \emph{$p$-automorphisms of finite $p$-groups}, London Math. Soc. Lecture Note Ser., 246
Cambridge University Press, Cambridge, 1998. xviii$+$204 pp.

\bibitem{LeibRat} A. Leibman, \emph{Rational sub-nilmanifolds of a compact nilmanifold}, Ergodic Theory Dynam. Systems (3) \textbf{26} (2006), 787--798.

\bibitem{MKS76} W. Magnus, A. Karrass, and D. Solitar, \textit{Combinatorial Group Theory: Presentations of Groups in Terms of Generators and Relations}, Reprint of the 1976 second edition, Dover Publications, Inc., Mineola, NY, 2004. xii+444 pp.

\bibitem{Mal'cev} A. I. Mal'cev, \emph{On the theory of the Lie groups in the large}, Sb. Math., {\bf 58:2} (1945). 

\bibitem{Mal'cev:nil} A. I. Mal'cev, \emph{On a class of homogeneous spaces}, Amer. Math. Soc. Translation 1951 (1951), no. 39, 33 pp.

\bibitem{Mu} J. R. Munkres, \emph{Topology, Second Edition}, Prentice Hall, Inc., Upper Saddle River, NJ, 2000.

\bibitem{Petresco} J. Petresco, \emph{Sur les commutateurs}, Math Z {\bf 61}, (1954), pp. 348--356.

\bibitem{Rag} M. S. Raghunathan, \emph{Discrete subgroups of Lie groups}, Ergebnisse der Mathematik und ihrer Grenzgebiete [Results in Mathematics and Related Areas], Band 68. Springer-Verlag, New York-Heidelberg, 1972. ix$+$227 pp. 

\bibitem{Shalom} O. Shalom, \emph{Multiple ergodic averages in abelian groups and Khintchine type recurrence}, Trans. Amer. Math. Soc. 375 (2022), 2729--2761.

\bibitem{Shalom2} O. Shalom, \emph{Host-Kra factors for $\bigoplus_{p\in \mc{P}} \mb{Z}/p\mb{Z}$-actions and finite dimensional nilpotent systems}, Anal. PDE, {\bf 17} (7) (2024), pp. 2379--2449.

\bibitem{Song-Ye} S. Shao, X. Ye, \emph{Regionally proximal relation of order d is an equivalence one for minimal systems and a combinatorial consequence},
Adv. Math. \textbf{231} (2012), no. 3-4, 1786--1817. 

\bibitem{SzegHigh} B. Szegedy, \emph{On higher-order Fourier analysis}, preprint. \url{https://arxiv.org/abs/1203.2260}

\bibitem{SzegFin} B. Szegedy, \emph{Structure of finite nilspaces and inverse theorems for the Gowers norms in bounded exponent groups}, preprint. \url{https://arxiv.org/abs/1011.1057}

\bibitem{T&Z-Low} T. Tao, T. Ziegler, \emph{The inverse conjecture for the Gowers norm over finite fields in low characteristic}, Ann. Comb. \textbf{16} (2012), 121--188.

\bibitem{Wil} B. Wilking, \emph{Rigidity of group actions on solvable Lie groups}, Math. Ann. 317 (2000), no. 2, 195--237.


\end{thebibliography}
\end{document}